\documentclass[nohyperref]{article}
\usepackage{fullpage}
\usepackage{times}
\usepackage{natbib}

\newif\ificml

\icmltrue 
\icmlfalse 

\usepackage{algorithm,algcompatible}
\usepackage{microtype}
\usepackage{graphicx}
\usepackage{subfigure}
\usepackage[makeroom]{cancel}
\usepackage{booktabs} 
\usepackage{makecell}
\usepackage[colorlinks=true,citecolor=blue]{hyperref}%
\usepackage{multirow}
\usepackage[table, svgnames, dvipsnames]{xcolor}
\usepackage{makecell, cellspace, caption}
\usepackage{booktabs,caption}
\usepackage[flushleft]{threeparttable}

\newcommand{\nosemic}{\renewcommand{\@endalgocfline}{\relax}}
\newcommand{\dosemic}{\renewcommand{\@endalgocfline}{\algocf@endline}}
\let\oldnl\nl
\newcommand{\nonl}{\renewcommand{\nl}{\let\nl\oldnl}}

\usepackage{amsmath, nccmath}
\usepackage{amssymb}
\usepackage{mathtools}
\usepackage{enumitem}
\usepackage{amsthm}
\usepackage{lipsum}

\usepackage[capitalize,noabbrev]{cleveref}
\allowdisplaybreaks

\theoremstyle{plain}
\newtheorem{theorem}{Theorem}
\newtheorem*{theorem*}{Theorem}

\newtheorem{lemma}{Lemma}[section]
\newtheorem{cor}{Corollary}
\theoremstyle{definition}
\newtheorem{definition}{Definition}
\newtheorem{assump}{Assumption}
\theoremstyle{remark}
\newtheorem{remark}{Remark}

\newcommand{\mc}{\mathcal}
\newcommand{\mco}{\mathcal O}
\newcommand{\mbb}{\mathbb}
\newcommand{\mbf}{\mathbf}
\newcommand{\mbe}{\mathbb E}

\newcommand{\bx}{{\mathbf x}}
\newcommand{\bxs}{{\bx^{\star}}}
\newcommand{\bxts}{{\bx_t^{\star}}}
\newcommand{\bxtaus}{{\bx_{\sync}^{\star}}}
\newcommand{\bxt}{{\mathbf x_t}}
\newcommand{\bxtp}{{\mathbf x_{t+1}}}
\newcommand{\bxit}{{\mathbf x^i_t}}
\newcommand{\bxitp}{{\mathbf x^i_{t+1}}}

\newcommand{\bxjtp}{{\mathbf x^j_{t+1}}}

\newcommand{\bxitk}{{\mathbf{x}^i_{t,k}}}
\newcommand{\bxitkp}{{\mathbf{x}^i_{t,k+1}}}

\newcommand{\Tbx}{\widetilde{\bx}}
\newcommand{\Tbxtp}{\Tbx_{t+\frac{1}{2}}}

\newcommand{\Tbxk}{\Tbx_{k}}
\newcommand{\Tbxitp}{\Tbx^i_{t+\frac{1}{2}}}

\newcommand{\Tby}{\Tilde{\by}}
\newcommand{\Tbytp}{\Tby_{t+\frac{1}{2}}}

\newcommand{\Tbyitp}{\Tby^i_{t+\frac{1}{2}}}

\newcommand{\by}{{\mathbf y}}
\newcommand{\bys}{{\by^{\star}}}
\newcommand{\byts}{{\by_t^{\star}}}
\newcommand{\byt}{{\mathbf y_t}}
\newcommand{\bytp}{{\mathbf y_{t+1}}}
\newcommand{\byit}{{\mathbf y^i_t}}
\newcommand{\byitp}{{\mathbf y^i_{t+1}}}

\newcommand{\byjtp}{{\mathbf y^j_{t+1}}}

\newcommand{\byitk}{{\mathbf{y}^i_{t,k}}}

\newcommand{\byitkp}{{\mathbf{y}^i_{t,k+1}}}

\newcommand{\bbxt}{{\bar{\bx}_t}}
\newcommand{\bbxT}{{\bar{\bx}_T}}
\newcommand{\bbxtp}{{\bar{\bx}_{t+1}}}

\newcommand{\Fgp}{{F_{\gamma, \psi}}}
\newcommand{\Fgpt}{{F_{\gamma, \pt}}}
\newcommand{\Fg}{{F_{\gamma}}}

\newcommand{\pt}{{\psi_t}}

\newcommand{\IEt}{\mc E_{t}}
\newcommand{\IEtp}{\mc E_{t+1}}
\newcommand{\FEt}{\mathfrak{E}_{t}}
\newcommand{\FE}{\mathfrak{E}}
\newcommand{\FEtp}{\mathfrak{E}_{t+1}}

\newcommand{\bdx}{{\mathbf d_x}}
\newcommand{\bdxt}{{\mathbf d_{x,t}}}
\newcommand{\bdxit}{{\mathbf d^i_{x,t}}}
\newcommand{\bdxitp}{{\mathbf d^i_{x,t+1}}}

\newcommand{\bdxtp}{{\mathbf d_{x,t+1}}}

\newcommand{\bdy}{{\mathbf d_y}}
\newcommand{\bdyt}{{\mathbf d_{y,t}}}
\newcommand{\bdyit}{{\mathbf d^i_{y,t}}}
\newcommand{\bdyitp}{{\mathbf d^i_{y,t+1}}}

\newcommand{\bdytp}{{\mathbf d_{y,t+1}}}

\newcommand{\xiit}{{\xi^i_{t}}}
\newcommand{\xiitp}{{\xi^i_{t+1}}}

\newcommand{\xijtp}{{\xi^j_{t+1}}}

\newcommand{\Lp}{L_{\Phi}}
\newcommand{\Lf}{L_f}
\newcommand{\Lft}{\widetilde{L}_f}
\newcommand{\lrk}{\eta_k}

\newcommand{\lr}{\eta}
\newcommand{\lrx}{\eta_x}
\newcommand{\lry}{\eta_y}

\newcommand{\Dphi}{\Delta_{\Phi}}
\newcommand{\sync}{\tau}

\newcommand{\heterox}{\varsigma_x}
\newcommand{\heteroy}{\varsigma_y}

\newcommand{\mom}{\beta}
\newcommand{\momx}{\beta_x}
\newcommand{\momy}{\beta_y}
\newcommand{\cvx}{\alpha}
\newcommand{\cvxt}{\alpha_t}

\newcommand{\sumin}{\sum_{i=1}^n}
\newcommand{\sumjn}{\sum_{j=1}^n}

\newcommand{\sumtT}{\sum_{t=0}^{T-1}}

\newcommand{\sumktm}{\sum_{k=s\sync}^{t-1}}

\newcommand{\avgin}{\frac{1}{n} \sumin}

\newcommand{\avgtT}{\frac{1}{T} \sumtT}

\newcommand{\G}{\nabla}
\newcommand{\Gx}{\nabla_{\bx}}
\newcommand{\Gy}{\nabla_{\by}}

\newcommand{\CExyt}{\Delta_{t}^{\bx,\by}}
\newcommand{\CExt}{\Delta_{t}^{\bx}}
\newcommand{\CEyt}{\Delta_{t}^{\by}}
\newcommand{\CEpt}{\Delta_{t}^{\bdx}}
\newcommand{\CEqt}{\Delta_{t}^{\bdy}}

\newcommand{\CExk}{\Delta_{k}^{\bx}}
\newcommand{\CEyk}{\Delta_{k}^{\by}}

\newcommand{\CExytp}{\Delta_{t+1}^{\bx,\by}}
\newcommand{\CEptp}{\Delta_{t+1}^{\bdx}}
\newcommand{\CEqtp}{\Delta_{t+1}^{\bdy}}
\newcommand{\CEptm}{\Delta_{t-1}^{\bdx}}
\newcommand{\CEqtm}{\Delta_{t-1}^{\bdy}}

\newcommand{\lp}{\left(}
\newcommand{\rp}{\right)}

\newcommand{\lcb}{\left\{}
\newcommand{\rcb}{\right\}}
\newcommand{\lb}{\left[}
\newcommand{\rb}{\right]}
\newcommand{\lnr}{\left\|}
\newcommand{\rnr}{\right\|}
\newcommand{\lan}{\left\langle}
\newcommand{\ran}{\right\rangle}

\newcommand\norm[1]{\lnr#1\rnr}

\newcommand{\nn}{\nonumber}

\DeclareMathOperator*{\argmin}{arg\,min}
\DeclareMathOperator*{\argmax}{arg\,max}

\usepackage{tikz}
\usetikzlibrary{calc,trees,positioning,arrows,chains,shapes.geometric,%
	decorations.pathreplacing,decorations.pathmorphing,shapes,%
	matrix,shapes.symbols}

\usepackage{tikz}
\def\checkmark{\tikz\fill[scale=0.4](0,.35) -- (.25,0) -- (1,.7) -- (.25,.15) -- cycle;}

\tikzstyle{startstop} = [rectangle, draw, rounded corners, align=center, minimum width=3cm, minimum height=1cm,text centered]
\tikzstyle{decision} = [diamond, draw, fill=blue!20, 
text width=4.5em, text badly centered, client distance=3cm, inner sep=0pt]
\tikzstyle{block} = [rectangle, draw, fill=blue!10, align=center, rounded corners, minimum width=3cm, minimum height=1cm]
\tikzstyle{blockcast} = [rectangle, draw, fill=red!10, align=center, rounded corners, minimum width=3cm, minimum height=0.45cm]
\tikzstyle{line} = [draw, -latex']
\tikzstyle{cloud} = [draw, ellipse,fill=red!20, client distance=3cm,
minimum height=2em]

\usepackage[textsize=tiny]{todonotes}

\title{
	Federated Minimax Optimization:
	\\
	Improved Convergence Analyses and Algorithms
}

\author{Pranay Sharma$^{*}$,
	Rohan Panda$^{*}$,
	Gauri Joshi$^{*}$ and Pramod K. Varshney$^{\dagger}$\\
	\\
	$^{*}$Department of Electrical and Computer Engineering,\\
	Carnegie Mellon University, Pittsburgh, PA\\
	$^{\dagger}$Department of Electrical Engineering and Computer Science, \\
	Syracuse University, Syracuse, NY\\
	\{pranaysh, rohanpan, gaurij\}@andrew.cmu.edu, varshney@syr.edu}

\begin{document}
	
	\maketitle

	
	
	
	
	\vskip 0.3in
	
	
	
	
	
	\begin{abstract}
		    In this paper, we consider nonconvex minimax optimization, which is gaining prominence in many modern machine learning applications such as GANs. Large-scale edge-based collection of training data in these applications calls for communication-efficient distributed optimization algorithms, such as those used in federated learning, to process the data. In this paper, we analyze Local stochastic gradient descent ascent (SGDA), the local-update version of the SGDA algorithm. SGDA is the core algorithm used in minimax optimization, but it is not well-understood in a distributed setting. We prove that Local SGDA has \textit{order-optimal} sample complexity for several classes of nonconvex-concave and nonconvex-nonconcave minimax problems, and also enjoys \textit{linear speedup} with respect to the number of clients. We provide a novel and tighter analysis, which improves the convergence and communication guarantees in the existing literature. For nonconvex-PL and nonconvex-one-point-concave functions, we improve the existing complexity results for centralized minimax problems. Furthermore, we propose a momentum-based local-update algorithm, which has the same convergence guarantees, but outperforms Local SGDA as demonstrated in our experiments.
	\end{abstract}

	\section{Introduction}
\label{sec:intro}

In the recent years, minimax optimization theory has found relevance in several modern machine learning applications including 
Generative Adversarial Networks (GANs) \cite{goodfellow14GANs_neurips, arjovsky17WGANs_icml, gulrajani17improved_WGANs_neurips},
adversarial training of neural networks 
\cite{sinha17certifiable_robust_iclr, madry18adversarial_iclr, wang21adversarial_minmax_neurips},
reinforcement learning \cite{dai17learning_aistats, dai18sbeed_RL_nonlin_FA_icml},
and robust optimization \cite{namkoong16SG_DRO_neurips, namkoong17var_regular_neurips, mohri19agnosticFL_icml}. 
Many of these problems lie outside the domain of classical convex-concave theory \cite{daskalakis21constr_minmax_sigact, hsieh21limits_minmax_icml}.

\ificml
\else
\begin{table*}[t]
	\begin{center}
		\begin{threeparttable}
			\caption{Comparison of different local-updates-based algorithms proposed to solve \eqref{eq:problem}, in terms of the number of stochastic gradient computations (per client) and the number of communication rounds needed to reach an $\epsilon$-stationary solution (see \cref{defn:stationarity}) of \eqref{eq:problem}. 
				Here, $\kappa = \Lf/\mu$ is the condition number (see Assumptions \ref{assum:smoothness}, \ref{assum:PL_y}).
			}
			\label{table:comparison}
			\vskip 0.15in
			\begin{small}
				\begin{tabular}{|c|c|c|c|c|}
					\hline
					Function Class & Work & \makecell{Number of Communication \\ Rounds} & \makecell{Stochastic Gradient \\ Complexity} \\
					\hline
					\multirow{4}{*}{\begin{tabular}[c]{@{}c@{}} \underline{N}on\underline{C}onvex- \\ \underline{S}trongly-\underline{C}oncave \\
							(NC-SC) \end{tabular}} &
					Baseline ($n=1$) \cite{lin_GDA_icml20}
					& - & $\mco \lp \frac{\kappa^3}{\epsilon^{4}} \rp$ \\
					& \cite{mahdavi21localSGDA_aistats} & $\mco \lp \frac{\kappa^8}{n^{1/3} \epsilon^4} \rp$ & $\mco \lp \frac{\kappa^{12}}{n \epsilon^6} \rp$ \\
					& \cellcolor{Gainsboro!60}\textbf{This Work} (Theorems \ref{thm:NC_PL}, \ref{thm:NC_PL_mom}) & \cellcolor{Gainsboro!60} {\color{red}$\mco \lp \frac{\kappa^3}{\epsilon^{3}} \rp$} & \cellcolor{Gainsboro!60} {\color{red}$\mco \lp \frac{\kappa^4}{n \epsilon^{4}} \rp$} \\
					\hline
					\multirow{4}{*}{\begin{tabular}[c]{@{}c@{}} \underline{N}on\underline{C}onvex-\underline{PL} \\ (NC-PL) \end{tabular}} &
					\cellcolor{Gainsboro!60} \begin{tabular}[c]{@{}c@{}}
						Baseline ($n=1$) \\ \textbf{This Work} (Theorems \ref{thm:NC_PL}, \ref{thm:NC_PL_mom}), \cite{yang21NCPL_arxiv}\tnote{a}
					\end{tabular}
					& \cellcolor{Gainsboro!60}- & \cellcolor{Gainsboro!60} {\color{red}$\mco \lp \frac{\kappa^4}{\epsilon^{4}} \rp$} \\
					& \cite{mahdavi21localSGDA_aistats}\tnote{b} & $\mco \lp \max \lcb \frac{\kappa^2}{\epsilon^4}, \frac{\kappa^4}{n^{2/3} \epsilon^4} \rcb \rp$
					& $\mco \lp \max \lcb \frac{\kappa^3}{n \epsilon^6}, \frac{\kappa^6}{n^2 \epsilon^6} \rcb \rp$ \\
					& \cellcolor{Gainsboro!60}\textbf{This Work} (Theorems \ref{thm:NC_PL}, \ref{thm:NC_PL_mom}) & \cellcolor{Gainsboro!60}{\color{red}$\mco \lp \frac{\kappa^3}{\epsilon^{3}} \rp$} & \cellcolor{Gainsboro!60}{\color{red}$\mco \lp \frac{\kappa^4}{n \epsilon^{4}} \rp$} \\
					\hline
					\multirow{3}{*}{\begin{tabular}[c]{@{}c@{}} \underline{N}on\underline{C}onvex- \\ \underline{C}oncave (NC-C) \end{tabular}} &
					Baseline ($n=1$) \cite{lin_GDA_icml20}
					& - & $\mco ( \epsilon^{-8} )$ \\
					& \cite{mahdavi20dist_robustfl_neurips}\tnote{c} 
					& $\mco (\epsilon^{-12})$ & $\mco ( \epsilon^{-16} )$ \\
					& \cellcolor{Gainsboro!60}\textbf{This Work} (Theorem \ref{thm:NC_C})
					& \cellcolor{Gainsboro!60}{\color{red}
						$\mco (\epsilon^{-7})$}
					& \cellcolor{Gainsboro!60}{\color{red}
						$\mco \lp \frac{1}{n \epsilon^8} \rp$} \\
					\hline
					\multirow{6}{*}{\begin{tabular}[c]{@{}c@{}} \underline{N}on\underline{C}onvex- \\ \underline{1}-\underline{P}oint-\underline{C}oncave \\
							(NC-1PC) \end{tabular}} &
					\cellcolor{Gainsboro!60} 
					Baseline ($n=1$) \textbf{This Work} (\cref{thm:NC_1PC})
					& \cellcolor{Gainsboro!60} - 
					& \cellcolor{Gainsboro!60}{\color{red}$\mco ( \epsilon^{-8} )$} \\
					& \cite{mahdavi21localSGDA_aistats} & $\mco \lp \frac{n^{1/6}}{\epsilon^{8}} \rp$ & $\mco ( \epsilon^{-12} )$ \\
					& \cite{liu20dec_GANs_neurips} & $\widetilde{\mco} ( \epsilon^{-12} )$\tnote{d} & $\mco (\epsilon^{-12})$ \\
					& \cellcolor{Gainsboro!60}
					\textbf{This Work} (\cref{thm:NC_1PC})
					& \cellcolor{Gainsboro!60}{\color{red}$\mco ( \epsilon^{-7} )$} & \cellcolor{Gainsboro!60}{\color{red}$\mco ( \epsilon^{-8} )$} \\
					& \cellcolor{Gainsboro!60} \textbf{This Work} ($\sync = 1$) (\cref{app:NC_1PC_tau_1})\tnote{e}
					& \cellcolor{Gainsboro!60}{\color{red}$\mco \lp \frac{1}{n \epsilon^{8}} \rp$} & \cellcolor{Gainsboro!60}{\color{red}$\mco \lp \frac{1}{n \epsilon^{7}} \rp$} \\
					\hline
				\end{tabular}
				\begin{tablenotes}
					\small
					\item[a] We came across this work during the preparation of this manuscript.
					\item[b] Needs the additional assumption of $G_x$-Lipschitz continuity of $f(x,y)$ in $x$.
					\item[c] The loss function is nonconvex in $\bx$ and linear in $\by$.
					\item[d] Decentralized algorithm. Requires $\mco (\log(1/\epsilon))$ communication rounds with the neighbors after each update step.
					\item[e] This is fully synchronized Local SGDA.
				\end{tablenotes}
			\end{small}
		\end{threeparttable}
	\end{center}
\end{table*}
\fi

In this work, we consider the following smooth
nonconvex minimax distributed optimization problem:
\ificml
\else
\begin{align}
	\min_{\bx \in \mbb R^{d_1}} \max_{\by \in \mbb R^{d_2}} \Big\{ f(\bx, \by) := \frac{1}{n} \sum_{i=1}^n f_i(\bx, \by) \Big\},
	\label{eq:problem}
\end{align}
\fi
where $n$ is the number of clients, and $f_i$ represents the local loss function at client $i$, defined as $f_i(\bx, \by) = \mbe_{\xi_i \sim \mc D_i} \lb L(\bx, \by; \xi_i) \rb$.
Here, $L(\cdot, \cdot; \xi_i)$ denotes the loss for the data point $\xi_i$, sampled from the local data distribution $\mc D_i$ at client $i$.
The functions $\{ f_i \}$ are smooth, nonconvex in $\bx$, and concave or nonconcave in $\by$.

Stochastic gradient descent ascent (SGDA) \cite{heusel17gans_neurips, daskalakis18GANs_iclr}, a simple generalization of SGD \cite{bottou18optML_siam}, is one of the simplest algorithms used to iteratively solve \eqref{eq:problem}. 
It carries out alternate (stochastic) gradient descent/ascent for the min/max problem.
The exact form of the convergence results depends on the (non)-convexity assumptions which the objective function $f$ in \eqref{eq:problem} satisfies with respect to $\mathbf{x}$ and $\mathbf{y}$. For example,  strongly-convex strongly-concave (in $\bx$ and $\by$, respectively), non-convex-strongly-concave, non-convex-concave, etc.

Most existing literature on minimax optimization problems is focused on solving the problem at a single client.
However, in big data applications that often rely on multiple sources or \textit{clients} for data collection \cite{xing2016strategies}, transferring the entire dataset to a single \textit{server} is often undesirable. Doing so might be costly in applications with high-dimensional data, or altogether prohibitive due to the privacy concerns of the clients \cite{leaute13protecting}.

Federated Learning (FL) is a recent paradigm \cite{konevcny16federated, kairouz19advancesFL_arxiv} proposed to address this problem.
In FL, the edge clients are not required to send their data to the server, improving the privacy afforded to the clients. Instead, the central server offloads some of its computational burden to the clients, which run the training algorithm on their local data. The models trained locally at the clients are periodically communicated to the server, which aggregates them and returns the updated model to the clients.
This infrequent communication with the server leads to communication savings for the clients. 
Local Stochastic Gradient Descent (Local SGD or FedAvg) \cite{fedavg17aistats, stich18localSGD_iclr} is one of the most commonly used algorithms for FL.
Tight convergence rates along with communication savings for Local SGD have been shown for smooth convex \cite{khaled20localSGD_aistats, spiridonoff21comm_eff_SGD_neurips} and nonconvex \cite{koloskova20unified_localSGD_icml} minimization problems. 
See \cref{app:local_SGD} for more details.
Despite the promise shown by FL in large-scale applications \cite{yang18FL_google_arxiv, bonawitz19towardsFL_arxiv}, much of the existing work focuses on solving standard minimization problems of the form $\min_{\mathbf{x}} g(\mathbf{x})$.
The goals of distributed/federated minimax optimization algorithms and their analyses are to show that by using $n$ clients, we can achieve error $\epsilon$, not only in $n$ times fewer total iterations, but also with fewer rounds of communication with the server. This means that more local updates are performed at the clients while the coordination with the central server is less frequent.
Also, this $n$-fold saving in computation at the clients is referred to as \textit{linear speedup} in the FL literature \cite{jiang18linear_neurips, yu19icml_momentum, yang21partial_client_iclr}.
Some recent works have attempted to achieve this goal for convex-concave \cite{mahdavi20dist_robustfl_neurips, hou21FedSP_arxiv, liao21local_AdaGrad_CC_arxiv}, for nonconvex-concave \cite{mahdavi20dist_robustfl_neurips}, and for nonconvex-nonconcave problems \cite{mahdavi21localSGDA_aistats, reisizadeh20robustfl_neurips, guo20DeepAUC_icml, yuan21FedDeepAUC_icml}.

However, in the context of stochastic smooth nonconvex minimax problems, the convergence guarantees of the existing distributed/federated approaches are, to the best of our knowledge, either asymptotic \cite{shen21fedmm_arxiv} or suboptimal \cite{mahdavi21localSGDA_aistats}.
In particular, they do not reduce to the existing baseline results for the centralized minimax problems $(n=1)$. 
See \cref{table:comparison}.

\paragraph{Our Contributions.}
In this paper, we consider the following four classes of minimax optimization problems and refer to them using the abbreviations given below:
\ificml
\else
\begin{enumerate}
	\item NC-SC: \underline{N}on\underline{C}onvex in $\bx$, \underline{S}trongly-\underline{C}oncave in $\by$,
	\item NC-PL: \underline{N}on\underline{C}onvex in $\bx$, \underline{PL}-condition in $\by$ (\cref{assum:PL_y}),
	\item NC-C: \underline{N}on\underline{C}onvex in $\bx$, \underline{C}oncave in $\by$,
	\item NC-1PC: \underline{N}on\underline{C}onvex in $\bx$, \underline{1}-\underline{P}oint-\underline{C}oncave in $\by$ (\cref{assum:1pc_y}).
\end{enumerate}
\fi
For each of these problems, we improve the convergence analysis of existing algorithms or propose a new local-update-based algorithm that gives a better sample complexity. 
A key feature of our results is the linear speedup in the sample complexity with respect to the number of clients, while also providing communication savings. We make the following main contributions, also summarized in \cref{table:comparison}.

\begin{itemize}[leftmargin=*]
	\setlength\itemsep{-0.5em}
	\item For NC-PL functions (\cref{sec:NC_PL}), we prove that Local SGDA
	has {\small$\mco (\kappa^4/(n \epsilon^{4}))$} gradient complexity, and {\small$\mco (\kappa^3/\epsilon^{3})$} communication cost (\cref{thm:NC_PL}).
	The results are optimal in $\epsilon$.\footnote{Even for simple nonconvex function minimization, the complexity guarantee cannot be improved beyond {\small$\mco (1/\epsilon^{4})$} \cite{arjevani19lower_stoch_NC_arxiv}. Further, our results match the complexity and communication guarantees for simple smooth nonconvex minimization with local SGD \cite{yu19icml_momentum}.}
	To the best of our knowledge, this complexity guarantee does not exist in the prior literature even for $n=1$.\footnote{During the preparation of this manuscript, we came across the centralized minimax work \cite{yang21NCPL_arxiv}, which achieves {\small$\mco (\kappa^4/ \epsilon^{4})$} complexity for NC-PL functions. However, our work is more general since we incorporate local updates at the clients.}
	\item Since the PL condition is weaker than strong-concavity, our result also extends to NC-SC functions.
	To the best of our knowledge, ours is the first work to prove optimal (in $\epsilon$) guarantees for SDGA in the case of NC-SC functions, with $\mco (1)$ batch-size. This way, we improve the result in \cite{lin_GDA_icml20} which necessarily requires {\small$\mco (1/\epsilon^2)$} batch-sizes.
	In the federated setting, ours is the first work to achieve the optimal (in $\epsilon$) guarantee.
	\item We propose a novel algorithm (Momentum Local SGDA - \cref{alg_NC_momentum}), which achieves the same theoretical guarantees as Local SGDA for NC-PL functions
	(\cref{thm:NC_PL_mom}), and also outperforms Local SGDA in experiments.
	\item For NC-C functions (\cref{sec:NC_C}), we utilize Local SGDA+ algorithm proposed in \cite{mahdavi21localSGDA_aistats}\footnote{\cite{mahdavi21localSGDA_aistats} does not analyze NC-C functions.}, and prove {\small$\mco (1/(n \epsilon^{8}))$} gradient complexity, and {\small$\mco (1/\epsilon^{7})$} communication cost (\cref{thm:NC_C}).
	This implies linear speedup over the $n=1$ result \cite{lin_GDA_icml20}.
	\item For NC-1PC functions (\cref{sec:NC_1PC}), using an improved analysis for Local SGDA+,
	we prove {\small$\mco (1/\epsilon^{8})$} gradient complexity, and {\small$\mco (1/\epsilon^{7})$} communication cost (\cref{thm:NC_1PC}).
	To the best of our knowledge, this result is the first to generalize the existing {\small$\mco (1/\epsilon^{8})$} complexity guarantee of SGDA (proved for NC-C problems in \cite{lin_GDA_icml20}), to the more general class of NC-1PC functions. 
\end{itemize}

\section{Related Work}
\label{sec:related_work}

\subsection{Single client minimax}

Until recently, the minimax optimization literature was focused largely on convex-concave problems \cite{nemirovski04prox_siam, nedic09subgradient_jota}.
However, since the advent of machine learning applications such as GANs \cite{goodfellow14GANs_neurips}, and adversarial training of neural networks (NNs) \cite{madry18adversarial_iclr}, the more challenging problems of nonconvex-concave and nonconvex-nonconcave minimax optimization have attracted increasing attention.

\paragraph{Nonconvex-Strongly Concave (NC-SC) Problems.}
For stochastic NC-SC problems, \cite{lin_GDA_icml20} proved {\small$\mco (\kappa^3/\epsilon^{4})$} stochastic gradient complexity for SGDA.
However, the analysis necessarily requires mini-batches of size {\small$\Theta (\epsilon^{-2})$}.
Utilizing momentum, \cite{qiu20single_timescale_ncsc} achieved the same {\small$\mco (\epsilon^{-4})$} convergence rate with {\small$\mco (1)$} batch-size.
\cite{qiu20single_timescale_ncsc, luo20SREDA_ncsc_neurips} utilize variance-reduction to further improve the complexity to {\small$\mco (\kappa^3/\epsilon^{3})$}.
However, whether these guarantees can be achieved in the federated setting, with multiple local updates at the clients, is an open question.
In this paper, we answer this question in the affirmative.

\paragraph{Nonconvex-Concave (NC-C) Problems.}
The initial algorithms \cite{nouiehed19minimax_neurips19, thekumparampil19NC_C_neurips, rafique18WCC_oms} for deterministic NC-C problems all have a nested-loop structure. For each $\bx$-update, the inner maximization with respect to $\by$ is approximately solved. Single-loop algorithms have been proposed in subsequent works by \cite{tomluo_1_loop_ncc_neurips20, lan_unified_ncc_arxiv20}.
However, for stochastic problems, to the best of our knowledge, \cite{lin_GDA_icml20} is the only work to have analyzed a single-loop algorithm (SGDA), which achieves {\small$\mco (1/\epsilon^{8})$} complexity.

\paragraph{Nonconvex-Nonconcave (NC-NC) Problems.}
Recent years have seen extensive research on NC-NC problems \cite{mertikopoulos18optMD_SP_iclr, diakonikolas21NC_NC_aistats, daskalakis21constr_minmax_sigact}.
However, of immediate interest to us are two special classes of functions.
\newline
1) Polyak-{\L}ojasiewicz (PL) condition \cite{polyak63PL} is weaker than strong concavity, and does not even require the objective to be concave.
Recently, PL-condition has been shown to hold in overparameterized neural networks \cite{charles18generalization_icml, liu22overparameter_NN_elsevier}.
Deterministic NC-PL problems have been analyzed in \cite{nouiehed19minimax_neurips19, yang20NCNC_VR_neurips, fiez21NC_PL_SC_neurips}.
During the preparation of this manuscript, we came across \cite{yang21NCPL_arxiv} which solves stochastic NC-PL minimax problems. 
Stochastic alternating gradient descent ascent (Stoc-AGDA) is proposed, which achieves {\small$\mco (\kappa^4/\epsilon^{4})$} iteration complexity. 
Further, another single-loop algorithm, \textit{smoothed GDA} is proposed, which improves dependence on $\kappa$ to {\small$\mco (\kappa^2/\epsilon^{4})$}.
\newline
2) One-Point-Concavity/convexity (1PC) has been observed in the dynamics of SGD for optimizing neural networks \cite{li17relu_neurips, kleinberg18icml}.
Deterministic and stochastic optimization guarantees for 1PC functions have been proved in \cite{gasnikov17acc_quasar_convex_arxiv, hinder20near_opt_star_convex_colt, jin20quasar_convex_arxiv}.
NC1PC minimax problems have been considered in \cite{mertikopoulos18optMD_SP_iclr} with asymptotic convergence results, and in  \cite{liu20dec_GANs_neurips}, with $\mco (1/\epsilon^{12})$ gradient complexity.
As we show in \cref{sec:NC_1PC}, this complexity result can be significantly improved.

\subsection{Distributed/Federated Minimax}

Recent years have seen a spur of interest in distributed minimax problems, driven by the need to train neural networks over multiple clients \cite{liu20dec_GANs_neurips, chen20dist_GAN_quantize_arxiv}.
Saddle-point problems and more generally variational inequalities have been studied extensively in the context of decentralized optimization by \cite{beznosikov20dist_SP_arxiv, gasnikov21dec_stoch_EG_VI_arxiv, beznosikov21dist_sp_neurips, rogozin21dec_local_global_var_cc_arxiv, xian21dec_ncsc_storm_neurips}.

Local updates-based algorithms for convex-concave problems have been analyzed in \cite{mahdavi20dist_robustfl_neurips, hou21FedSP_arxiv, liao21local_AdaGrad_CC_arxiv}.
\cite{reisizadeh20robustfl_neurips} considers PL-PL and NC-PL minimax problems in the federated setting.
However, the clients only communicate min variables to the server. 
The limited client availability problem of FL is considered for NC-PL problems in \cite{xie21NC_PL_FL_arxiv}. 
However, the server is responsible for additional computations, to compute the global gradient estimates.
In our work, we consider a more general setting, where both the min and max variables need to be communicated to the server periodically.
The server is more limited in functionality, and only computes and returns the averages to the clients.
\cite{mahdavi20dist_robustfl_neurips} shows a suboptimal convergence rate for nonconvex-linear minimax problems (see \cref{table:comparison}). 
We consider more general NC-C problems, improve the convergence rate, and show linear speedup in $n$.

\paragraph{Comparison with \cite{mahdavi21localSGDA_aistats}.}
The work most closely related to ours is \cite{mahdavi21localSGDA_aistats}. 
The authors consider three classes of smooth nonconvex minimax functions: NC-SC, NC-PL, and NC-1PC. 
However, the gradient complexity and communication cost results achieved are suboptimal.
For all three classes of functions, we provide tighter analyses, resulting in improved gradient complexity with improved communication savings.
See \cref{table:comparison} for a comprehensive comparison of results.

\section{Preliminaries}
\label{sec:prelim}

\ificml
\else
\paragraph{Notations.} Throughout the paper, we let $\norm{\cdot}$ denote the Euclidean norm $\norm{\cdot}_2$.
Given a positive integer $m$, the set of numbers $\{ 1, 2, \hdots, m \}$ is denoted by $[m]$. Vectors at client $i$ are denoted with superscript $i$, for e.g., $\bx^i$.
Vectors at time $t$ are denoted with subscript $t$, for e.g., $\by_t$.
Average across clients appear without a superscript, for e.g., $\bxt = \frac{1}{n} \sumin \bxit$.
We define the gradient vector as $\G f_i(\bx, \by) = \lb \Gx f_i(\bx, \by)^{\top}, \Gy f_i(\bx, \by)^{\top} \rb^{\top}$.
For a generic function $g(\bx, \by)$, we denote its stochastic gradient vector as $\G g(\bx, \by; \xi^i) = \lb \Gx g(\bx, \by; \xi^i)^{\top}, \Gy g(\bx, \by; \xi^i)^{\top} \rb^{\top}$, where $\xi^i$ denotes the randomness.
\fi

\paragraph{Convergence Metrics.} Since the loss function $f$ is nonconvex, we cannot prove convergence to a global saddle point. 
We instead prove convergence to an \textit{approximate} stationary point, which is defined next.

\begin{definition}[$\epsilon$-Stationarity]
	\label{defn:stationarity}
	A point $\Tbx$ is an $\epsilon$-stationary point of a differentiable function $g$ if $\norm{\G g (\Tbx)} \leq \epsilon$.
\end{definition}

\begin{definition}
	Stochastic Gradient (SG) complexity is the total number of gradients computed by a single client during the course of the algorithm.
\end{definition}

Since all the algorithms analyzed in this paper are single-loop and use a $\mco (1)$ batchsize, if the algorithm runs for $T$ iterations, then the SG complexity is $\mco (T)$.

During a communication round, the clients send their local vectors to the server, where the aggregate is computed, and communicated back to the clients.
Consequently, we define the number of communication rounds as follows.

\begin{definition}[Communication Rounds]
	The number of communication rounds in an algorithm is the number of times clients communicate their local models to the server.
\end{definition}
If the clients perform $\sync$ local updates between successive communication rounds, the total number of communication rounds is $\lceil T/\sync \rceil$.
Next, we discuss the assumptions that will be used throughout the rest of the paper.

\begin{assump}[Smoothness]
	\label{assum:smoothness}
	Each local function $f_i$ is differentiable and has Lipschitz continuous gradients.
	That is, there exists a constant $\Lf > 0$ such that at each client $i \in [n]$, for all $\bx, \bx' \in \mbb R^{d_1}$ and $\by, \by' \in \mbb R^{d_2}$,
	\ificml
\else
	\begin{align*}
		\lnr \G f_i(\bx, \by) - \G f_i(\bx', \by') \rnr \leq \Lf \lnr (\bx, \by) - (\bx', \by') \rnr.
	\end{align*}
	\fi
\end{assump}

\begin{assump}[Bounded Variance]
	\label{assum:bdd_var}
	The stochastic gradient oracle at each client is unbiased with bounded variance, i.e., there exists a constant $\sigma > 0$ such that at each client $i \in [n]$, for all $\bx, \by$,
	\ificml
\else
	\begin{align*}
		\mbe_{\xi_i} [ \G f_i(\bx, \by; \xi^i) ] &= \G f_i(\bx, \by), \\
		\mbe_{\xi_i} \| \G f_i(\bx, \by; \xi^i) - \G f_i(\bx, \by) \|^2 & \leq \sigma^2.
	\end{align*}
	\fi
	
\end{assump}

\begin{assump}[Bounded Heterogeneity]
	\label{assum:bdd_hetero}
	To measure the heterogeneity of the local functions $\{ f_i(\bx, \by) \}$ across the clients, we define
	\ificml
\else
	\newline
	\begin{align*}
		\heterox^2 &= \sup_{\bx \in \mbb R^{d_1}, \by \in \mbb R^{d_2}} \frac{1}{n} \textstyle \sumin \lnr \Gx f_i(\bx, \by) - \Gx f(\bx, \by) \rnr^2, \\
		\heteroy^2 &= \sup_{\bx \in \mbb R^{d_1}, \by \in \mbb R^{d_2}} \frac{1}{n} \textstyle \sumin \lnr \Gy f_i(\bx, \by) - \Gy f(\bx, \by) \rnr^2.
	\end{align*}
	We assume that $\heterox$ and $\heteroy$ are bounded.
	\fi
\end{assump}

\section{Algorithms and their Convergence Analyses}
\label{sec:algo_theory}

In this section, we discuss local updates-based algorithms to solve nonconvex-concave and nonconvex-nonconcave minimax problems.
Each client runs multiple update steps on its local models using local stochastic gradients.
Periodically, the clients communicate their local models to the server, which returns the average model.
In this section, we demonstrate that this leads to communication savings at the clients, without sacrificing the convergence guarantees.

In the subsequent subsections, for each class of functions considered (NC-PL, NC-C, NC-1PC), we first discuss an algorithm.
Next, we present the convergence result, followed by a discussion of the gradient complexity and the communication cost needed to reach an $\epsilon$ stationary point.
See \cref{table:comparison} for a summary of our results, along with comparisons with the existing literature.

\subsection{Nonconvex-PL (NC-PL) Problems} \label{sec:NC_PL}

In this subsection, we consider smooth nonconvex functions which satisfy the following assumption.

\begin{assump}[Polyak {\L}ojasiewicz (PL) Condition in $\by$]
	\label{assum:PL_y}
	The function $f$ satisfies $\mu$-PL condition in $\by$ ($\mu > 0$), if for any fixed $\bx$: 1) $\max_{\by'} f(\bx, \by')$ has a nonempty solution set; 
	2) {\small$\norm{\Gy f(\bx, \by)}^2 \geq 2 \mu ( \max_{\by'} f(\bx, \by') - f(\bx, \by) )$}, for all $\by$.
\end{assump}

First, we present an improved convergence result for Local SGDA (\cref{alg_local_SGDA}), proposed in \cite{mahdavi21localSGDA_aistats}. Then we propose a novel momentum-based algorithm (\cref{alg_NC_momentum}), which achieves the same convergence guarantee, and has improved empirical performance (see \cref{sec:exp}).

\paragraph{Improved Convergence of Local SGDA.}
Local Stochastic Gradient Descent Ascent (SGDA) (\cref{alg_local_SGDA}) proposed in \cite{mahdavi21localSGDA_aistats}, is a simple extension of the centralized algorithm SGDA \cite{lin_GDA_icml20}, to incorporate local updates at the clients. At each time $t$, clients updates their local models $\{ \bxit, \byit \}$ using local stochastic gradients $\{ \Gx f_i (\bxit, \byit; \xiit), \Gy f_i (\bxit, \byit; \xiit) \}$.
Once every $\sync$ iterations, the clients communicate $\{ \bxit, \byit \}$ to the server, which computes the average models $\{ \bxt, \byt \}$, and returns these to the clients.
Next, we discuss the finite-time convergence of \cref{alg_local_SGDA}.
We prove convergence to an approximate stationary point of the envelope function $\Phi(\bx) = \max_\by f(\bx, \by)$.\footnote{Under Assumptions \ref{assum:smoothness}, \ref{assum:PL_y}, $\Phi$ is smooth \cite{nouiehed19minimax_neurips19}.}

\begin{algorithm}[ht]
	\caption{Local SGDA \cite{mahdavi21localSGDA_aistats}}
	\label{alg_local_SGDA}
	\begin{algorithmic}[1]
		\STATE{\textbf{Input: }{\small$\bx_0^i = \bx_0, \by_0^i = \by_0$}, for all $i \in [n]$; step-sizes $\lrx, \lry$; $\sync$, $T$}
		\FOR[At all clients $i=1,\hdots, n$]{$t=0$ to $T-1$}
		\STATE{Sample minibatch $\xiit$ from local data}
		\STATE{$\bxitp = \bxit - \lrx \Gx f_i (\bxit, \byit; \xiit)$}
		\STATE{$\byitp = \byit + \lry \Gy f_i (\bxit, \byit; \xiit)$}
		\IF{$t+1$ mod $\sync = 0$}
		\STATE{Clients send $\{ \bxitp, \byitp \}$ to the server}
		\STATE{Server computes averages $\bxtp \triangleq \frac{1}{n} \sumin \bxitp$, 
			$\bytp \triangleq \frac{1}{n} \sumin \byitp$, and sends to all the clients}
		\STATE{$\bxitp = \bxtp$, $\byitp = \bytp$, for all $i \in [n]$}
		\ENDIF
		\ENDFOR
		\STATE{\textbf{Return: }$\bbxT$ drawn uniformly at random from $\{ \bxt \}_{t=1}^T$, where $\bxt \triangleq \frac{1}{n} \sumin \bxit$}
	\end{algorithmic}
\end{algorithm}

\begin{theorem}
	\label{thm:NC_PL}
	Suppose the local loss functions $\{ f_i \}_i$ satisfy Assumptions \ref{assum:smoothness}, \ref{assum:bdd_var}, \ref{assum:bdd_hetero}, and the global function $f$ satisfies \cref{assum:PL_y}.
	Suppose the step-sizes $\lrx, \lry$ are chosen such that {\small$\lry \leq \frac{1}{8 \Lf \sync}$, $\frac{\lrx}{\lry} \leq \frac{1}{8 \kappa^2}$}, where {\small$\kappa = \Lf/\mu$} is the condition number.
	Then, for the output $\bbxT$ of \cref{alg_local_SGDA}, the following holds.
	\ificml
\else
	\begin{equation}
		\begin{aligned}
			& \mbe \norm{\G \Phi (\bbxT)}^2 \leq \underbrace{\mco \lp \kappa^2 \lb \frac{\Dphi}{\lry T} + \frac{\lry \sigma^2}{n} \rb \rp}_{\text{Error with full synchronization}} + \underbrace{\mco \lp \kappa^2 (\sync-1)^2 \lb \lry^2 \lp \sigma^2 + \heteroy^2 \rp + \lrx^2 \heterox^2 \rb \rp}_{\text{Error due to local updates}},
		\end{aligned}
		\label{eq:thm:NC_PL}
	\end{equation}
	\fi
	where {\small$\Phi(\cdot) \triangleq \max_\by f(\cdot, \by)$} is the envelope function, {\small$\Dphi \triangleq \Phi (\bx_0) - \min_\bx \Phi (\bx)$}.
	Using {\small$\lrx = \mco ( \frac{1}{\kappa^2} \sqrt{\frac{n}{T}} )$, $\lry = \mco ( \sqrt{n/T} )$}, we can bound {\small$\mbe \norm{\G \Phi (\bbxT)}^2$} as
	\ificml
\else
	\begin{align}
		& \mco \Big( \frac{\kappa^2 ( \sigma^2 + \Dphi )}{\sqrt{n T}} + \kappa^2 (\sync-1)^2 \frac{n ( \sigma^2 + \heterox^2 + \heteroy^2 )}{T} \Big).
		\label{eq:thm:NC_PL_conv_rate}
	\end{align}
	\fi
\end{theorem}

\begin{proof}
	See \cref{app:ncpl}.
\end{proof}

\begin{remark}
	\label{rem:NC_PL_local_SGDA_1}
	The first term of the error decomposition in \eqref{eq:thm:NC_PL} represents the optimization error for a fully synchronous algorithm ($\sync = 1$), in which the local models are averaged after every update.
	The second term arises due to the clients carrying out multiple $(\sync > 1)$ local updates between successive communication rounds.
	This term is impacted by the data heterogeneity across clients $\heterox, \heteroy$.
	Since the dependence on step-sizes $\lrx, \lry$ is quadratic, as seen in \eqref{eq:thm:NC_PL_conv_rate}, for small enough $\lrx, \lry$, and carefully chosen $\sync$, having multiple local updates does not impact the asymptotic convergence rate $\mco (1/\sqrt{nT})$.
\end{remark}

\begin{cor}
	\label{cor:NC_PL_comm_cost}
	To reach an $\epsilon$-accurate point $\bbxT$, assuming $T \geq \Theta (n^3)$, the stochastic gradient complexity of \cref{alg_local_SGDA} is $\mco (\kappa^4/(n \epsilon^4))$.
	The number of communication rounds required for the same is $T/\sync = \mco ( \kappa^3/\epsilon^{3} )$.
\end{cor}

\begin{remark}
	\label{rem:NC_PL_local_SGDA_2}
	Our analysis improves the existing complexity results for Local SGDA \cite{mahdavi21localSGDA_aistats}.
	The analysis in \cite{mahdavi21localSGDA_aistats} also requires the additional assumption of $G_x$-Lipschitz continuity of $f(\cdot, \by)$, which we do not need.
	The complexity result is optimal in $\epsilon$.\footnote{In terms of dependence on $\epsilon$, our complexity and communication results match the corresponding results for the simple smooth nonconvex minimization with local SGD \cite{yu19icml_momentum}.}
	To the best of our knowledge, this complexity guarantee does not exist in the prior literature even for $n=1$.\footnote{During the preparation of this manuscript, we came across the centralized minimax work \cite{yang21NCPL_arxiv}, which achieves {\small$\mco (\kappa^4/ \epsilon^{4})$}, using stochastic alternating GDA.}
	Further, we also provide communication savings, requiring model averaging only once every $\mco ( \kappa/(n \epsilon) )$ iterations.
\end{remark}

\begin{remark}[Nonconvex-Strongly-Concave (NC-SC) Problems]
	\label{rem:NC_PL_local_SGDA_3}
	Since the PL condition is more general than strong concavity, we also achieve the above result for NC-SC minimax problems.
	Moreover, unlike the analysis in \cite{lin_GDA_icml20} which necessarily requires $\mco (1/\epsilon^{2})$ batch-sizes, to the best of our knowledge, ours is the first result to achieve $\mco (1/\epsilon^{4})$ rate for SGDA with $\mco (1)$ batch-size.
\end{remark}

\paragraph{Momentum-based Local SGDA.}

Next, we propose a novel momentum-based local updates algorithm (\cref{alg_NC_momentum}) for NC-PL minimax problems.
The motivation behind using momentum in local updates is to control the effect of stochastic gradient noise,
via historic averaging of stochastic gradients.
Since momentum is widely used in practice for training deep neural networks, it is a natural question to ask, whether the same theoretical guarantees as Local SGDA can be proved for a momentum-based algorithm.
A similar question has been considered in \cite{yu19icml_momentum} in the context of smooth minimization problems.
\cref{alg_NC_momentum} is a local updates-based extension of the approach proposed in \cite{qiu20single_timescale_ncsc} for centralized problems.
At each step, each client uses momentum-based gradient estimators {\small$\{ \bdxit, \bdyit \}$} to arrive at intermediate iterates {\small$\{ \Tbxitp, \Tbyitp \}$}.
The local updated model is a convex combination of the intermediate iterate and the current model.
Once every $\sync$ iterations, the clients communicate {\small$\{ \bxit, \byit, \bdxit, \bdyit \}$} to the server, which computes the averages {\small$\{ \bxt, \byt, \bdxt, \bdyt \}$}, and returns these to the clients.\footnote{The direction estimates {\small$\{ \bdxit, \bdyit \}$} only need to be communicated for the sake of analysis. In our experiments in \cref{sec:exp}, as in Local SGDA, only the models are communicated.}

\ificml
\else
\begin{algorithm}[ht]
	\caption{Momentum Local SGDA}
	\label{alg_NC_momentum}
	\begin{algorithmic}[1]
		\STATE{\textbf{Input:} $\bx_0^i = \bx_0, \by_0^i = \by_0$, $\mbf d_{x,0}^i = \Gx f_i (\bx^i_0, \by^i_0; \xi^i_0)$, $\mbf d_{y,0}^i = \Gy f_i (\bx^i_0, \by^i_0; \xi^i_0)$} for all $i \in [n]; \lrx, \lry, \sync, T$
		\FOR[At all clients $i=1,\hdots, n$]{$t=0$ to $T-1$}
		\STATE{$\Tbxitp = \bxit - \lrx \bdxit$, 
			$\ \bxitp = \bxit + \cvxt ( \Tbxitp - \bxit )$}
		\STATE{$\Tbyitp = \byit + \lry \bdyit$, $\ \byitp = \byit + \cvxt ( \Tbyitp - \byit )$}
		\STATE{Sample minibatch $\xiitp$ from local data}
		\STATE{$\bdxitp = (1 - \momx \cvxt) \bdxit + \momx \cvxt \Gx f_i (\bxitp, \byitp; \xiitp)$}
		\STATE{$\bdyitp = (1 - \momy \cvxt) \bdyit + \momy \cvxt \Gy f_i (\bxitp, \byitp; \xiitp)$}
		\IF{$t+1$ mod $\sync = 0$}
		\STATE{Clients send $\{ \bxitp, \byitp, \bdxitp, \bdyitp \}$ to the server}
		\STATE{Server computes averages 
			\begin{align*}
				\bxtp \triangleq \frac{1}{n} \sumin \bxitp, \quad \bytp \triangleq \frac{1}{n} \sumin \byitp, \quad \bdxtp \triangleq \frac{1}{n} \sumin \bdxitp, \quad \bdytp \triangleq \frac{1}{n} \sumin \bdyitp
			\end{align*}
			\hspace{7mm} and sends to the clients}
		\STATE{$\bxitp = \bxtp$, $\byitp = \bytp$, $\bdxitp = \bdxtp$, $\bdyitp = \bdytp$, for all $i \in [n]$}
		\ENDIF
		\ENDFOR
		\STATE{\textbf{Return: }$\bbxT$ drawn uniformly at random from $\{ \bxt \}$, where $\bxt \triangleq \frac{1}{n} \sumin \bxit$}
	\end{algorithmic}
\end{algorithm}
\fi

Next, we discuss the finite-time convergence of \cref{alg_NC_momentum}.

\begin{theorem}
	\label{thm:NC_PL_mom}
	Suppose the local loss functions $\{ f_i \}_i$ satisfy Assumptions \ref{assum:smoothness}, \ref{assum:bdd_var}, \ref{assum:bdd_hetero}, and the global function $f$ satisfies \cref{assum:PL_y}.
	Suppose in \cref{alg_NC_momentum}, 
	$\momx = \momy = \mom = 3$, {\small$\cvxt \equiv \cvx \leq \min \big\{ \frac{\mom}{6 \Lf^2 (\lry^2 + \lrx^2)}, \frac{1}{48 \sync} \big\}$}, for all $t$, and the step-sizes $\lrx, \lry$ are chosen such that $\lry \leq \frac{\mu}{8 \Lf^2}$, and $\frac{\lrx}{\lry} \leq \frac{1}{20 \kappa^2}$, where {\small$\kappa = \Lf/\mu$} is the condition number.
	Then, for the output $\bbxT$ of \cref{alg_NC_momentum}, the following holds.
	\ificml
\else
	\begin{equation}
		\begin{aligned}
			\mbe \norm{\G \Phi (\bbxT)}^2 & \leq \underbrace{\mco \Big( \frac{\kappa^2}{\lry \cvx T} + \frac{\cvx}{\mu \lry} \frac{\sigma^2}{n} \Big)}_{\text{Error with full synchronization}} + \underbrace{\mco \big( (\sync - 1)^2 \cvx^2 ( \sigma^2 + \heterox^2 + \heteroy^2 ) \big)}_{\text{Error due to local updates}},
		\end{aligned}
		\label{eq:thm:NC_PL_mom}
	\end{equation}
	\fi
	where $\Phi(\cdot) \triangleq \max_\by f(\cdot, \by)$ is the envelope function.
	With {\small$\cvx = \sqrt{n/T}$}, the bound in \eqref{eq:thm:NC_PL_mom} simplifies to
	\ificml
\else
	\begin{align}
		\mco \Big( \frac{\kappa^2 + \sigma^2}{\sqrt{n T}} + (\sync-1)^2 \frac{n ( \sigma^2 + \heterox^2 + \heteroy^2 )}{T} \Big).
		\label{eq:thm:NC_PL_mom_conv_rate}
	\end{align}
	\fi
\end{theorem}

\begin{proof}
	See \cref{app:NC_PL_mom}.
\end{proof}

\begin{remark}
	As in the case of \cref{thm:NC_PL}, the second term in \eqref{eq:thm:NC_PL_mom} arises due to the clients carrying out multiple ($\sync > 1$) local updates between successive communication rounds. 
	However, the dependence of this term on $\cvx$ is quadratic. Therefore, as seen in \eqref{eq:thm:NC_PL_mom_conv_rate}, for small enough $\cvx$ and carefully chosen $\sync$, having multiple local updates does not affect the asymptotic convergence rate $\mco (1/\sqrt{nT})$.
\end{remark}

\begin{cor}
	\label{cor:NC_PL_mom_comm_cost}
	To reach an $\epsilon$-accurate point $\bbxT$, assuming $T \geq \Theta (n^3)$, the stochastic gradient complexity of \cref{alg_NC_momentum} is $\mco (\kappa^4/(n \epsilon^4))$.
	The number of communication rounds required for the same is $T/\sync = \mco ( \kappa^3/\epsilon^{3} )$.
\end{cor}

The stochastic gradient complexity and the number of communication rounds required are identical (up to multiplicative constants) for both \cref{alg_local_SGDA} and \cref{alg_NC_momentum}.
Therefore, the discussion following \cref{thm:NC_PL} (Remarks \ref{rem:NC_PL_local_SGDA_2}, \ref{rem:NC_PL_local_SGDA_3}) applies to \cref{thm:NC_PL_mom} as well.
We demonstrate the practical benefits of Momentum Local SGDA in \cref{sec:exp}.

\subsection{Nonconvex-Concave (NC-C) Problems} \label{sec:NC_C}

In this subsection, we consider smooth nonconvex functions which satisfy the following assumptions.

\begin{assump}[Concavity]
	\label{assum:concavity}
	The function $f$ is concave in $\by$ if for a fixed $\bx \in \mbb R^{d_1}$, for all  $\by, \by' \in \mbb R^{d_2}$,
	\ificml
\else
	\begin{align*}
		f(\bx, \by) \leq f(\bx, \by') + \lan \Gy f(\bx, \by'), \by - \by' \ran.
	\end{align*}
	\fi
\end{assump}

\begin{assump}[Lipschitz continuity in $\bx$]
	\label{assum:Lips_cont_x}
	For the function $f$, there exists a constant $G_x$, such that for each $\by \in \mbb R^{d_2}$, and all $\bx, \bx' \in \mbb R^{d_1}$,
	\ificml
\else
	\begin{align*}
		\norm{f(\bx, \by) - f(\bx', \by)} \leq G_x \norm{\bx - \bx'}.
	\end{align*}
	\fi
\end{assump}

In the absence of strong-concavity or PL condition on $\by$, the envelope function $\Phi(\bx) = \max_\by f(\bx, \by)$ defined earlier need not be smooth.
Instead, we use the alternate definition of stationarity, proposed in \cite{davis19wc_siam}, utilizing the Moreau envelope of $\Phi$, which is defined next.

\begin{definition}[Moreau Envelope]
	A function $\Phi_{\lambda}$ is the $\lambda$-Moreau envelope of $\Phi$, for $\lambda > 0$, if for all $\bx \in \mbb R^{d_1}$,
	\ificml
\else
	\begin{align*}
		\Phi_\lambda(\bx) = \min_{\bx'} \Phi (\bx') + \frac{1}{2 \lambda} \norm{\bx' - \bx}^2.   
	\end{align*}
	\fi
\end{definition}

A small value of $\norm{\G \Phi_\lambda(\bx)}$ implies that $\bx$ is near some point $\Tbx$ that is \textit{nearly stationary} for $\Phi$ \cite{drusvyatskiy19wc_mathprog}.
Hence, we focus on minimizing $\norm{\G \Phi_\lambda(\bx)}$.

\paragraph{Improved Convergence Analysis for NC-C Problems.}

For centralized NC-C problems, \cite{lin_GDA_icml20} analyze the convergence of SGDA.
However, this analysis does not seem amenable to local-updates-based modification.
Another alternative is a double-loop algorithm, which approximately solves the inner maximization problem $\max f(\bx, \cdot)$ after each $\bx$-update step.
However, double-loop algorithms are complicated to implement.
\cite{mahdavi21localSGDA_aistats} propose Local SGDA+ (see \cref{alg_local_SGDA_plus} in \cref{app:NC_C}), a modified version of SGDA \cite{lin_GDA_icml20}, to resolve this impasse.
Compared to Local SGDA, the $\bx$-updates are identical.
However, for the $\by$-updates, stochastic gradients $\Gy f_i (\Tbx, \byit; \xiit)$ are evaluated with the $x$-component fixed at $\Tbx$, which is updated every $S$ iterations.

In \cite{mahdavi21localSGDA_aistats}, Local SGDA+ is used for solving nonconvex-one-point-concave (NC-1PC) problems (see \cref{sec:NC_1PC}).
However, the guarantees provided are far from optimal (see \cref{table:comparison}).
In this and the following subsection, we present improved convergence results for Local SGDA+, for NC-C and NC-1PC minimax problems.

\begin{theorem}
	\label{thm:NC_C}
	Suppose the local loss functions $\{ f_i \}$ satisfy Assumptions \ref{assum:smoothness}, \ref{assum:bdd_var}, \ref{assum:bdd_hetero}, \ref{assum:concavity}, \ref{assum:Lips_cont_x}.
	Further, let $\norm{\byt}^2 \leq D$ for all $t$.
	Suppose the step-sizes $\lrx, \lry$ are chosen such that $\lrx, \lry \leq \frac{1}{8 \Lf \sync}$.
	Then, for the output $\bbxT$ of \cref{alg_local_SGDA_plus},
	\ificml
\else
	\begin{equation}
		\begin{aligned}
			\mbe \norm{\G \Phi_{1/2\Lf} (\bbxT)}^2 & \leq \underbrace{\mco \lp \frac{\widetilde{\Delta}_{\Phi}}{\lrx T} + \lrx \Big( G_x^2 + \frac{\sigma^2}{n} \Big) \rp + \mco \lp \frac{\lry \sigma^2}{n} + \Big[ \lrx G_x S \sqrt{G_x^2 + \sigma^2/n} + \frac{D}{\lry S} \Big] \rp}_{\text{Error with full synchronization}} \\
			& \qquad + \underbrace{\mco \Big( (\sync-1)^2 \lb \lp \lrx^2 + \lry^2 \rp \sigma^2 + \lp \lrx^2 \heterox^2 + \lry^2 \heteroy^2 \rp \rb \Big)}_{\text{Error due to local updates}},
		\end{aligned}
		\label{eq:thm:NC_C}
	\end{equation}
	\fi
	where {\small$\Phi_{1/2\Lf}(\bx) \triangleq \min_{\bx'} \Phi (\bx') + \Lf \norm{\bx' - \bx}^2$}, {\small$\widetilde{\Delta}_{\Phi} \triangleq \Phi_{1/2 \Lf} (\bx_0) - \min_\bx \Phi_{1/2 \Lf} (\bx)$}.
	Using {\small$S = \Theta ( \sqrt{T/n} )$}, {\small$\lrx = \Theta \lp \frac{n^{1/4}}{T^{3/4}} \rp$}, {\small$\lry = \Theta \lp \frac{n^{3/4}}{T^{1/4}} \rp$}, the bound in \eqref{eq:thm:NC_C} simplifies to
	\ificml
\else
	\begin{equation}
		\begin{aligned}
			& \mbe \norm{\G \Phi_{1/2\Lf} (\bbxT)}^2 \leq \underbrace{\mco \lp \frac{1}{(nT)^{1/4}} + \frac{n^{1/4}}{T^{3/4}} \rp}_{\text{Error with full synchronization}} + \underbrace{\mco \lp \frac{n^{3/2} (\sync-1)^2}{T^{1/2}} + (\sync-1)^2 \frac{\sqrt{n}}{T^{3/2}} \rp}_{\text{Error due to local updates}}.
		\end{aligned}
		\label{eq:thm:NC_C_conv_rate}
	\end{equation}
	\fi
\end{theorem}

\begin{proof}
	See \cref{app:NC_C}.
\end{proof}

\ificml
\else
\begin{remark}
	\label{rem:NC_C_local_SGDA_plus_1}
	The first term in the error decomposition in \eqref{eq:thm:NC_C}, represents the optimization error for a fully synchronous algorithm.
	This is exactly the error observed in the centralized case \cite{lin_GDA_icml20}.
	The second term arises due to multiple ($\sync > 1$) local updates.
	As seen in \eqref{eq:thm:NC_C_conv_rate}, for small enough $\lry, \lrx$, and carefully chosen $S, \sync$, this does not impact the asymptotic convergence rate $\mco (1/(nT)^{1/4})$.
\end{remark}
\fi

\begin{cor}
	\label{cor:NC_C_comm_cost}
	To reach an $\epsilon$-accurate point, i.e., $\bx$ such that {\small$\mbe \| \G \Phi_{1/2\Lf} (\bx) \| \leq \epsilon$},
	assuming {\small$T \geq \Theta (n^7)$},
	the stochastic gradient complexity of \cref{alg_local_SGDA_plus} is {\small$\mco (1/(n \epsilon^8))$}.
	The number of communication rounds required is {\small$T/\sync = \mco ( 1/\epsilon^{7} )$}.
\end{cor}

\begin{remark}
	\label{rem:NC_C_local_SGDA_plus_2}
	Ours is the first work to match the centralized ($n=1$) results in \cite{lin_GDA_icml20} ({\small$\mco ( 1/\epsilon^{8} )$} using SGDA), and provide linear speedup for $n>1$ with local updates.
	In addition, we also provide communication savings, requiring model averaging only once every {\small$\mco ( 1/(n \epsilon) )$} iterations.
\end{remark}

\subsection{Nonconvex-One-Point-Concave (NC-1PC) Problems} \label{sec:NC_1PC}

In this subsection, we consider smooth nonconvex functions which also satisfy the following assumption.

\begin{assump}[One-point-Concavity in $\by$]
	\label{assum:1pc_y}
	The function $f$ is said to be one-point-concave in $\by$ if fixing $\bx \in \mbb R^{d_1}$, for all  $\by \in \mbb R^{d_2}$,
	\ificml
\else
	\begin{align*}
		\lan \Gy f(\bx, \by'), \by - \by^*(\bx) \ran \leq f(\bx, \by) - f(\bx, \by^*(\bx)),   
	\end{align*}
	\fi
	where $\by^*(\bx) \in \argmax_\by f(\bx, \by)$.
\end{assump}

Due to space limitations, we only state the sample and communication complexity results for \cref{alg_local_SGDA_plus} with NC-1PC functions. The complete result is stated in \cref{app:NC_1PC}.

\begin{theorem} 
	\label{thm:NC_1PC}
	Suppose the local loss functions $\{ f_i \}$ satisfy Assumptions \ref{assum:smoothness}, \ref{assum:bdd_var}, \ref{assum:bdd_hetero}, \ref{assum:Lips_cont_x}, \ref{assum:1pc_y}.
	Further, let {\small$\norm{\byt}^2 \leq D$} for all $t$.
	Then, to reach a point $\bx$ such that {\small$\mbe \| \G \Phi_{1/2\Lf} (\bx) \| \leq \epsilon$}, the sample complexity of \cref{alg_local_SGDA_plus} is {\small$\mco (1/\epsilon^8)$}, and the number of communication rounds required is {\small$\mco ( 1/\epsilon^{7} )$}.
\end{theorem}

\begin{remark}
	\label{rem:NC_1PC_local_SGDA_plus_2}
	Since one-point-concavity is more general than concavity, for $n=1$, our gradient complexity result $\mco (1/\epsilon^8)$ generalizes the corresponding result for NC-C functions \cite{lin_GDA_icml20}. To the best of our knowledge, ours is the first work to provide this guarantee for NC-1PC problems. We also reduce the communication cost by requiring model averaging only once every $\mco ( 1/\epsilon )$ iterations.
	Further, our analysis improves the corresponding results in \cite{mahdavi21localSGDA_aistats} substantially (see \cref{table:comparison}).
\end{remark}

\section{Experiments}
\label{sec:exp}

In this section, we present the empirical performance of the algorithms discussed in the previous sections.
To evaluate the performance of Local SGDA and Momentum Local SGDA, we consider the problem of fair classification \cite{mohri19agnosticFL_icml, nouiehed19minimax_neurips19} using the FashionMNIST dataset \cite{xiao17fashionMNIST}.
Similarly, we evaluate the performance of Local SGDA+ and Momentum Local SGDA+, a momentum-based algorithm (see \cref{alg_mom_local_SGDA_plus} in \cref{app:add_exp}), on a robust neural network training problem \cite{madry18adversarial_iclr, sinha17certifiable_robust_iclr}, using the CIFAR10 dataset.
We conducted our experiments on a cluster of 20 machines (clients), each equipped with an NVIDIA TitanX GPU. Ethernet connections communicate the parameters and related information amongst the clients.
We implemented our algorithm based on parallel training tools offered by PyTorch 1.0.0 and Python 3.6.3.
Additional experimental results, and the details of the experiments, along with the specific parameter values can be found in \cref{app:add_exp}.

\subsection{Fair Classification}
\label{sec:exp_fair}
We consider the following NC-SC minimax formulation of the fair classification problem \cite{nouiehed19minimax_neurips19}.
\ificml
\else
\begin{align}
	\min_\bx \max_{\by \in \mc Y} \sum_{c=1}^C y_c F_c(\bx) -\frac{\lambda}{2} \norm{\by}^2,
	\label{eq:exp_fair_2}
\end{align}
\fi
where $\bx$ denotes the parameters of the NN, $F_1, F_2, \hdots, F_C$ denote the individual losses corresponding to the $C(=10)$ classes, and {\small$\mc Y = \{ \by \in \mbb R^C: y_c \geq 0, \sum_{c=1}^C y_c = 1 \}$}.

\ificml
\else
\begin{figure}[t]
	\centering
	\includegraphics[width=0.55\textwidth]{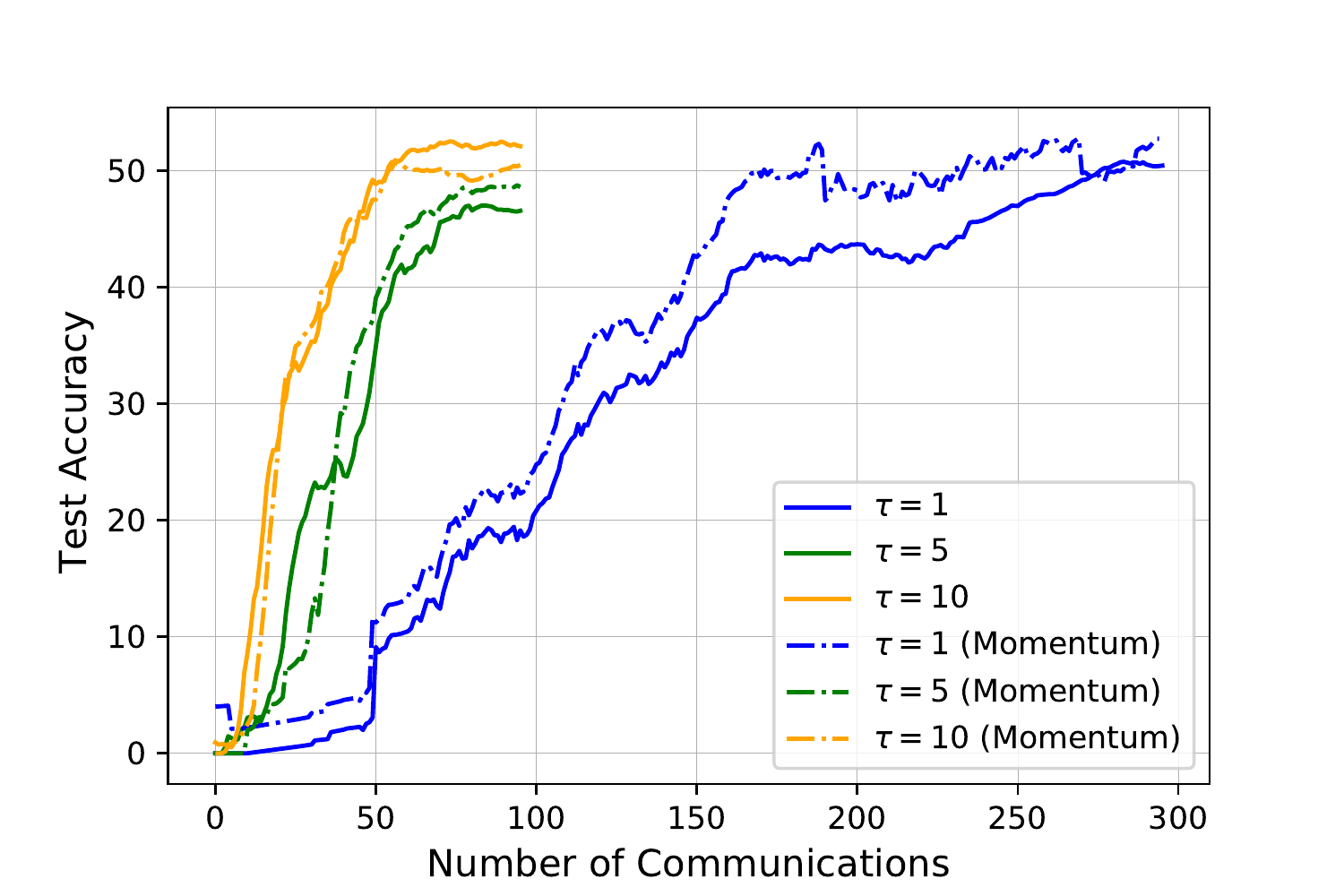}
	\vspace{-3mm}
	\caption{Comparison of the effects of increasing $\sync$ on the performance of Local SGDA and Momentum Local SGDA algorithms, for the fair classification problem on the FashionMNIST dataset, with a VGG11 model. The figure shows the test accuracy for the worst distribution. \label{fig:fairclass_fashionmnist}}
\end{figure}
\fi

We ran the experiment with a VGG11 network.
The network has $20$ clients.
The data is partitioned across the clients using a Dirichlet distribution $\text{Dir}_{20}(0.1)$ as in \cite{wang19FL_iclr}, to create a non-iid partitioning of data across clients. We use different values of synchronization frequency $\sync \in \{1, 5, 10\}$.
In accordance with \eqref{eq:exp_fair_2}, we plot the worst distribution test accuracy in \cref{fig:fairclass_fashionmnist}.
We plot the curves for the number of communications it takes to reach $50\%$ test accuracy on the worst distribution in each case.
From \cref{fig:fairclass_fashionmnist}, we see the communication savings which result from using higher values of $\sync$, since fully synchronized SGDA ($\sync = 1$) requires significantly more communication rounds to reach the same accuracy.
We also note the superior performance of Momentum Local SGDA, compared to Local SGDA.

\subsection{Robust Neural Network Training}
\label{sec:exp_robustnn}
Next, we consider the problem of robust neural network (NN) training, in the presence of adversarial perturbations \cite{madry18adversarial_iclr, sinha17certifiable_robust_iclr}.
We consider a similar problem as considered in \cite{mahdavi21localSGDA_aistats}.
\ificml
\else
\begin{align}
	\min_\bx \max_{\norm{\by}^2 \leq 1} \sum_{j=1}^N \ell \lp h_\bx (\mbf a_i + \by), b_i \rp, \label{eq:exp_robustnn}
\end{align}
\fi
where $\bx$ denotes the parameters of the NN, $\by$ denotes the perturbation, $(a_i, b_i)$ denotes the $i$-th data sample.
\ificml
\else
\begin{figure}[t]
	\centering
	\includegraphics[width=0.55\textwidth]{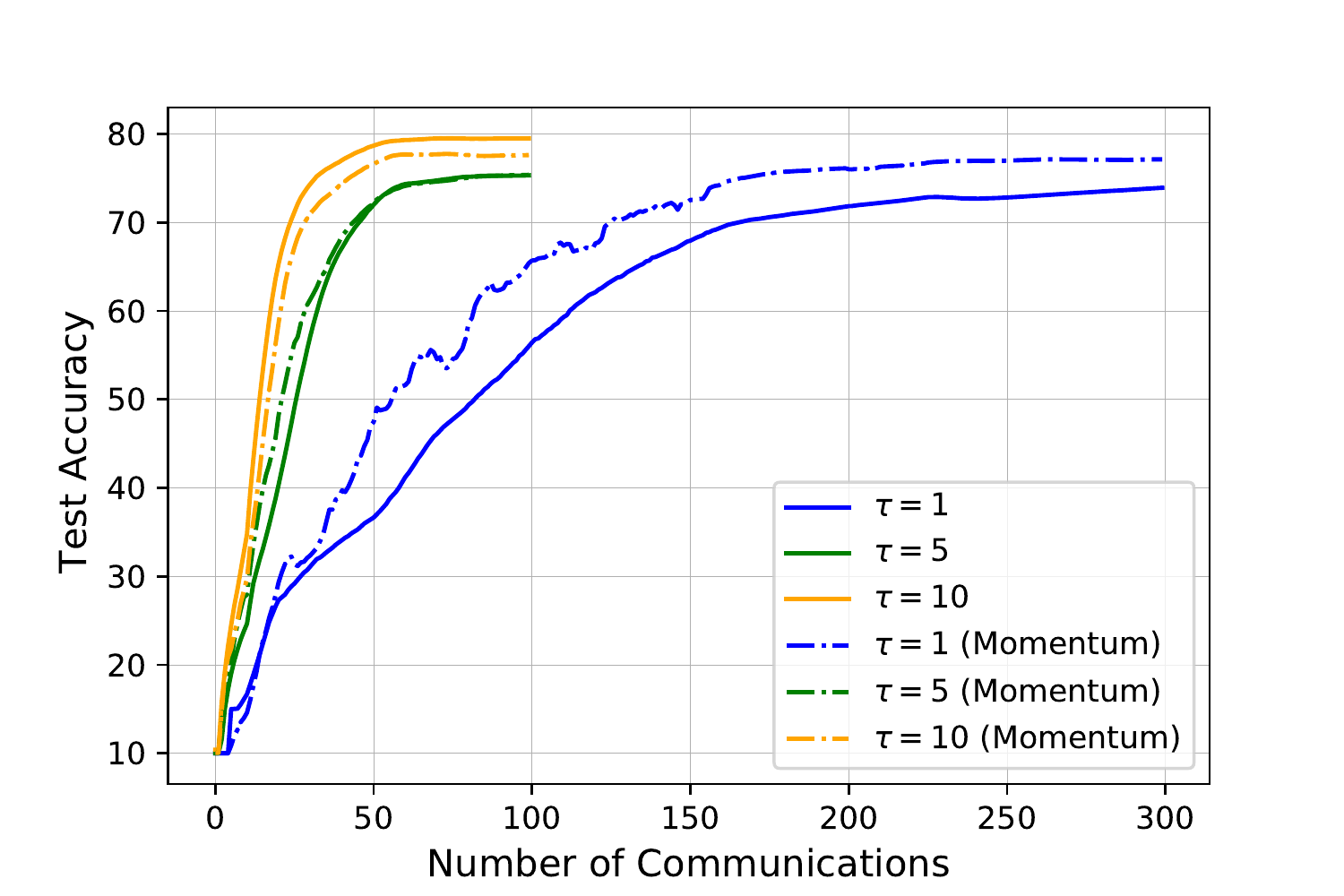}
	\vspace{-3mm}
	\caption{Comparison of the effects of $\sync$ on the performance of Local SGDA and Momentum Local SGDA algorithms, for the robust NN training problem on the CIFAR10 dataset, with the VGG11 model. The figure shows the robust test accuracy. \label{fig:robustnn_cifar10}}
\end{figure}
\fi

We ran the experiment using a VGG11 network, with the same network and data partitioning as in the previous subsection.
We use different values of $\sync \in \{1, 5, 10\}$.
For both Local SGDA+ and Momentum Local SGDA+, we use $S = \sync^2$.
In \cref{fig:robustnn_cifar10}, we plot the robust test accuracy. 
From \cref{fig:robustnn_cifar10}, we see the communication savings which result from using higher values of $\sync$, since for both the algorithms, $\sync = 1$ case requires significantly more communication rounds to reach the same accuracy.
We also note the superior performance of Momentum Local SGDA+, compared to Local SGDA+ to reach the same accuracy level.

\section{Concluding Remarks}
\label{sec:conclude}
In this work, we analyzed existing and newly proposed distributed communication-efficient algorithms for nonconvex minimax optimization problems.
We proved \textit{order-optimal} complexity results, along with communication savings, for several classes of minimax problems. Our results showed linear speedup in the number of clients, which enables scaling up distributed systems.
Our results for nonconvex-nonconcave functions improve the existing results for centralized minimax problems.
An interesting future direction is to analyze these algorithms for more complex systems with partial and erratic client participation \cite{gu21mifa_neurips, ruan21device_part_FL_aistats}, and with a heterogeneous number of local updates at each client \cite{joshi20fednova_neurips}.

\nocite{lin20near_opt_det_colt, nesterov18book, yoon21acc_ncc_icml, ouyang21lower_cc_bilinear_mathprog, wang20improved_cc_neurips, li21lower_bd_NCSC_neurips, lei21stability_Minimax_icml, zhang21NCSC_uai, kiyavash20catalyst_neurips, lee21NCNC_structured_neurips, lu20HiBSA_NC_C_ieee, tran20hybrid_NCLin_neurips, jin20local_opt_NCNC_icml, liang20proxGDA_KL_iclr, luo21near_opt_FS_cc_arxiv, xie20lower_FS_cc_icml, gasnikov21decen_deter_cc_icoa, ozdaglar19dec_prox_sp_arxiv, richtarik21dist_VI_comm_arxiv, gasnikov21dec_person_FL_arxiv, jacot18NTK_neurips}

	

	\bibliography{References}
	\bibliographystyle{icml2022}

	\newpage
	\appendix
	\onecolumn
	\newpage

\begin{center}
    {\LARGE \textbf{Appendices}}
\end{center}
The appendices are organized as follows. 
In Section \ref{app:prelim} we mention some basic mathematical results and inequalities which are used throughout the paper.
In Section \ref{app:ncpl} we prove the non-asymptotic convergence of Local SGDA \cref{alg_local_SGDA} for smooth nonconvex-PL (NC-SC) functions, and derive gradient complexity and communication cost of the algorithm to achieve an $\epsilon$-stationary point.
In \cref{app:NC_PL_mom}, we analyze the proposed Momentum Local SGDA algorithm (\cref{alg_NC_momentum}), for the same class of NC-PL functions.
Similarly, in the following sections, we prove the non-asymptotic convergence of \cref{alg_local_SGDA_plus} for smooth nonconvex-concave (NC-C) functions (in \cref{app:NC_C}), and for smooth nonconvex-1-point-concave (NC-1PC) functions (in \cref{app:NC_1PC}).
Finally, in \cref{app:add_exp} we provide the details of the additional experiments we performed.

\begin{table}[ht]
\caption{Abbreviations for the different classes of minimax problems $\min_\bx \max_\by f(\bx, \by)$ mentioned in the paper.}
\label{table_abbrv_sec}
\vskip 0.15in
\begin{center}
\begin{tabular}{|l|c|c|}
\hline
Function Class & Abbreviation & Our Work \\
\hline
\underline{S}trongly-\underline{C}onvex in  $\bx$, \underline{S}trongly-\underline{C}oncave in $\by$ & SC-SC & - \\
\underline{S}trongly-\underline{C}onvex in  $\bx$, \underline{C}oncave in $\by$ & SC-C & - \\
\underline{C}onvex in $\bx$,  \underline{C}oncave in $\by$ & C-C & - \\
\underline{N}on\underline{C}onvex in $\bx$, \underline{S}trongly-\underline{C}oncave in $\by$ & NC-SC & \checkmark \ (\cref{sec:NC_PL}) \\
\underline{N}on\underline{C}onvex in $\bx$ , \underline{PL} in $\by$ & NC-PL & \checkmark \ (\cref{sec:NC_PL}) \\
\underline{N}on\underline{C}onvex in $\bx$ , \underline{C}oncave in $\by$ & NC-C & \checkmark \ (\cref{sec:NC_C}) \\
\underline{N}on\underline{C}onvex in $\bx$, \underline{1}-\underline{P}oint-\underline{C}oncave in $\by$ & NC-1PC & \checkmark \ (\cref{sec:NC_1PC}) \\
\underline{PL} in $\bx$, \underline{PL} in $\by$ & PL-PL & - \\
\underline{N}on\underline{C}onvex in $\bx$, \underline{N}on-\underline{C}oncave in $\by$ & NC-NC & \checkmark \ (Sections \ref{sec:NC_PL}, \ref{sec:NC_1PC}) \\
\hline
\end{tabular}
\end{center}
\vskip -0.1in
\end{table}

\section{Preliminary Results}
\label{app:prelim}

\begin{lemma}[Young's inequality]
\label{lem:Young}
Given two same-dimensional vectors $\mbf u, \mbf v \in \mbb R^d$, the Euclidean inner product can be bounded as follows:
$$\lan \mbf u, \mbf v \ran \leq \frac{\norm{\mbf u}^2}{2 \gamma} + \frac{\gamma \norm{\mbf v}^2}{2}$$
for every constant $\gamma > 0$.
\end{lemma}

\begin{lemma}[Strong Concavity]
A function $g: \mc X \times \mc Y$ is strongly concave in $\by$, if there exists a constant $\mu > 0$, such that for all $\bx \in \mc X$, and for all $\by, \by' \in \mc Y$, the following inequality holds.
$$g(\bx, \by) \leq g(\bx, \by') + \lan \Gy g(\bx, \by'), \by' - \by \ran - \frac{\mu}{2} \norm{\by - \by'}^2.$$
\end{lemma}

\begin{lemma}[Jensen's inequality]
\label{lem:jensens}
Given a convex function $f$ and a random variable $X$, the following holds.
$$f \lp \mbe [X] \rp \leq \mbe \lb f(X) \rb.$$
\end{lemma}

\begin{lemma}[Sum of squares]
\label{lem:sum_of_squares}
For a positive integer $K$, and a set of vectors $x_1, \hdots, x_K$, the following holds:
\begin{align*}
    \norm{\sum_{k=1}^K x_k}^2 \leq K \sum_{k=1}^K \norm{x_k}^2.
\end{align*}
\end{lemma}

\begin{lemma}[Quadratic growth condition \cite{schmidt16lin_conv_PL_kdd}]
\label{lem:quad_growth}
If function $g$ satisfies Assumptions \ref{assum:smoothness}, \ref{assum:PL_y}, then for all $x$, the following conditions holds
\begin{align*}
    g(x) - \min_{z} g(z) & \geq \frac{\mu}{2} \norm{x_p - x}^2, \\
    \norm{\G g(x)}^2 & \geq 2 \mu \lp g(x) - \min_z g(z) \rp.
\end{align*}
\end{lemma}

\subsection{Local SGD}
\label{app:local_SGD}
Local SGD is the algorithm which forms the basis of numerous Federated Learning algorithms \cite{konevcny16federated, fedavg17aistats}.
Each client running Local SGD (\cref{alg_local_SGD}), runs a few SGD iterations locally and only then communicates with the server, which in turn computes the average and returns to the clients. 
This approach saves the limited communication resources of the clients, without sacrificing the convergence guarantees.

The algorithm has been analyzed for both convex and nonconvex minimization problems.
With identical distribution of client data, Local SGD has been analyzed in \cite{stich18localSGD_iclr, stich20error_fb_jmlr, khaled20localSGD_aistats, spiridonoff21comm_eff_SGD_neurips} for (strongly) convex objectives, and in \cite{wang21coopSGD_jmlr, zhou18localSGD_ijcai} for nonconvex objectives.
With heterogeneous client data Local SGD has been analyzed in \cite{khaled20localSGD_aistats, koloskova20unified_localSGD_icml} for (strongly) convex objectives, and in \cite{jiang18linear_neurips, haddadpour19conv_FL_arxiv, koloskova20unified_localSGD_icml} for nonconvex objectives.

\begin{algorithm}[ht]
\caption{Local SGD}
\label{alg_local_SGD}
\begin{algorithmic}[1]
	\STATE{\textbf{Input: }{\small$\bx_0^i = \bx_0$}, for all $i \in [n]$, step-size $\lr$, $\sync$, $T$}
	\FOR[At all clients $i=1,\hdots, n$]{$t=0$ to $T-1$}
	    \STATE{Sample minibatch $\xiit$ from local data}
        \STATE{$\bxitp = \bxit - \lr \G g_i (\bxit; \xiit)$}
        \IF{$t+1$ mod $\sync = 0$}
            \STATE{Clients send $\{ \bxitp \}$ to the server}
            \STATE{Server computes averages $\bxtp \triangleq \frac{1}{n} \sumin \bxitp$, and sends to all the clients}
            \STATE{$\bxitp = \bxtp$, for all $i \in [n]$}
        \ENDIF
	\ENDFOR
	\STATE{\textbf{Return: }$\bbxT$ drawn uniformly at random from $\{ \bxt \}$, where $\bxt \triangleq \frac{1}{n} \sumin \bxit$}
\end{algorithmic}
\end{algorithm}

\begin{lemma}[Local SGD for Convex Function Minimization \cite{khaled20localSGD_aistats}]
\label{lem:local_SGD_khaled}

Suppose that the local functions $\{ g_i \}$ satisfy Assumptions \ref{assum:smoothness}, \ref{assum:bdd_var}, \ref{assum:bdd_hetero}, and are all convex.\footnote{The result actually holds under slightly weaker assumptions on the noise and heterogeneity.}
Suppose, the step-size $\eta$ is chosen such that $\eta \leq \min \lcb \frac{1}{4 \Lf}, \frac{1}{8 \Lf (\tau - 1)} \rcb$.
Then, the iterates generated by Local SGD (\cref{alg_local_SGD}) algorithm satisfy
\begin{align*}
    \mbe \lb g(\bbxT) \rb - g(\bx^*) \leq \avgtT \mbe \lb g(\bxt) - g(\bx^*) \rb \leq \frac{4 \norm{\bx_0 - \bx^*}^2}{\eta T} + \frac{20 \eta \sigma^2}{n} + 16 \eta^2 \Lf (\sync-1)^2 \lp \sigma^2 + \heterox^2 \rp,
\end{align*}
where $\bbxT \triangleq \avgtT \bxt$.
\end{lemma}


\newpage
\section{Nonconvex-PL (NC-PL) Functions: Local SGDA (\texorpdfstring{\cref{thm:NC_PL}}{Theorem 1})} \label{app:ncpl}
In this section we prove the convergence of \cref{alg_local_SGDA} for Nonconvex-PL functions, and provide the complexity and communication guarantees.

We organize this section as follows. First, in \cref{sec:NC_PL_int_results} we present some intermediate results, which we use to prove the main theorem. Next, in \cref{sec:NC_PL_thm_proof}, we present the proof of \cref{thm:NC_PL}, which is followed by the proofs of the intermediate results in \cref{sec:NC_PL_int_results_proofs}.
We utilize some of the proof techniques of \cite{mahdavi21localSGDA_aistats}.
However, the algorithm we analyze for NC-PL functions is different. Also, we provide an improved analysis, resulting in better convergence guarantees.

The problem we solve is
\begin{align*}
    \min_{\bx} \max_{\by} \lcb f(\bx, \by) \triangleq \frac{1}{n} \sumin f_i(\bx, \by) \rcb.
\end{align*}
We define
\begin{align}
    \Phi (\bx) \triangleq \max_{\by} f(\bx, \by) \quad \text{and} \quad \by^* (\bx) \in \argmax_{\by} f(\bx, \by).
\end{align}
Since $f(\bx, \cdot)$ is $\mu$-PL, $\by^*(\bx)$ \textit{need not} be unique.

For the sake of analysis, we define \textit{virtual} sequences of average iterates:
\begin{align*}
    & \bxt \triangleq \frac{1}{n} \sumin \bxit; \quad \byt \triangleq \frac{1}{n} \sumin \byit.
\end{align*}
Note that these sequences are constructed only for the sake of analysis. During an actual run of the algorithm, these sequences exist only at the time instants when the clients communicate with the server.
We next write the update expressions for these virtual sequences, using the updates in Algorithm \ref{alg_local_SGDA}.
\begin{equation}
    \begin{aligned}
        \bxtp &= \bxt - \lrx \frac{1}{n} \sumin \Gx f_i (\bxit, \byit; \xiit) \\
        \bytp &= \byt + \lry \frac{1}{n} \sumin \Gy f_i (\bxit, \byit; \xiit)
    \end{aligned}
    \label{eq:NC_PL_update_avg}
\end{equation}
Next, we present some intermediate results which we use in the proof of \cref{thm:NC_PL}. To make the proof concise, the proofs of these intermediate results is relegated to \cref{sec:NC_PL_int_results_proofs}.

\subsection{Intermediate Lemmas} \label{sec:NC_PL_int_results}

We use the following result from \cite{nouiehed19minimax_neurips19} about the smoothness of $\Phi(\cdot)$.

\begin{lemma}
\label{lem:Phi_smooth_nouiehed}
If the function $f(\bx, \cdot)$ satisfies Assumptions \ref{assum:smoothness}, \ref{assum:PL_y} ($\Lf$-smoothness and $\mu$-PL condition in $\by$), then $\Phi (\bx)$ is $\Lp$-smooth with $\Lp = \kappa L/2 + L$, where $\kappa = L/\mu$ is the condition number.
\end{lemma}

\begin{lemma}
\label{lem:NC_PL_Phi_decay_one_iter}
Suppose the local client loss functions $\{ f_i \}$ satisfy Assumptions \ref{assum:smoothness}, \ref{assum:PL_y} and the stochastic oracles for the local functions satisfy \cref{assum:bdd_var}.
Then the iterates generated by \cref{alg_local_SGDA} satisfy
\begin{equation}
    \begin{aligned}
        \mbe \lb \Phi (\bxtp) \rb & \leq \mbe \lb \Phi (\bxt) \rb - \frac{\lrx}{2} \mbe \norm{\G \Phi (\bxt)}^2 - \frac{\lrx}{2} \lp 1 - \Lp \lrx \rp \mbe \norm{\frac{1}{n} \sumin \Gx f_i(\bxit, \byit)}^2 \nn \\
        & \quad + \frac{2 \lrx \Lf^2}{\mu} \mbe \lb \Phi (\bxt) - f(\bxt, \byt) \rb + 2 \lrx \Lf^2 \CExyt + \frac{\Lp \lrx^2 \sigma^2}{2 n},
    \end{aligned}
\end{equation}
where, we define $\CExyt \triangleq \frac{1}{n} \sumin \mbe \lp \lnr \bxit - \bxt \rnr^2 + \lnr \byit - \byt \rnr^2 \rp$, the synchronization error.
\end{lemma}

\begin{lemma}
\label{lem:NC_PL_phi_error}
Suppose the local loss functions $\{ f_i \}$ satisfy Assumptions \ref{assum:smoothness}, \ref{assum:bdd_hetero},
and the stochastic oracles for the local
functions satisfy \cref{assum:bdd_var}.
Further, in \cref{alg_local_SGDA}, we choose step-sizes $\lrx, \lry$ satisfying $\lry \leq 1/\mu$, $\frac{\lrx}{\lry} \leq \frac{1}{8 \kappa^2}$.
Then the following inequality holds.
\begin{equation}
    \begin{aligned}
        & \frac{1}{T} \sumtT \mbe \lp \Phi (\bxt) - f(\bxt, \byt) \rp \\
        & \leq \frac{2 \lp \Phi (\bx_0) - f(\bx_0, \by_0) \rp}{\lry \mu T} + \frac{2 \Lf^2}{\mu \lry} \lp 2 \lrx (1 - \lry \mu) + \lry \rp \frac{1}{T} \sumtT \CExyt + (1 - \lry \mu) \frac{\lrx}{\lry \mu} \frac{1}{T} \sumtT \mbe \norm{\G \Phi(\bxt)}^2 \nn \\
        & \quad + \lb (1 - \lry \mu) \frac{\lrx^2}{2} \lp \Lf + \Lp \rp + \lry \Lf^2 \lrx^2 \rb \frac{2}{\lry \mu T} \sumtT \mbe \norm{\frac{1}{n} \sumin \Gx f_i(\bxit, \byit)}^2 \\
        & \quad + \frac{\sigma^2}{\mu n} \lp \lry \Lf + 2 \Lf^2 \lrx^2 \rp + \frac{(1 - \lry \mu)}{\mu \lry} \frac{\lrx^2 \sigma^2}{n} \lp \Lf + \Lp \rp.
    \end{aligned}
\end{equation}
\end{lemma}

\begin{remark}[Comparison with \cite{mahdavi21localSGDA_aistats}]
Note that to derive a result similar to \cref{lem:NC_PL_phi_error}, the analysis in \cite{mahdavi21localSGDA_aistats} requires the additional assumption of $G_x$-Lipschitz continuity of $f(\cdot, \by)$.
Also, the algorithm we analyze (Local SGDA) is simpler than the algorithm analyzed in \cite{mahdavi21localSGDA_aistats} for NC-PL functions.
\end{remark}

\begin{lemma}
\label{lem:NC_PL_consensus_error}
Suppose the local loss functions $\{ f_i \}$ satisfy Assumptions \ref{assum:smoothness}, \ref{assum:bdd_hetero},
and the stochastic oracles for the local
functions satisfy \cref{assum:bdd_var}.
Further, in \cref{alg_local_SGDA}, we choose step-sizes $\lrx, \lry \leq \frac{1}{8 \sync \Lf}$.
Then, the iterates $\{ \bxit, \byit \}$ generated by \cref{alg_local_SGDA} satisfy
\begin{equation}
    \begin{aligned}
        \frac{1}{T} \sumtT \CExyt & \triangleq \frac{1}{T} \sumtT \frac{1}{n} \sumin \mbe \lp \lnr \bxit - \bxt \rnr^2 + \lnr \byit - \byt \rnr^2 \rp \nn \\
        & \leq 2 (\sync-1)^2 \lp \lrx^2 + \lry^2 \rp \sigma^2 \lp 1 + \frac{1}{n} \rp + 6 (\sync-1)^2 \lp \lrx^2 \heterox^2 + \lry^2 \heteroy^2 \rp.
    \end{aligned}
    \label{eq:lem:NC_PL_consensus_error}
\end{equation}
\end{lemma}

\subsection{Proof of \texorpdfstring{\cref{thm:NC_PL}}{Theorem 1}}
\label{sec:NC_PL_thm_proof}
For the sake of completeness, we first state the full statement of \cref{thm:NC_PL} here.

\begin{theorem*}
Suppose the local loss functions $\{ f_i \}_i$ satisfy Assumptions \ref{assum:smoothness}, \ref{assum:bdd_var}, \ref{assum:bdd_hetero}, and the global function $f$ satisfies \cref{assum:PL_y}.
Suppose the step-sizes $\lrx, \lry$ are chosen
such that $\lry \leq \frac{1}{8 \Lf \sync}$, $\frac{\lrx}{\lry} = \frac{1}{8 \kappa^2}$, where $\kappa = \frac{\Lf}{\mu}$ is the condition number.
Then for the output $\bbxT$ of \cref{alg_local_SGDA}, the following holds.
\begin{align}
    \mbe \norm{\G \Phi (\bbxT)}^2 = & \frac{1}{T} \sumtT \mbe \norm{\G \Phi (\bxt)}^2 \nn \\
    & \leq \underbrace{\mco \lp \kappa^2 \lb \frac{\Dphi}{\lry T} + \frac{\Lf \lry \sigma^2}{n} \rb \rp}_{\text{Error with full synchronization}} + \underbrace{\mco \lp \Lf^2 \kappa^2 (\sync-1)^2 \lb \lry^2 \lp \sigma^2 + \heteroy^2 \rp + \lrx^2 \heterox^2 \rb \rp}_{\text{Error due to local updates}},
    \label{eq_proof:thm_NC_PL}
\end{align}
where $\Phi(\bx) \triangleq \max_\by f(\bx, \by)$ is the envelope function, $\Dphi \triangleq \Phi (\bx_0) - \min_\bx \Phi (\bx)$.
Using $\lry = \sqrt{\frac{n}{\Lf T}}$ and $\lrx = \frac{1}{8 \kappa^{2}} \sqrt{\frac{n}{\Lf T}}$, we get
\begin{align}
    & \mbe \norm{\G \Phi (\bbxT)}^2 \leq \mco \lp \frac{\kappa^2 \lp \sigma^2 + \Dphi \rp}{\sqrt{n T}} + \kappa^2 (\sync-1)^2 \frac{n \lp \sigma^2 + \heterox^2 + \heteroy^2 \rp}{T} \rp. \nn
\end{align}
\end{theorem*}

\begin{proof}
We start by summing the expression in \cref{lem:NC_PL_Phi_decay_one_iter} over $t = 0, \hdots, T-1$.
\begin{align}
    \frac{1}{T} \sumtT \mbe \lb \Phi (\bxtp) -  \Phi (\bxt) \rb & \leq - \frac{\lrx}{2} \frac{1}{T} \sumtT \mbe \norm{\G \Phi (\bxt)}^2 - \frac{\lrx}{2} \lp 1 - \Lp \lrx \rp \frac{1}{T} \sumtT \mbe \norm{\frac{1}{n} \sumin \Gx f_i(\bxit, \byit)}^2 \nn \\
    & \quad + \frac{2 \lrx \Lf^2}{\mu} \frac{1}{T} \sumtT \mbe \lb \Phi (\bxt) - F(\bxt, \byt) \rb + 2 \lrx \Lf^2 \frac{1}{T} \sumtT \CExyt + \frac{\Lp \lrx^2 \sigma^2}{2 n}. \label{eq_proof:thm:NC_PL_1}
\end{align}
Substituting the bound on $\frac{1}{T} \sumtT \CExyt$ from \cref{lem:NC_PL_consensus_error}, and the bound on $\frac{1}{T} \sumtT \mbe \lb \Phi (\bxt) - F(\bxt, \byt) \rb$ from \cref{lem:NC_PL_phi_error}, and rearranging the terms in \eqref{eq_proof:thm:NC_PL_1}, we get
\begin{align}
    & \frac{\mbe \Phi (\bx_T) - \Phi (\bx_0)}{T} \nn \\
    & \leq - \underbrace{\lp \frac{\lrx}{2} - (1 - \lry \mu) \frac{2 \lrx^2 \Lf^2}{\lry \mu^2} \rp}_{\geq \lrx/4} \frac{1}{T} \sumtT \mbe \norm{\G \Phi (\bxt)}^2 \nn \\
    & \quad - \underbrace{\frac{\lrx}{2} \lp 1 - \Lp \lrx - \frac{8 \Lf^2}{\mu^2 \lry} \lb (1 - \lry \mu) \frac{\lrx^2}{2} \lp L + \Lp \rp + \lry \Lf^2 \lrx^2 \rb \rp}_{\geq 0} \frac{1}{T} \sumtT \mbe \norm{\frac{1}{n} \sumin \Gx f_i(\bxit, \byit)}^2 \nn \\
    & \quad + \lb \frac{2 \lrx \Lf^2}{\mu} \lp \frac{2 \Lf^2}{\mu} + \frac{4 \lrx \Lf^2 (1 - \lry \mu)}{\mu \lry} \rp + 2 \lrx \Lf^2 \rb \frac{1}{T} \sumtT \CExyt \nn \\
    & \quad + \frac{2 \lrx \Lf^2}{\mu} \lb \frac{2 \lp \Phi (\bx_0) - f(\bx_0, \by_0) \rp}{\lry \mu T} + \frac{\sigma^2}{\mu n} \lp \lry \Lf + 2 \Lf^2 \lrx^2 \rp + \frac{(1 - \lry \mu)}{\mu \lry} \frac{\lrx^2 \sigma^2}{n} \lp \Lf + \Lp \rp \rb + \frac{\Lp \lrx^2 \sigma^2}{2 n}. \label{eq_proof:thm:NC_PL_2}
\end{align}
Here, $\frac{\lrx}{2} - \frac{2 \lrx^2 (1-\mu \lry)\Lf^2}{\mu^2 \lry} \geq \frac{\lrx}{4}$ holds since $\frac{\lrx}{\lry} \leq \frac{1}{8 \kappa^2}$.
Also, $1 - \Lp \lrx - \frac{8 \Lf^2}{\mu^2 \lry} \lb (1 - \lry \mu) \frac{\lrx^2}{2} \lp L + \Lp \rp + \lry \Lf^2 \lrx^2 \rb \geq 0$ follows from the bounds on $\lrx, \lry$.
Rearranging the terms in \eqref{eq_proof:thm:NC_PL_2} and using \cref{lem:NC_PL_consensus_error}, we get
\begin{align}
    & \frac{1}{T} \sumtT \mbe \norm{\G \Phi (\bxt)}^2 \leq \frac{4 \lp \Phi (\bx_0) - \mbe \Phi (\bx_T) \rp}{\lrx T} \nn \\
    & \quad + \frac{4}{\lrx} 2 \lrx \Lf^2 \lb 1 + 2 \kappa^2 + 4 \kappa^2 \frac{\lrx}{\lry} \rb 2 (\sync-1)^2 \lb \lp \lrx^2 + \lry^2 \rp \sigma^2 \lp 1 + \frac{1}{n} \rp + 3 \lp \lrx^2 \heterox^2 + \lry^2 \heteroy^2 \rp \rb \nn \\
    & \quad + \frac{4}{\lrx} \lb \frac{4 \lrx \kappa^2}{\lry} \frac{\lp \Phi (\bx_0) - f(\bx_0, \by_0) \rp}{T} + \frac{2 \lrx \kappa^2 \sigma^2}{n} \lp \lry \Lf + 2 \Lf^2 \lrx^2 \rp + \frac{2 \lrx \kappa^2}{\lry} \frac{\lrx^2 \sigma^2}{n} \lp \Lf + \Lp \rp \rb + \frac{4}{\lrx} \frac{\Lp \lrx^2 \sigma^2}{2 n} \nn \\
    & \overset{(a)}{\leq} \frac{4 \Dphi}{\lrx T} + 8 \Lf^2 \lb 2 + 2 \kappa^2 \rb 2 (\sync-1)^2 \lb \lp \lrx^2 + \lry^2 \rp \sigma^2 \lp 1 + \frac{1}{n} \rp + 3 \lp \lrx^2 \heterox^2 + \lry^2 \heteroy^2 \rp \rb \nn \\
    & \quad + \frac{16 \kappa^2 \Dphi}{\lry T} + \frac{8 \kappa^2 \sigma^2}{n} \lp \lry \Lf + 2 \Lf^2 \lrx^2 \rp + \frac{8 \kappa^2 \lrx}{\lry} \frac{\lrx \sigma^2}{n} \lp \Lf + \Lp \rp + \frac{2 \Lp \lrx \sigma^2}{n} \nn \\
    & \overset{(b)}{\leq} \frac{4 \Dphi}{\lrx T} + 192 \Lf^2 \kappa^2 (\sync-1)^2 \lb \lp \lrx^2 + \lry^2 \rp \sigma^2 + \lrx^2 \heterox^2 + \lry^2 \heteroy^2 \rb + \frac{16 \kappa^2 \Dphi}{\lry T} + \frac{8 \kappa^2 \sigma^2}{n} \lp \lry \Lf + 2 \Lf^2 \lrx^2 \rp + \frac{4 \Lp \lrx \sigma^2}{n} \nn \\
    & = \mco \lp \frac{\Dphi}{\lrx T} + \frac{\Lp \lrx \sigma^2}{n} + \kappa^2 \lb \frac{\Dphi}{\lry T} + \frac{\Lf \lry \sigma^2}{n} \rb + \Lf^2 \kappa^2 (\sync-1)^2 \lb \lp \lrx^2 + \lry^2 \rp \sigma^2 + \lrx^2 \heterox^2 + \lry^2 \heteroy^2 \rb \rp. \nn \\
    & = \underbrace{\mco \lp \kappa^2 \lb \frac{\Dphi}{\lry T} + \frac{\Lf \lry \sigma^2}{n} \rb \rp}_{\text{Error with full synchronization}} + \underbrace{\mco \lp \Lf^2 \kappa^2 (\sync-1)^2 \lb \lry^2 \lp \sigma^2 + \heteroy^2 \rp + \lrx^2 \heterox^2 \rb \rp}_{\text{Error due to local updates}}. \tag{$\because \kappa \geq 1$}
\end{align}
where, we denote $\Dphi \triangleq \Phi (\bx_0) - \min_\bx \Phi (\bx)$.
$(a)$ follows from $\frac{\lrx}{\lry} \leq \frac{1}{8 \kappa^2}$;
$(b)$ follows since $\kappa \geq 1$ and $\Lp \geq \Lf$.
Therefore, $\frac{8 \kappa^2 \lrx}{\lry} \frac{\lrx \sigma^2}{n} (\Lf + \Lp) \leq \frac{\lrx \sigma^2}{n} (\Lf + \Lp) \leq \frac{2 \Lp \lrx \sigma^2}{n}$, which results in \eqref{eq_proof:thm_NC_PL}.

Using $\lry = \sqrt{\frac{n}{\Lf T}}$ and $\lrx = \frac{1}{8 \kappa^{2}} \sqrt{\frac{n}{\Lf T}} \leq \frac{\lry}{8 \kappa^2}$, and since $\kappa \geq 1$, we get
\begin{align}
    & \frac{1}{T} \sumtT \mbe \lnr \G \Phi (\bxt) \rnr^2 \leq \mco \lp \frac{\kappa^2 \lp \sigma^2 + \Dphi \rp}{\sqrt{n T}} + \kappa^2 (\sync-1)^2 \frac{n}{T} \lb \sigma^2 + \frac{\heterox^2}{\kappa^4} + \heteroy^2 \rb \rp. \nn
\end{align}
\end{proof}

\begin{proof}[Proof of \cref{cor:NC_PL_comm_cost}]
We assume $T \geq n^3$.
To reach an $\epsilon$-accurate point, i.e., $\bx$ such that $\mbe \lnr \G \Phi (\bx) \rnr \leq \epsilon$, we need
\begin{align*}
    \mbe \lnr \G \Phi (\bbxT) \rnr = \lb \frac{1}{T} \sumtT \mbe \lnr \G \Phi (\bxt) \rnr^2 \rb^{1/2} \leq \mco \lp \frac{\kappa \sqrt{\sigma^2 + \Dphi}}{(nT)^{1/4}} + \kappa (\sync-1) \sqrt{\frac{n \lp \sigma^2 + \heterox^2 + \heteroy^2 \rp}{T}} \rp.
\end{align*}
If we choose $\sync = \mco \lp \frac{T^{1/4}}{n^{3/4}} \rp$, we need $T = \mco \lp \kappa^4/(n \epsilon^4) \rp$ iterations, to reach an $\epsilon$-accurate point.
The number of communication rounds is $\mc O \lp \frac{T}{\sync} \rp = \mc O \lp (n T)^{3/4} \rp = \mco \lp \kappa^3/\epsilon^3 \rp$. 
\end{proof}


\subsection{Proofs of the Intermediate Lemmas}
\label{sec:NC_PL_int_results_proofs}

\begin{proof}[Proof of Lemma \ref{lem:NC_PL_Phi_decay_one_iter}]
\label{proof:lem:NC_PL_Phi_decay_one_iter}
In the proof, we use the quadratic growth property of $\mu$-PL function $f(\bx, \cdot)$ (\cref{lem:quad_growth}), i.e.,
\begin{align}
    \frac{\mu}{2} \norm{\by - \by^*(\bx)}^2 \leq \max_{\by'} f(\bx, \by') - f(\bx, \by), \quad \forall \ \bx,\by \label{eq:quad_growth_PL}
\end{align}
where $\by^*(\bx) \in \argmax_{\by'} f(\bx, \by')$.
See \cite{mahdavi21localSGDA_aistats} for the entire proof.
\end{proof}

\begin{proof}[Proof of Lemma \ref{lem:NC_PL_consensus_error}]
We define the separate synchronization errors for $\bx$ and $\by$
\begin{align*}
    \CExt \triangleq \frac{1}{n} \sumin \mbe \lnr \bxit - \bxt \rnr^2, \qquad \CEyt \triangleq \frac{1}{n} \sumin \mbe \lnr \byit - \byt \rnr^2,
\end{align*}
such that $\CExyt = \CExt + \CEyt$.
We first bound the $\bx$- synchronization error $\CExt$.
Define $s = \lfloor t/\sync \rfloor$, such that $s \sync + 1 \leq t \leq (s+1) \sync - 1$.
Then,
\begin{align}
    \CExt & \triangleq \frac{1}{n} \sumin \mbe \lnr \bxit - \bxt \rnr^2 \nn \\
    &= \frac{1}{n} \sumin \mbe \norm{\Big( \bx^i_{s \sync} - \lrx \sumktm \Gx f_i (\bx_k^i, \by_k^i; \xi^i_k) \Big) - \Big( \bx_{s \sync} - \lrx \frac{1}{n} \sumjn \sumktm \Gx f_j (\bx_k^j, \by_k^j; \xi_k^j) \Big)}^2 \tag{see \eqref{eq:NC_PL_update_avg}} \\
    &= \lrx^2 \frac{1}{n} \sumin \mbe \norm{\sumktm \Gx f_i (\bx_k^i, \by_k^i; \xi^i_k) - \frac{1}{n} \sumjn \sumktm \Gx f_j (\bx_k^j, \by_k^j; \xi_k^j)}^2 \tag{$\because \bx^i_{s \sync} = \bx_{s \sync}, \forall \ i \in [n]$} \\
    & \overset{(a)}{\leq} \lrx^2 \frac{1}{n} (t-s\sync) \sumktm \sumin \mbe \Big\| \Gx f_i (\bx_k^i, \by_k^i; \xi^i_k) - \Gx f_i (\bx_k^i, \by_k^i) + \Gx f_i (\bx_k^i, \by_k^i) - \Gx f_i (\bx_k, \by_k) + \Gx f_i (\bx_k, \by_k) \nn \\
    & \qquad - \Gx f (\bx_k, \by_k) - \frac{1}{n} \sumjn \lp \Gx f_j (\bx_k^j, \by_k^j, \xi_k^j) - \Gx f_j (\bx_k^j, \by_k^j) + \Gx f_j (\bx_k^j, \by_k^j) - \Gx f_j (\bx_k, \by_k) \rp \Big\|^2 \nn \\
    & \overset{(b)}{=} \frac{\lrx^2 (t-s\sync)}{n} \sumktm \sumin \mbe \Bigg[ \norm{\Gx f_i (\bx_k^i, \by_k^i; \xi^i_k) - \Gx f_i (\bx_k^i, \by_k^i)}^2 + \Big\| \frac{1}{n} \sumjn \lp \Gx f_j (\bx_k^j, \by_k^j, \xi_k^j) - \Gx f_j (\bx_k^j, \by_k^j) \rp \Big\|^2 \nn \\
    & + \Big\| \Gx f_i (\bx_k^i, \by_k^i) - \Gx f_i (\bx_k, \by_k) + \Gx f_i (\bx_k, \by_k) - \Gx f (\bx_k, \by_k) - \frac{1}{n} \sumjn \lp \Gx f_j (\bx_k^j, \by_k^j) - \Gx f_j (\bx_k, \by_k) \rp \Big\|^2 \Bigg] \nn \\
    & \overset{(c)}{\leq} \frac{\lrx^2 (\sync-1)}{n} \sumktm \sumin \mbe \Bigg[ \sigma^2 + \frac{\sigma^2}{n} + 3 \norm{\Gx f_i (\bx_k^i, \by_k^i) - \Gx f_i (\bx_k, \by_k)}^2 + 3 \norm{\Gx f_i (\bx_k, \by_k) - \Gx f (\bx_k, \by_k)}^2 \nn \\
    & \qquad \qquad \qquad \qquad \qquad \qquad + 3 \Big\| \frac{1}{n} \sumjn \lp \Gx f_j (\bx_k^j, \by_k^j) - \Gx f_j (\bx_k, \by_k) \rp \Big\|^2 \Bigg] \nn \\
    & \overset{(d)}{\leq} \frac{\lrx^2 (\sync-1)}{n} \sumktm \sumin \mbe \Bigg[ \sigma^2 + \frac{\sigma^2}{n} + 3 \Lf^2 \lb \norm{\bx_k^i - \bx_k}^2 + \norm{\by_k^i - \by_k}^2 \rb + 3 \heterox^2 \nn \\
    & \qquad \qquad \qquad \qquad \qquad \qquad + \frac{3}{n} \sumjn \Lf^2 \lb \norm{\bx_k^j - \bx_k}^2 + \norm{\by_k^j - \by_k}^2 \rb \Bigg] \nn \\
    &= \lrx^2 (\sync-1) \sumktm \lb \sigma^2 \lp 1 + \frac{1}{n} \rp + 3 \heterox^2 + 6 \Lf^2 \lp \CExk + \CEyk \rp \rb, \nn
\end{align}
where $(a)$ follows from \cref{lem:sum_of_squares};
$(b)$ follows from \cref{assum:bdd_var} (unbiasedness of stochastic gradients);
$(c)$ follows from \cref{assum:bdd_var} (bounded variance of stochastic gradients);
$(d)$ follows from \cref{assum:smoothness}, \ref{assum:bdd_hetero}, and Jensen's inequality (\cref{lem:jensens}) for $\| \cdot \|^2$.

Furthermore, $\CExt = 0$ for $t = s \sync$. Therefore,
\begin{align}
    \sum_{t=s\sync}^{(s+1)\sync-1} \CExt = \sum_{t=s\sync+1}^{(s+1)\sync-1} \CExt & \leq \lrx^2 (\sync-1) \sum_{t=s\sync+1}^{(s+1)\sync-1} \sumktm \lb \sigma^2 \lp 1 + \frac{1}{n} \rp + 3 \heterox^2 + 6 \Lf^2 \lp \CExk + \CEyk \rp \rb \nn \\
    & \leq \lrx^2 (\sync-1)^2 \sum_{t=s\sync+1}^{(s+1)\sync-1} \lb \sigma^2 \lp 1 + \frac{1}{n} \rp + 3 \heterox^2 + 6 \Lf^2 \CExyt \rb. \label{eq_proof:lem:NC_PL_x_consensus_error}
\end{align}
The $\by$- synchronization error $\CEyt$ following a similar analysis and we get.
\begin{align}
    \sum_{t=s\sync}^{(s+1)\sync-1} \CEyt & \leq \lry^2 (\sync-1)^2 \sum_{t=s\sync+1}^{(s+1)\sync-1} \lb \sigma^2 \lp 1 + \frac{1}{n} \rp + 3 \heteroy^2 + 6 \Lf^2 \CExyt \rb. \label{eq_proof:lem:NC_PL_y_consensus_error}
\end{align}
Combining \eqref{eq_proof:lem:NC_PL_x_consensus_error} and \eqref{eq_proof:lem:NC_PL_y_consensus_error}, we get
\begin{align}
    \sum_{t=s\sync}^{(s+1)\sync-1} \CExyt & \leq (\sync-1)^2 \lb \sync \lp \lrx^2 + \lry^2 \rp \sigma^2 \lp 1 + \frac{1}{n} \rp + 3 \sync \lp \lrx^2 \heterox^2 + \lry^2 \heteroy^2 \rp + 6 \Lf^2 \lp \lrx^2 + \lry^2 \rp  \sum_{t=s\sync+1}^{(s+1)\sync-1} \CExyt \rb. \nn
\end{align}
Using our choice of $\lrx, \lry$, we have $6 \Lf^2 \lp \lrx^2 + \lry^2 \rp (\sync - 1)^2 \leq 1/2$, then
\begin{align}
    \sum_{t=s\sync}^{(s+1)\sync-1} \CExyt & \leq 2 \sync (\sync-1)^2 \lb \lp \lrx^2 + \lry^2 \rp \sigma^2 \lp 1 + \frac{1}{n} \rp + 3 \lp \lrx^2 \heterox^2 + \lry^2 \heteroy^2 \rp \rb \nn \\
    \Rightarrow \frac{1}{T} \sum_{s=0}^{T/\sync - 1} \sum_{t=s\sync}^{(s+1)\sync-1} \CExyt = \frac{1}{T} \sumtT \CExyt & \leq 2 (\sync-1)^2 \lb \lp \lrx^2 + \lry^2 \rp \sigma^2 \lp 1 + \frac{1}{n} \rp + 3 \lp \lrx^2 \heterox^2 + \lry^2 \heteroy^2 \rp \rb. \nn
\end{align}
\end{proof}

\begin{proof}[Proof of Lemma \ref{lem:NC_PL_phi_error}]
Using $\Lf$-smoothness of $f(\bx, \cdot)$,
\begin{align}
    f(\bxtp, \byt) &+ \lan \Gy f(\bxtp, \byt), \bytp - \byt \ran - \frac{\Lf}{2} \norm{\bytp - \byt}^2 \leq f(\bxtp, \bytp) \nn \\
    \Rightarrow f(\bxtp, \byt) & \leq f(\bxtp, \bytp) - \lry \lan \Gy f(\bxtp, \byt), \frac{1}{n} \sumin \Gy f_i (\bxit, \byit; \xiit) \ran + \frac{\lry^2 \Lf}{2} \norm{\frac{1}{n} \sumin \Gy f_i (\bxit, \byit; \xiit)}^2 \tag{using \eqref{eq:NC_PL_update_avg}} \\
    \Rightarrow \mbe f(\bxtp, \byt) & \leq \mbe f(\bxtp, \bytp) - \lry \mbe \lan \Gy f(\bxtp, \byt), \frac{1}{n} \sumin \Gy f_i (\bxit, \byit) \ran \nn \\
    & \qquad + \frac{\lry^2 \Lf}{2} \lb \frac{\sigma^2}{n} + \norm{\frac{1}{n} \sumin \Gy f_i (\bxit, \byit)}^2 \rb \tag{\cref{assum:bdd_var}} \\
    &= \mbe f(\bxtp, \bytp) - \frac{\lry}{2} \mbe \norm{\Gy f(\bxtp, \byt)}^2 - \frac{\lry}{2} \lp 1 - \lry \Lf \rp \mbe \norm{\frac{1}{n} \sumin \Gy f_i (\bxit, \byit)}^2 \nn \\
    & \qquad + \frac{\lry}{2} \mbe \norm{\Gy f(\bxtp, \byt) - \Gy f(\bxt, \byt) + \Gy f(\bxt, \byt) - \frac{1}{n} \sumin \Gy f_i (\bxit, \byit)}^2 + \frac{\lry^2 \Lf \sigma^2}{2n} \nn \\
    & \leq \mbe f(\bxtp, \bytp) - \frac{\lry}{2} \mbe \norm{\Gy f(\bxtp, \byt)}^2 - \frac{\lry}{2} \lp 1 - \lry \Lf \rp \mbe \norm{\frac{1}{n} \sumin \Gy f_i (\bxit, \byit)}^2 \nn \\
    & \qquad + \lry \Lf^2 \mbe \norm{\bxtp - \bxt}^2 + \lry \Lf^2 \CExyt + \frac{\lry^2 \Lf \sigma^2}{2n}, \label{eq_proof:lem:NC_PL_phi_error_1}
\end{align}
where \eqref{eq_proof:lem:NC_PL_phi_error_1} follows from Jensen's inequality (\cref{lem:jensens}) for $\norm{\cdot}^2$, \cref{assum:smoothness} and Young's inequality (\cref{lem:Young}) for $\gamma = 1$, $\lan \mbf a, \bf b \ran \leq \frac{1}{2} \norm{\mbf a}^2 + \frac{1}{2} \norm{\mbf b}^2$.
Next, note that using \cref{assum:bdd_var}
\begin{align}
    \mbe \norm{\bxtp - \bxt}^2 &= \lrx^2 \mbe \norm{\frac{1}{n} \sumin \Gx f_i (\bxit, \byit; \xiit)}^2 \leq \lrx^2 \norm{\frac{1}{n} \sumin \Gx f_i (\bxit, \byit)}^2 + \frac{\lrx^2 \sigma^2}{n}. \label{eq_proof:lem:NC_PL_phi_error_2a}
\end{align}
Also, using \cref{assum:PL_y},
\begin{align}
    \norm{\Gy f(\bxtp, \byt)}^2 \geq 2 \mu \lp \max_\by f(\bxtp, \by) - f(\bxtp, \byt) \rp = 2 \mu \lp \Phi (\bxtp) - f(\bxtp, \byt) \rp. \label{eq_proof:lem:NC_PL_phi_error_2b}
\end{align}
Substituting \eqref{eq_proof:lem:NC_PL_phi_error_2a}, \eqref{eq_proof:lem:NC_PL_phi_error_2b} in \eqref{eq_proof:lem:NC_PL_phi_error_1}, and rearranging the terms, we get
\begin{align}
    & \lry \mu \mbe \lp \Phi (\bxtp) - f(\bxtp, \byt) \rp \nn \\
    & \leq \mbe f(\bxtp, \bytp) - \mbe f(\bxtp, \byt) - \frac{\lry}{2} \lp 1 - \lry \Lf \rp \mbe \norm{\frac{1}{n} \sumin \Gy f_i (\bxit, \byit)}^2 + \frac{\lry^2 \Lf \sigma^2}{2n} \nn \\
    & \quad + \lry \Lf^2 \lb \lrx^2 \mbe \norm{\frac{1}{n} \sumin \Gx f_i (\bxit, \byit)}^2 + \frac{\lrx^2 \sigma^2}{n} \rb + \lry \Lf^2 \CExyt \nn \\
    \Rightarrow & \mbe \lp \Phi (\bxtp) - f(\bxtp, \bytp) \rp \nn \\
    & \leq (1 - \lry \mu) \mbe \lp \Phi (\bxtp) - f(\bxtp, \byt) \rp - \frac{\lry}{2} \lp 1 - \lry \Lf \rp \mbe \norm{\frac{1}{n} \sumin \Gy f_i (\bxit, \byit)}^2 + \frac{\lry^2 \Lf \sigma^2}{2n} \nn \\
    & \quad + \lry \Lf^2 \lb \lrx^2 \mbe \norm{\frac{1}{n} \sumin \Gx f_i (\bxit, \byit)}^2 + \frac{\lrx^2 \sigma^2}{n} \rb + \lry \Lf^2 \CExyt. \label{eq_proof:lem:NC_PL_phi_error_3}
\end{align}
Next, we bound $\mbe \lp \Phi (\bxtp) - f(\bxtp, \byt) \rp$.
\begin{align}
    & \mbe \lb \Phi (\bxtp) - f(\bxtp, \byt) \rb \nn \\
    &= \underbrace{\mbe \lb \Phi (\bxtp) - \Phi (\bxt) \rb}_{I_1} + \mbe \lb \Phi (\bxt) - f(\bxt, \byt) \rb + \underbrace{\mbe \lb f(\bxt, \byt) - f(\bxtp, \byt) \rb}_{I_2}
\end{align}
$I_1$ is bounded in \cref{lem:NC_PL_Phi_decay_one_iter}.
We next bound $I_2$. Using $\Lf$-smoothness of $f(\cdot, \byt)$,
\begin{align}
    & f(\bxt, \byt) + \lan \Gx f(\bxt, \byt), \bxtp - \bxt \ran - \frac{\Lf}{2} \norm{\bxtp - \bxt}^2 \leq f(\bxtp, \byt) \nn \\
    \Rightarrow I_2 &= \mbe \lb f(\bxt, \byt) - f(\bxtp, \byt) \rb \nn \\
    & \leq \lrx \mbe \lan \Gx f(\bxt, \byt), \frac{1}{n} \sumin \Gx f_i (\bxit, \byit; \xiit) \ran + \frac{\lrx^2 \Lf}{2} \mbe \norm{\frac{1}{n} \sumin \Gx f_i (\bxit, \byit; \xiit)}^2 \nn \\
    & \leq \lrx \mbe \lan \Gx f(\bxt, \byt), \frac{1}{n} \sumin \Gx f_i (\bxit, \byit) \ran + \frac{\lrx^2 \Lf}{2} \lb \frac{\sigma^2}{n} + \mbe \norm{\frac{1}{n} \sumin \Gx f_i (\bxit, \byit)}^2 \rb \tag{\cref{assum:bdd_var}} \\
    & \leq \frac{\lrx}{2} \mbe \lb \norm{\Gx f(\bxt, \byt)}^2 + \norm{\frac{1}{n} \sumin \Gx f_i (\bxit, \byit)}^2 \rb + \frac{\lrx^2 \Lf}{2} \lb \frac{\sigma^2}{n} + \mbe \norm{\frac{1}{n} \sumin \Gx f_i (\bxit, \byit)}^2 \rb \nn \\
    & \leq \lrx \mbe \lb \norm{\G \Phi(\bxt)}^2 + \norm{\Gx f(\bxt, \byt) - \G \Phi(\bxt)}^2 \rb + \frac{\lrx^2 \Lf \sigma^2}{2 n} + \frac{\lrx}{2} \lp 1 + \lrx \Lf \rp \mbe \norm{\frac{1}{n} \sumin \Gx f_i (\bxit, \byit)}^2 \nn \\
    & \overset{(a)}{\leq} \lrx \mbe \norm{\G \Phi(\bxt)}^2 + \lrx \Lf^2 \mbe \norm{\byt - \by^*(\bxt)}^2 + \frac{\lrx^2 \Lf \sigma^2}{2 n} + \frac{\lrx}{2} \lp 1 + \lrx \Lf \rp \mbe \norm{\frac{1}{n} \sumin \Gx f_i (\bxit, \byit)}^2 \nn \\
    & \leq \lrx \mbe \norm{\G \Phi(\bxt)}^2 + \frac{2 \lrx \Lf^2}{\mu} \mbe \lb \Phi(\bxt) - f(\bxt, \byt) \rb + \frac{\lrx^2 \Lf \sigma^2}{2 n} + \frac{\lrx}{2} \lp 1 + \lrx \Lf \rp \mbe \norm{\frac{1}{n} \sumin \Gx f_i (\bxit, \byit)}^2. \label{eq_proof:lem:NC_PL_phi_error_4}
\end{align}
where $(a)$ follows from \cref{assum:smoothness} and \cref{lem:Phi_smooth_nouiehed}. 
Also, recall that $\by^*(\bx) \in \argmax_{\by'} f(\bx, \by')$.
\eqref{eq_proof:lem:NC_PL_phi_error_4} follows from the quadratic growth property of $\mu$-PL functions (\cref{lem:quad_growth}).
Substituting the bounds on $I_1, I_2$ from \cref{lem:NC_PL_Phi_decay_one_iter} and \eqref{eq_proof:lem:NC_PL_phi_error_4} respectively, in \eqref{eq_proof:lem:NC_PL_phi_error_3}, we get
\begin{align}
    & \mbe \lp \Phi (\bxtp) - f(\bxtp, \bytp) \rp \nn \\
    & \leq (1 - \lry \mu) \lp 1 + \frac{4 \lrx \Lf^2}{\mu} \rp \mbe \lp \Phi (\bxt) - f(\bxt, \byt) \rp \nn \\
    & \quad + (1 - \lry \mu) \lb - \frac{\lrx}{2} \mbe \norm{\G \Phi (\bxt)}^2 - \frac{\lrx}{2} \lp 1 - \Lp \lrx \rp \mbe \norm{\frac{1}{n} \sumin \Gx f_i(\bxit, \byit)}^2 + 2 \lrx \Lf^2 \CExyt + \frac{\Lp \lrx^2 \sigma^2}{2 n} \rb \nn \\
    & \quad + (1 - \lry \mu) \lb \lrx \mbe \norm{\G \Phi(\bxt)}^2 + \frac{\lrx^2 \Lf \sigma^2}{2 n} + \frac{\lrx}{2} \lp 1 + \lrx \Lf \rp \mbe \norm{\frac{1}{n} \sumin \Gx f_i (\bxit, \byit)}^2 \rb \nn \\
    & \quad - \frac{\lry}{2} \lp 1 - \lry \Lf \rp \mbe \norm{\frac{1}{n} \sumin \Gy f_i (\bxit, \byit)}^2 + \frac{\lry^2 \Lf \sigma^2}{2n} \nn \\
    & \quad + \lry \Lf^2 \lb \lrx^2 \norm{\frac{1}{n} \sumin \Gx f_i (\bxit, \byit)}^2 + \frac{\lrx^2 \sigma^2}{n} \rb + \lry \Lf^2 \CExyt \nn \\
    & \leq \lp 1 - \frac{\lry \mu}{2} \rp \mbe \lp \Phi (\bxt) - f(\bxt, \byt) \rp + \frac{\lry^2 \Lf \sigma^2}{2n} + \frac{\lry \Lf^2 \lrx^2 \sigma^2}{n} + \lry \Lf^2 \CExyt \nn \\
    & \quad + \lb (1 - \lry \mu) \frac{\lrx^2}{2} \lp \Lf + \Lp \rp + \lry \Lf^2 \lrx^2 \rb \mbe \norm{\frac{1}{n} \sumin \Gx f_i(\bxit, \byit)}^2 \nn \\
    & \quad + (1 - \lry \mu) \lb \frac{\lrx}{2} \mbe \norm{\G \Phi(\bxt)}^2 + \frac{\lrx^2 \Lf \sigma^2}{2 n} + 2 \lrx \Lf^2 \CExyt + \frac{\Lp \lrx^2 \sigma^2}{2 n} \rb, \label{eq_proof:lem:NC_PL_phi_error_5}
\end{align}
where we choose $\lrx$ such that $(1 - \lry \mu) \lp 1 + \frac{4 \lrx \Lf^2}{\mu} \rp \leq \lp 1 - \frac{\lry \mu}{2} \rp$. 
This holds if $\frac{4 \lrx \Lf^2}{\mu} \leq \frac{\lry \mu}{2} \Rightarrow \lrx \leq \frac{\lry}{8 \kappa^2}$.
Summing \eqref{eq_proof:lem:NC_PL_phi_error_5} over $t=0, \hdots, T-1$, and rearranging the terms, we get
\begin{align}
    & \frac{1}{T} \sumtT \mbe \lp \Phi (\bxtp) - f(\bxtp, \bytp) \rp \nn \\
    & \leq \lp 1 - \frac{\lry \mu}{2} \rp \frac{1}{T} \sumtT \mbe \lp \Phi (\bxt) - f(\bxt, \byt) \rp + \Lf^2 \lp 2 \lrx (1 - \lry \mu) + \lry \rp \frac{1}{T} \sumtT \CExyt \nn \\
    & \quad + \lb (1 - \lry \mu) \frac{\lrx^2}{2} \lp \Lf + \Lp \rp + \lry \Lf^2 \lrx^2 \rb \frac{1}{T} \sumtT \mbe \norm{\frac{1}{n} \sumin \Gx f_i(\bxit, \byit)}^2 + (1 - \lry \mu) \frac{\lrx}{2} \frac{1}{T} \sumtT \mbe \norm{\G \Phi(\bxt)}^2 \nn \\
    & \quad + \frac{\lry^2 \Lf \sigma^2}{2n} + \frac{\lry \Lf^2 \lrx^2 \sigma^2}{n} + (1 - \lry \mu) \lb \frac{\lrx^2 \Lf \sigma^2}{2 n} + \frac{\Lp \lrx^2 \sigma^2}{2 n} \rb. \nn
\end{align}
Rearranging the terms, we get
\begin{align}
    & \frac{1}{T} \sumtT \mbe \lp \Phi (\bxt) - f(\bxt, \byt) \rp \nn \\
    & \leq \frac{2}{\lry \mu} \lb \frac{\Phi (\bx_0) - f(\bx_0, \by_0)}{ T} - \frac{\mbe \lp \Phi (\bx_T) - f(\bx_T, \by_T) \rp}{ T} \rb + \frac{2 \Lf^2}{\mu \lry} \lp 2 \lrx (1 - \lry \mu) + \lry \rp \frac{1}{T} \sumtT \CExyt \nn \\
    & \quad + \lb (1 - \lry \mu) \frac{\lrx^2}{2} \lp \Lf + \Lp \rp + \lry \Lf^2 \lrx^2 \rb \frac{2}{\lry \mu T} \sumtT \mbe \norm{\frac{1}{n} \sumin \Gx f_i(\bxit, \byit)}^2 + (1 - \lry \mu) \frac{\lrx}{\lry \mu T} \sumtT \mbe \norm{\G \Phi(\bxt)}^2 \nn \\
    & \quad + \frac{\lry \Lf \sigma^2}{\mu n} + \frac{2 \Lf^2 \lrx^2 \sigma^2}{\mu n} + \frac{(1 - \lry \mu)}{\mu \lry} \lb \frac{\lrx^2 \Lf \sigma^2}{n} + \frac{\Lp \lrx^2 \sigma^2}{n} \rb \nn \\
    & \leq \frac{2 \lp \Phi (\bx_0) - f(\bx_0, \by_0) \rp}{\lry \mu T} + \frac{2 \Lf^2}{\mu \lry} \lp 2 \lrx (1 - \lry \mu) + \lry \rp \frac{1}{T} \sumtT \CExyt \tag{$\because \Phi (\bx_T) \triangleq \argmax_\by f(\bx_T, \by)$} \\
    & \quad + \lb (1 - \lry \mu) \frac{\lrx^2}{2} \lp \Lf + \Lp \rp + \lry \Lf^2 \lrx^2 \rb \frac{2}{\lry \mu T} \sumtT \mbe \norm{\frac{1}{n} \sumin \Gx f_i(\bxit, \byit)}^2 + (1 - \lry \mu) \frac{\lrx}{\lry \mu T} \sumtT \mbe \norm{\G \Phi(\bxt)}^2 \nn \\
    & \quad + \frac{\lry \Lf \sigma^2}{\mu n} + \frac{2 \Lf^2 \lrx^2 \sigma^2}{\mu n} + \frac{(1 - \lry \mu)}{\mu \lry} \lb \frac{\lrx^2 \Lf \sigma^2}{n} + \frac{\Lp \lrx^2 \sigma^2}{n} \rb, \nn
\end{align}
which concludes the proof.
\end{proof}

\newpage
\section{Nonconvex-PL (NC-PL) Functions: Momentum Local SGDA (\texorpdfstring{\cref{thm:NC_PL_mom}}{Theorem 2})} \label{app:NC_PL_mom}
In this section we prove the convergence of \cref{alg_NC_momentum} for Nonconvex-PL functions, and provide the complexity and
communication guarantees.

We organize this section as follows. First, in \cref{sec:NC_PL_mom_int_results} we present some intermediate results. 
Next, in \cref{sec:NC_PL_mom_thm_proof}, we present the proof of \cref{thm:NC_PL_mom}, which is followed by the proofs of the intermediate results in \cref{sec:NC_PL_mom_int_results_proofs}.

Again, the problem we solve is
\begin{align*}
    \min_{\bx} \max_{\by} \lcb f(\bx, \by) \triangleq \frac{1}{n} \sumin f_i(\bx, \by) \rcb.
\end{align*}
We define
\begin{align}
    \Phi (\bx) \triangleq \max_{\by} f(\bx, \by) \quad \text{and} \quad \by^* (\bx) \in \argmax_{\by} f(\bx, \by). \label{eq:Phi_defn}
\end{align}
Since $f(\bx, \cdot)$ is $\mu$-PL (\cref{assum:PL_y}), $\by^*(\bx)$ is not necessarily unique.

For the sake of analysis, we define \textit{virtual} sequences of average iterates and average direction estimates:
\begin{align*}
    & \bxt \triangleq \frac{1}{n} \sumin \bxit; \quad \byt \triangleq \frac{1}{n} \sumin \byit; \\
    & \Tbxtp \triangleq \frac{1}{n} \sumin \Tbxitp; \quad \Tbytp \triangleq \frac{1}{n} \sumin \Tbyitp; \\
    & \bdxt \triangleq \frac{1}{n} \sumin \bdxit; \quad \bdyt \triangleq \frac{1}{n} \sumin \bdyit.
\end{align*}
Note that these sequences are constructed only for the sake of analysis. During an actual run of the algorithm, these sequences exist only at the time instants when the clients communicate with the server.
We next write the update expressions for these virtual sequences, using the updates in Algorithm \ref{alg_NC_momentum}.
\begin{equation}
    \begin{aligned}
        & \Tbxtp = \bxt - \lrx \bdxt, \qquad \bxtp = \bxt + \cvxt \lp \Tbxtp - \bxt \rp \\
        & \Tbytp = \byt + \lry\bdyt, \qquad \bytp = \byt + \cvxt \lp \Tbytp - \byt \rp \\
        & \bdxtp = (1 - \momx \cvxt) \bdxt + \momx \cvxt \frac{1}{n} \sumin \Gx f_i (\bxitp, \byitp; \xiitp) \\
        & \bdytp = (1 - \momy \cvxt) \bdyt + \momy \cvxt \frac{1}{n} \sumin \Gy f_i (\bxitp, \byitp; \xiitp).
    \end{aligned}
    \label{eq:NC_mom_update_avg}
\end{equation}
Next, we present some intermediate results which we use in the proof of \cref{thm:NC_PL_mom}. To make the proof concise, the proofs of these intermediate results is relegated to \cref{sec:NC_PL_mom_int_results_proofs}.

\subsection{Intermediate Lemmas} \label{sec:NC_PL_mom_int_results}

We use the following result from \cite{nouiehed19minimax_neurips19} about the smoothness of $\Phi(\cdot)$.

\begin{lemma}
\label{lem:Phi_PL_smooth_nouiehed}
If the function $f(\bx, \cdot)$ satisfies Assumptions \ref{assum:smoothness}, \ref{assum:PL_y} ($\Lf$-smoothness and $\mu$-PL condition in $\by$), then $\Phi (\bx)$ is $\Lp$-smooth with $\Lp = \kappa \Lf/2 + \Lf$, where $\kappa = \Lf/\mu$, and 
$$\G \Phi(\cdot) = \Gx f(\cdot, \by^*(\cdot)),$$
where $\by^*(\cdot) \in \argmax_\by f(\cdot, \by)$.
\end{lemma}

\begin{lemma}
\label{lem:NC_PL_mom_Phi_1_step_decay}
Suppose the loss function $f$ satisfies Assumptions \ref{assum:smoothness}, \ref{assum:PL_y}, and the step-size $\lrx$, and $\cvxt$ satisfy $0 < \cvxt \lrx \leq \frac{\mu}{4 \Lf^2}$.
Then the iterates generated by Algorithm \ref{alg_NC_momentum} satisfy
\begin{align}
    \Phi (\bxtp) - \Phi (\bxt) & \leq - \frac{\cvxt}{2 \lrx} \lnr \Tbxtp - \bxt \rnr^2 + \frac{4 \lrx \cvxt \Lf^2}{\mu} \lb \Phi(\bxt) - f(\bxt, \byt) \rb + 2 \lrx \cvxt \lnr \Gx f(\bxt, \byt) - \bdxt \rnr^2, \nn
\end{align}
where $\Phi (\cdot)$ is defined in \eqref{eq:Phi_defn}.
\end{lemma}

Next, we bound the difference $\Phi(\bxt) - f(\bxt, \byt)$.

\begin{lemma}
\label{lem:NC_PL_mom_phi_error}
Suppose the loss function $f$ satisfies Assumptions \ref{assum:smoothness}, \ref{assum:PL_y}, and the step-sizes $\lrx, \lry$, and $\cvxt$ satisfy $0 < \cvxt \lry \leq \frac{1}{2 \Lf}$, $0 < \cvxt \lrx \leq \frac{\mu}{8 \Lf^2}$, and $\lrx \leq \frac{\lry}{8 \kappa^2}$.
Then the iterates generated by Algorithm \ref{alg_NC_momentum} satisfy
\begin{align}
    \Phi (\bxtp) - f(\bxtp, \bytp) & \leq \lp 1 - \frac{\cvxt \lry \mu}{2} \rp \lb \Phi (\bxt) - f(\bxt, \byt) \rb - \frac{\cvxt}{4 \lry} \norm{\Tbytp - \byt}^2 \nn \\
    & \quad + \frac{\cvxt}{2 \lrx} \lnr \Tbxtp - \bxt \rnr^2 + \cvxt \lry \norm{\Gy f(\bxt, \byt) - \bdyt}^2. \nn
\end{align}
\end{lemma}

The next result bounds the variance in the average direction estimates $\bdxt, \bdyt$ \eqref{eq:NC_mom_update_avg} w.r.t. the partial gradients of the global loss function $\Gx f(\bxt, \byt), \Gy f(\bxt, \byt)$, respectively.

\begin{lemma}
\label{lem:NC_PL_mom_grad_var_bound}
Suppose the local loss functions $\{ f_i \}$ satisfy \cref{assum:smoothness}, and the stochastic oracles for the local functions $\{ f_i \}$ satisfy \cref{assum:bdd_var}.
Further, in \cref{alg_NC_momentum}, we choose $\momx = \momy = \mom$, and $\cvxt$ such that $0 < \cvxt < 1/\mom$.
Then the following holds.
\begin{equation}
    \begin{aligned}
        & \mbe \lnr \Gx f(\bxtp, \bytp) - \bdxtp \rnr^2 \leq \lp 1 - \frac{\mom \cvxt}{2} \rp \mbe \lnr \Gx f(\bxt, \byt) - \bdxt \rnr^2 + \frac{\mom^2 \cvxt^2 \sigma^2}{n} \\
        & \quad + \frac{2 \Lf^2 \cvxt}{\mom} \mbe \lp \lnr \Tbxtp - \bxt \rnr^2 + \lnr \Tbytp - \byt \rnr^2 \rp + \mom \cvxt \frac{1}{n} \sumin \Lf^2 \mbe \lp \lnr \bxitp - \bxtp \rnr^2 + \lnr \byitp - \bytp \rnr^2 \rp,
    \end{aligned}
    \label{eq:lem:NC_PL_mom_grad_var_bound_x}
\end{equation}
\begin{equation}
    \begin{aligned}
        & \mbe \lnr \Gy f(\bxtp, \bytp) - \bdytp \rnr^2 \leq \lp 1 - \frac{\mom \cvxt}{2} \rp \mbe \lnr \Gy f(\bxt, \byt) - \bdyt \rnr^2 + \frac{\mom^2 \cvxt^2 \sigma^2}{n} \\
        & \quad + \frac{2 \Lf^2 \cvxt}{\mom} \mbe \lp \lnr \Tbxtp - \bxt \rnr^2 + \lnr \Tbytp - \byt \rnr^2 \rp + \mom \cvxt \frac{1}{n} \sumin \Lf^2 \mbe \lp \lnr \bxitp - \bxtp \rnr^2 + \lnr \byitp - \bytp \rnr^2 \rp.
    \end{aligned}
    \label{eq:lem:NC_PL_mom_grad_var_bound_y}
\end{equation}
\end{lemma}
Notice that the bound depends on the disagreement of the individual iterates with the \textit{virtual} global average: $\mbe \lnr \bxitp - \bxtp \rnr^2$, $\mbe \lnr \byitp - \bytp \rnr^2$, which is nonzero since $\sync > 1$, and the clients carry out multiple local updates between successive rounds of communication with the server.
Next, we bound these synchronization errors.
Henceforth, for the sake of brevity, we use the following notations:
\begin{align*}
    \CExyt & \triangleq \frac{1}{n} \sumin \mbe \lp \lnr \bxit - \bxt \rnr^2 + \lnr \byit - \byt \rnr^2 \rp, \\
    \CEpt & \triangleq \frac{1}{n} \sumin \mbe \lnr \bdxit - \bdxt \rnr^2, \\
    \CEqt & \triangleq \frac{1}{n} \sumin \mbe \lnr \bdyit - \bdyt \rnr^2.
\end{align*}

\begin{lemma}
\label{lem:NC_PL_mom_cons_errs_recursion}
Suppose the local loss functions $\{ f_i \}$ satisfy Assumptions \ref{assum:smoothness},  \ref{assum:bdd_hetero}, and the stochastic oracles for the local functions $\{ f_i \}$ satisfy \cref{assum:bdd_var}.
Further, in \cref{alg_NC_momentum}, we choose $\momx = \momy = \mom$, and $\cvxt$ such that $0 < \cvxt < 1/\mom$.
Then, the iterates $\{ \bxit, \byit \}$ and direction estimates $\{ \bdxit, \bdyit \}$ generated by Algorithm \ref{alg_NC_momentum} satisfy
\begin{align}
    \CExytp & \leq (1+c_1) \CExyt + \lp 1 + \mfrac{1}{c_1} \rp \cvxt^2 \lp \lrx^2 \CEpt + \lry^2 \CEqt \rp, \qquad \text{ for any constant } c_1 > 0
    \label{eq:lem:NC_PL_mom_xy_cons_errs_recursion} 
    \\
    \CEptp & \leq (1-\mom \cvxt) \CEpt + 6 \Lf^2 \mom \cvxt \CExytp + \mom \cvxt \lb \sigma^2 \lp 1 + \mfrac{1}{n} \rp + 3 \heterox^2 \rb,
    \label{eq:lem:NC_PL_mom_p_cons_errs_recursion}
    \\
    \CEqtp & \leq (1-\mom \cvxt) \CEqt + 6 \Lf^2 \mom \cvxt \CExytp + \mom \cvxt \lb \sigma^2 \lp 1 + \mfrac{1}{n} \rp + 3 \heteroy^2 \rb.
    \label{eq:lem:NC_PL_mom_q_cons_errs_recursion}
\end{align}
\end{lemma}

\begin{lemma}
\label{lem:NC_PL_mom_induct_bd_cons_error_xy}
Suppose the local loss functions $\{ f_i \}$ satisfy Assumptions \ref{assum:smoothness},  \ref{assum:bdd_hetero}, and the stochastic oracles for the local functions $\{ f_i \}$ satisfy \cref{assum:bdd_var}.
Further, in \cref{alg_NC_momentum}, we choose $\momx = \momy = \mom$, and step-sizes $\lrx, \lry, \cvxt$ such that 
$\cvxt \equiv \cvx \leq \min \lcb \frac{\mom}{6 \Lf^2 (\lry^2 + \lrx^2)}, \frac{1}{16 \mom \sync} \rcb$ for all $t$, and $\Lf^2 (\lry^2 + \lrx^2) \leq \frac{\mom^2}{6}$.
Suppose $s \sync + 1 \leq t \leq (s+1) \sync -1$ for some positive integer $s$ (i.e., $t$ is between two consecutive synchronizations).
Also, let $1 \leq k < \sync$ such that $t - k \geq s \sync + 1$.
Then, the consensus error satisfies
\begin{align}
    \CExyt \leq (1 + 2 k \theta) \Delta_{t-k}^{x,y} + 2 k \mfrac{\cvx}{\mom} (1-\mom \cvx) \lp \lrx^2 \Delta_{t-k-1}^{\bdx} + \lry^2 \Delta_{t-k-1}^{\bdy} \rp + k^2 (1+\theta) \Upsilon,
    \label{eq:lem:NC_PL_mom_induct_bd_cons_error_xy}
\end{align}
where, $\theta = c_1 + 6 \Lf^2 \cvx^2 (\lry^2 + \lrx^2)$, $c_1 = \frac{\mom \cvx}{1 - \mom \cvx}$, and $\Upsilon = \cvx^2 \lb \lp \lrx^2 + \lry^2 \rp \sigma^2 \lp 1 + \mfrac{1}{n} \rp + 3 \lrx^2 \heterox^2 + 3 \lry^2 \heteroy^2 \rb$.
\end{lemma}

\begin{cor}
\label{cor:NC_PL_mom_induct_bd_cons_error_xy}
Since the clients in Algorithm \ref{alg_NC_momentum} communicate with the server every $\sync$ iterations, for all $t = 0, \hdots, T-1$, then under the conditions of \cref{lem:NC_PL_mom_induct_bd_cons_error_xy}, the iterate consensus error is bounded as follows.
\begin{align*}
    \CExyt \leq \Theta \lp (\sync - 1)^2 \cvx^2 \lp \lp \lrx^2 + \lry^2 \rp \sigma^2 + \lrx^2 \heterox^2 + \lry^2 \heteroy^2 \rp \rp.
\end{align*}
\end{cor}


\subsection{Proof of \texorpdfstring{\cref{thm:NC_PL_mom}}{Theorem 2}}
\label{sec:NC_PL_mom_thm_proof}

For the sake of completeness, we first state the full statement of \cref{thm:NC_PL_mom}, in a slightly more general form.

\begin{theorem*}
Suppose the local loss functions $\{ f_i \}_i$ satisfy Assumptions \ref{assum:smoothness}, \ref{assum:bdd_var}, \ref{assum:bdd_hetero}, and the global function $f$ satisfies \cref{assum:PL_y}.
Suppose in \cref{alg_NC_momentum}, 
$\momx = \momy = \mom = 3$, $\cvxt \equiv \cvx \leq \min \big\{ \frac{\mom}{6 \Lf^2 (\lry^2 + \lrx^2)}, \frac{1}{48 \sync} \big\}$, for all $t$, and the step-sizes $\lrx, \lry$ are chosen such that $\lry \leq \frac{\mu}{8 \Lf^2}$, and $\frac{\lrx}{\lry} \leq \frac{1}{20 \kappa^2}$, where $\kappa = \Lf/\mu$ is the condition number.
Then the iterates generated by \cref{alg_NC_momentum} satisfy
\begin{equation}
    \begin{aligned}
        & \frac{1}{T} \sumtT \lb \frac{1}{\lrx^2} \mbe \lnr \Tbxtp - \bxt \rnr^2 + \frac{2 \Lf^2}{\mu} \mbe \lb \Phi (\bxt) - f(\bxt, \byt) \rb + \mbe \lnr \Gx f(\bxt, \byt) - \bdxt \rnr^2 \rb \\
        & \leq \underbrace{\mco \lp \frac{\kappa^2}{\lry \cvx T} + \frac{\cvx}{\mu \lry} \frac{\sigma^2}{n} \rp}_{\text{Error with full synchronization}} + \underbrace{\mco \Big( (\sync - 1)^2 \cvx^2 \lp \sigma^2 + \heterox^2 + \heteroy^2 \rp \Big)}_{\text{Error due to local updates}}.
    \end{aligned}
    \label{eq_proof:thm_NC_PL_mom_conv_rate}
\end{equation}
Recall that $\sigma^2$ is the variance of stochastic gradient oracle (\cref{assum:bdd_var}), and $\heterox, \heteroy$ quantify the heterogeneity of local functions (\cref{assum:bdd_hetero}).
With $\cvx = \sqrt{\frac{n}{T}}$ in \eqref{eq_proof:thm_NC_PL_mom_conv_rate}, we get
\begin{align}
    & \frac{1}{T} \sumtT \mbe \lb \frac{1}{\lrx^2} \lnr \Tbxtp - \bxt \rnr^2 + \frac{\Lf^2}{\mu} \lb \Phi (\bxt) - f(\bxt, \byt) \rb + \lnr \Gx f(\bxt, \byt) - \bdxt \rnr^2 \rb \nn \\
    & \qquad \leq \mco \lp \frac{\kappa^2 + \sigma^2}{\sqrt{nT}} \rp + \mco \lp \frac{n (\sync-1)^2 \lp \sigma^2 + \heterox^2 + \heteroy^2 \rp}{T} \rp. \nn
\end{align}
\end{theorem*}

\begin{remark}[Convergence results in terms of $\norm{\Phi (\cdot)}$]
The inequality \eqref{eq:thm:NC_PL_mom} results from the following reasoning.
\begin{align}
    \norm{\G \Phi (\bxt)} &= \norm{\Gx f(\bxt, \by^*(\bxt))} \tag{\cref{lem:Phi_PL_smooth_nouiehed}} \\
    & \leq \norm{\Gx f(\bxt, \by^*(\bxt)) - \Gx f(\bxt, \byt)} + \norm{\Gx f(\bxt, \byt)} \tag{Triangle inequality} \\
    & \leq \Lf \norm{\by^*(\bxt) - \byt} + \norm{\Gx f(\bxt, \byt) - \bdxt} + \norm{\bdxt} \tag{\cref{assum:smoothness}} \\
    &= \Lf \sqrt{\frac{2}{\mu} \lb \Phi (\bxt) - f(\bxt, \byt) \rb} + \norm{\Gx f(\bxt, \byt) - \bdxt} + \frac{1}{\lrx} \norm{\Tbxtp - \bxt}. \tag{quadratic growth of $\mu$-PL functions (\cref{lem:quad_growth})}
    \label{eq:NC_PL_mom_compare_metrics}
    \\
    \Rightarrow \frac{1}{T} \sumtT \mbe \norm{\G \Phi(\bxt)}^2 & \leq \frac{3}{T} \sumtT \mbe \lp \frac{1}{\lrx^2} \lnr \Tbxtp - \bxt \rnr^2 + \frac{2 \Lf^2}{\mu} \lb \Phi (\bxt) - f(\bxt, \byt) \rb + \lnr \Gx f(\bxt, \byt) - \bdxt \rnr^2 \rp. \nn
\end{align}
\end{remark}

\begin{proof}[Proof of \cref{thm:NC_PL_mom}]

Multiplying both sides of \cref{lem:NC_PL_mom_phi_error} by $10 \Lf^2 \lrx / (\mu^2 \lry)$, we get

\begin{align}
    & \frac{10 \Lf^2 \lrx}{\mu^2 \lry} \Big[ \lb \Phi (\bxtp) - f(\bxtp, \bytp) \rb - \lb \Phi (\bxt) - f(\bxt, \byt) \rb \Big] \nn \\
    & \leq - \frac{5 \lrx \cvxt \Lf^2}{\mu} \lb \Phi (\bxt) - f(\bxt, \byt) \rb - \frac{5 \kappa^2 \cvxt \lrx}{2 \lry^2} \norm{\Tbytp - \byt}^2 \nn \\
    & \quad + \frac{5 \Lf^2 \cvxt}{\mu^2 \lry} \lnr \Tbxtp - \bxt \rnr^2 + 10 \kappa^2 \lrx \cvxt \norm{\Gy f(\bxt, \byt) - \bdyt}^2. 
    \label{eq_proof:thm:NC_PL_mom_1}
\end{align}


Define
\begin{align}
    \IEt \triangleq \Phi(\bxt) - \Phi^* + \frac{10 \Lf^2 \lrx}{\mu^2 \lry} \lb \Phi (\bxt) - f(\bxt, \byt) \rb. \nn
\end{align}
Then, using \cref{lem:NC_PL_mom_phi_error} and \eqref{eq_proof:thm:NC_PL_mom_1}, we get
\begin{align}
    \IEtp - \IEt & \leq - \lp \frac{\cvxt}{2 \lrx} - \frac{5 \Lf^2 \cvxt}{\mu^2 \lry} \rp \lnr \Tbxtp - \bxt \rnr^2 - \frac{\lrx \cvxt \Lf^2}{\mu} \lb \Phi (\bxt) - f(\bxt, \byt) \rb - \frac{5 \kappa^2 \cvxt \lrx}{2 \lry^2} \norm{\Tbytp - \byt}^2 \nn \\
    & \quad + 2 \lrx \cvxt \lnr \Gx f(\bxt, \byt) - \bdxt \rnr^2 + 10 \kappa^2 \lrx \cvxt \norm{\Gy f(\bxt, \byt) - \bdyt}^2 \nn \\
    & \leq - \frac{\cvxt}{4 \lrx} \lnr \Tbxtp - \bxt \rnr^2 - \frac{\lrx \cvxt \Lf^2}{\mu} \lb \Phi (\bxt) - f(\bxt, \byt) \rb - \frac{5 \kappa^2 \cvxt \lrx}{2 \lry^2} \norm{\Tbytp - \byt}^2 \nn \\
    & \quad + 2 \lrx \cvxt \lnr \Gx f(\bxt, \byt) - \bdxt \rnr^2 + 2 \cvxt \lry \norm{\Gy f(\bxt, \byt) - \bdyt}^2. \label{eq_proof:thm:NC_PL_mom_2}
\end{align}
where, $- \frac{\cvxt}{2 \lrx} + \frac{5 \kappa^2 \cvxt}{\lry} \leq - \frac{\cvxt}{4 \lrx}$, since $\lrx \leq \frac{\lry}{20 \kappa^2}$. 
Next, we choose $\momx = \momy = \mom = 3$, and define
\begin{align}
    \FEt \triangleq \IEt + \frac{2 \lrx}{\mu \lry} \lnr \Gx f(\bxt, \byt) - \bdxt \rnr^2 + \frac{2 \lrx}{\mu \lry} \lnr \Gy f(\bxt, \byt) - \bdyt \rnr^2, \quad t \geq 0. \nn
\end{align}
Then, using the bounds in \cref{lem:NC_PL_mom_grad_var_bound} and \eqref{eq_proof:thm:NC_PL_mom_2}, we get
\begin{align}
    \mbe \lb \FEtp - \FEt \rb & \leq - \lp \frac{\cvxt}{2 \lrx} - 2 \frac{2 \lrx}{\mu \lry} \frac{2 \Lf^2 \cvxt}{3} \rp \mbe \lnr \Tbxtp - \bxt \rnr^2 - \frac{\lrx \cvxt \Lf^2}{\mu} \mbe \lb \Phi (\bxt) - f(\bxt, \byt) \rb \nn \\
    & \quad - \lp \frac{2 \lrx}{\mu \lry} \frac{3 \cvxt}{2} - 2 \cvxt \lrx \rp \mbe \lb \lnr \Gx f(\bxt, \byt) - \bdxt \rnr^2 + \lnr \Gy f(\bxt, \byt) - \bdyt \rnr^2 \rb \nn \\
    & \quad - \lp \frac{5 \cvxt \kappa^2 \lrx}{2 \lry^2} - 2 \frac{2 \lrx}{\mu \lry} \frac{2 \Lf^2 \cvxt}{3} \rp \mbe \lnr \Tbytp - \byt \rnr^2 + 2 \frac{2 \lrx}{\mu \lry} 3 \cvxt \Lf^2 \CExytp + 2 \frac{2 \lrx}{\mu \lry} \frac{9 \cvxt^2 \sigma^2}{n} \nn \\
    & \leq - \frac{\cvxt}{4 \lrx} \mbe \lnr \Tbxtp - \bxt \rnr^2 - \frac{\lrx \cvxt \Lf^2}{\mu} \mbe \lb \Phi (\bxt) - f(\bxt, \byt) \rb -\frac{\cvxt \kappa^2 \lrx}{\lry^2} \mbe \lnr \Tbytp - \byt \rnr^2 \nn \\
    & \quad - \frac{2 \cvxt \lrx}{\mu \lry} \mbe \lb \lnr \Gx f(\bxt, \byt) - \bdxt \rnr^2 + \lnr \Gy f(\bxt, \byt) - \bdyt \rnr^2 \rb + \frac{4 \lrx}{\mu \lry} \lb 3 \cvxt \Lf^2 \CExytp + \frac{9 \cvxt^2 \sigma^2}{n} \rb \label{eq_proof:thm:NC_PL_mom_3}
\end{align}

Here, using $\lry\leq 1/(8 \Lf) \leq 1/(8 \mu)$ and $\lry\geq 20 \lrx \kappa^2$, we simplify the coefficients in \eqref{eq_proof:thm:NC_PL_mom_3} as follows
\begin{align*}
    & - \frac{\cvxt}{2 \lrx} \lp 1 - \frac{16 \lrx^2 \Lf^2}{3 \mu \lry} \rp = - \frac{\cvxt}{2 \lrx} + \frac{\cvxt}{2 \lrx} \frac{16 \mu \lry \kappa^2}{3} \frac{\lrx^2}{\lry^2} \leq - \frac{\cvxt}{2 \lrx} + \frac{\cvxt}{2 \lrx} \frac{16}{3} \frac{1}{8} \frac{1}{400 \kappa^2} \leq -\frac{\cvxt}{4 \lrx} \tag{$\because \kappa \geq 1$} \\
    & -\lp \frac{2 \lrx}{\mu \lry} \frac{3 \cvxt}{2} - 2 \lrx \cvxt \rp \leq - \frac{3 \lrx \cvxt}{\mu \lry} + \frac{2 \lrx \cvxt}{8 \mu \lry} \leq - \frac{2 \lrx \cvxt}{\mu \lry}, \tag{$\because 1 \leq 1/(8 \mu \lry)$} \nn \\
    & - \lp \frac{5 \cvxt \kappa^2 \lrx}{2 \lry^2} - \frac{4 \lrx}{\mu \lry} \frac{2 \Lf^2 \cvxt}{3} \rp = \frac{\cvxt \kappa^2 \lrx}{\lry^2} \lp -\frac{5}{2} + \frac{8}{3} \lry \mu \rp \leq \frac{\cvxt \kappa^2 \lrx}{\lry^2} \lp -\frac{5}{2} + \frac{1}{3} \rp \leq -\frac{\cvxt \kappa^2 \lrx}{\lry^2}. \tag{$\because 1 \leq 1/(8 \mu \lry)$}
\end{align*}
Summing \eqref{eq_proof:thm:NC_PL_mom_3} over $t=0, \hdots, T-1$ and rearranging the terms, we get
\begin{align}
    & \frac{1}{T} \sumtT \frac{\cvxt \lrx}{4} \lb \frac{1}{\lrx^2} \mbe \lnr \Tbxtp - \bxt \rnr^2 + \frac{4 \Lf^2}{\mu} \mbe \lb \Phi (\bxt) - f(\bxt, \byt) \rb + \frac{8}{\mu \lry} \mbe \lnr \Gx f(\bxt, \byt) - \bdxt \rnr^2 \rb \nn \\
    & \qquad \leq \frac{1}{T} \sumtT \frac{4 \lrx}{\mu \lry} \lb 9 \cvxt^2 \frac{\sigma^2}{n} + 3 \cvxt \Lf^2 \CExytp \rb + \frac{1}{T} \sumtT \mbe \lb \FEt - \FEtp \rb. \nn
\end{align}

We choose $\cvxt = \cvx$ for all $t$. $\frac{1}{8 \mu \lry} \geq 1$. Also, $\FEt \geq 0, \forall \ t$. Therefore,
\begin{align}
    & \frac{1}{T} \sumtT \lb \frac{1}{\lrx^2} \mbe \lnr \Tbxtp - \bxt \rnr^2 + \frac{2 \Lf^2}{\mu} \mbe \lb \Phi (\bxt) - f(\bxt, \byt) \rb + \mbe \lnr \Gx f(\bxt, \byt) - \bdxt \rnr^2 \rb \nn \\
    & \qquad \leq \frac{4 \FE_0}{\lrx \cvx T} + \frac{1}{T} \sumtT \frac{16}{\mu \lry} \lb 9 \cvx \frac{\sigma^2}{n} + 3 \Lf^2 \CExytp \rb \tag{$\because \FEt \geq 0$ for all $t$} \\
    & \qquad \leq \mco \lp \frac{\FE_0}{\lrx \cvx T} + \frac{\cvx}{\mu \lry} \frac{\sigma^2}{n} \rp + \mco \lp \frac{\Lf^2}{\mu \lry} (\sync - 1)^2 \cvx^2 \lp \lp \lrx^2 + \lry^2 \rp \sigma^2 + \lrx^2 \heterox^2 + \lry^2 \heteroy^2 \rp \rp \tag{\cref{cor:NC_PL_mom_induct_bd_cons_error_xy}} \\
    & \qquad = \mco \lp \frac{\kappa^2}{\lry \cvx T} + \frac{\cvx}{\mu \lry} \frac{\sigma^2}{n} \rp + \mco \lp \kappa^2 \mu (\sync - 1)^2 \cvx^2 \lp \lry \lp \sigma^2 + \heteroy^2 \rp + \frac{\lrx^2}{\lry} \lp \sigma^2 + \heterox^2 \rp \rp \rp \nn \\
    & \qquad = \mco \lp \frac{\kappa^2}{\lry \cvx T} + \frac{\cvx}{\mu \lry} \frac{\sigma^2}{n} \rp + \mco \lp (\sync - 1)^2 \cvx^2 \lp \sigma^2 + \heteroy^2 \rp + \mu (\sync - 1)^2 \cvx^2 \lp \lrx \lp \sigma^2 + \heterox^2 \rp \rp \rp \tag{$\because \lry \leq \frac{\mu}{8 \Lf^2}, \frac{\lrx}{\lry} \leq \frac{1}{20 \kappa^2}$} \\
    & \qquad \leq \underbrace{\mco \lp \frac{\kappa^2}{\lry \cvx T} + \frac{\cvx}{\mu \lry} \frac{\sigma^2}{n} \rp}_{\substack{\text{Single client} \\
    \text{convergence error}}} + \underbrace{\mco \Big( (\sync - 1)^2 \cvx^2 \lp \sigma^2 + \heterox^2 + \heteroy^2 \rp \Big)}_{\text{Error due to local updates}}. \tag{$\because \mu \lrx \leq 1$}
\end{align}
Finally, since $\FE_0$ is a constant, and using $\lry \geq 20 \lrx \kappa^2$, we get \eqref{eq_proof:thm_NC_PL_mom_conv_rate}.

Further, with $\cvx = \sqrt{\frac{n}{T}}$ in \eqref{eq_proof:thm_NC_PL_mom_conv_rate}, we get
\begin{align}
    & \frac{1}{T} \sumtT \mbe \lb \frac{1}{\lrx^2} \lnr \Tbxtp - \bxt \rnr^2 + \frac{2 \Lf^2}{\mu} \lb \Phi (\bxt) - f(\bxt, \byt) \rb + \lnr \Gx f(\bxt, \byt) - \bdxt \rnr^2 \rb \nn \\
    & \qquad \leq \mco \lp \frac{\kappa^2 + \sigma^2}{\sqrt{nT}} \rp + \mco \lp \frac{n (\sync-1)^2 \lp \sigma^2 + \heterox^2 + \heteroy^2 \rp}{T} \rp. \nn
\end{align}
\end{proof}



\begin{proof}[Proof of \cref{cor:NC_PL_mom_comm_cost}]
We assume $T \geq n^3$.
To reach an $\epsilon$-accurate point, we note that using Jensen's inequality
\begin{align}
    & \min_{t \in [T-1]} \mbe \lb \frac{1}{\lrx} \lnr \Tbxtp - \bxt \rnr + \Lf \sqrt{\frac{2}{\mu} \lb \Phi (\bxt) - f(\bxt, \byt) \rb} + \lnr \Gx f(\bxt, \byt) - \bdxt \rnr \rb \nn \\
    & \leq \frac{1}{T} \sumtT \mbe \lb \frac{1}{\lrx} \lnr \Tbxtp - \bxt \rnr + \Lf \sqrt{\frac{2}{\mu} \lb \Phi (\bxt) - f(\bxt, \byt) \rb} + \lnr \Gx f(\bxt, \byt) - \bdxt \rnr \rb \nn \\
    &  \leq \lb \frac{3}{T} \sumtT \mbe \lp \frac{1}{\lrx^2} \lnr \Tbxtp - \bxt \rnr^2 + \frac{2 \Lf^2}{\mu} \lb \Phi (\bxt) - f(\bxt, \byt) \rb + \lnr \Gx f(\bxt, \byt) - \bdxt \rnr^2 \rp \rb^{1/2} \nn \\
    & \leq \mco \lp \frac{\kappa + \sigma}{(n T)^{1/4}} \rp + \mco \lp \sync \sqrt{\frac{n \lp \sigma^2 + \heterox^2 + \heteroy^2 \rp}{T}} \rp, \nn
\end{align}
where we use $\sqrt{a+b} \leq \sqrt{a} + \sqrt{b}$.
Hence,  we need $T = \mco \lp \kappa^4/(n \epsilon^4) \rp$ iterations, to reach an $\epsilon$-accurate point.
We can choose $\sync \leq \mco \lp \frac{T^{1/4}}{n^{3/4}} \rp$ without affecting the convergence rate.
Hence, the number of communication rounds is $\mc O \lp \frac{T}{\sync} \rp = \mc O \lp (n T)^{3/4} \rp = \mco \lp \kappa^3/\epsilon^3 \rp$. 
\end{proof}

\subsection{Proofs of the Intermediate Lemmas}
\label{sec:NC_PL_mom_int_results_proofs}

\begin{proof}[Proof of \cref{lem:NC_PL_mom_Phi_1_step_decay}]
Using $\Lp$-smoothnes of $\Phi(\cdot)$ (\cref{lem:Phi_PL_smooth_nouiehed})
\begin{align}
    & \Phi (\bxtp) - \Phi (\bxt) \leq \langle \G \Phi(\bxt), \bxtp - \bxt \rangle + \frac{\Lp}{2} \lnr \bxtp - \bxt \rnr^2 \nn
    \\
    & \quad = \cvxt \langle \G \Phi(\bxt), \Tbxtp - \bxt \rangle + \frac{\Lp \cvxt^2}{2} \lnr \bxtp - \bxt \rnr^2 \tag{see updates in \eqref{eq:NC_mom_update_avg}} \\
    & \quad = \cvxt \langle \bdxt, \Tbxtp - \bxt \rangle + \cvxt \langle \Gx f(\bxt, \byt) - \bdxt, \Tbxtp - \bxt \rangle \nn \\
    & \qquad + \cvxt \lan \G \Phi(\bxt) - \Gx f(\bxt, \byt), \Tbxtp - \bxt \ran + \frac{\Lp \cvxt^2}{2} \lnr \Tbxtp - \bxt \rnr^2. \label{eq:lem:NC_PL_mom_Phi_1_step_decay_1}
\end{align}
Next, we bound the individual inner product terms in \eqref{eq:lem:NC_PL_mom_Phi_1_step_decay_1}.
\begin{align}
    \cvxt \langle \bdxt, \Tbxtp - \bxt \rangle &= -\frac{\cvxt}{\lrx} \lnr \Tbxtp - \bxt \rnr^2, \label{eq:lem:NC_PL_mom_Phi_1_step_decay_2a} \\
    \cvxt \langle \G \Phi(\bxt) - \Gx f(\bxt, \byt), \Tbxtp - \bxt \rangle & \overset{(a)}{\leq} \frac{\cvxt}{8 \lrx} \lnr \Tbxtp - \bxt \rnr^2 + \cvxt 2 \lrx \lnr \G \Phi(\bxt) - \Gx f(\bxt, \byt) \rnr^2, \nn
    \\
    & \overset{(b)}{\leq} \frac{\cvxt}{8 \lrx} \lnr \Tbxtp - \bxt \rnr^2 + 2 \lrx \cvxt \Lf^2 \lnr \by^*(\bxt) - \byt \rnr^2, \nn \\
    & \leq \frac{\cvxt}{8 \lrx} \lnr \Tbxtp - \bxt \rnr^2 + \frac{4 \lrx \cvxt \Lf^2}{\mu} \lb f(\bxt, \by^*(\bxt)) - f(\bxt, \byt) \rb, \nn \\
    & = \frac{\cvxt}{8 \lrx} \lnr \Tbxtp - \bxt \rnr^2 + \frac{4 \lrx \cvxt \Lf^2}{\mu} \lb \Phi(\bxt) - f(\bxt, \byt) \rb, \label{eq:lem:NC_PL_mom_Phi_1_step_decay_2b} \\
    \cvxt \langle \Gx f(\bxt, \byt) - \bdxt, \Tbxtp - \bxt \rangle & \leq \frac{\cvxt}{8 \lrx} \lnr \Tbxtp - \bxt \rnr^2 + 2 \lrx \cvxt \lnr \Gx f(\bxt, \byt) - \bdxt \rnr^2, \label{eq:lem:NC_PL_mom_Phi_1_step_decay_2c}
\end{align}
where \eqref{eq:lem:NC_PL_mom_Phi_1_step_decay_2a} follows from the update expression of \textit{virtual} averages in \eqref{eq:NC_mom_update_avg}; 
$(a)$ and \eqref{eq:lem:NC_PL_mom_Phi_1_step_decay_2c} both follow from Young's inequality \cref{lem:Young}
(with $\gamma = 4 \lrx$); 
$(b)$ follows from \cref{lem:Phi_PL_smooth_nouiehed} and $\Lf$-smoothness of $f(\bxt, \cdot)$ (\cref{assum:smoothness});
and \eqref{eq:lem:NC_PL_mom_Phi_1_step_decay_2b} follows from the quadratic growth condition of $\mu$-PL functions (\cref{lem:quad_growth}).
Substituting \eqref{eq:lem:NC_PL_mom_Phi_1_step_decay_2a}-\eqref{eq:lem:NC_PL_mom_Phi_1_step_decay_2c} in \eqref{eq:lem:NC_PL_mom_Phi_1_step_decay_1}, we get
\begin{align*}
    \Phi (\bxtp) - \Phi (\bxt) & \leq - \lp \frac{3 \cvxt}{4 \lrx} - \frac{\Lp \cvxt^2}{2} \rp \lnr \Tbxtp - \bxt \rnr^2 + \frac{4 \lrx \cvxt \Lf^2}{\mu} \lb \Phi(\bxt) - f(\bxt, \byt) \rb + 2 \lrx \cvxt \lnr \Gx f(\bxt, \byt) - \bdxt \rnr^2.
\end{align*}
Notice that for $\cvxt \leq \frac{\mu}{4 \lrx \Lf^2}$, $\frac{\Lp \cvxt^2}{2} \leq \kappa \Lf \cvxt^2 \leq \frac{\cvxt}{4 \lrx}$. Hence the result follows.
\end{proof}

\begin{proof}[Proof of Lemma \ref{lem:NC_PL_mom_phi_error}]
Using $\Lf$-smoothness of $f(\bx, \cdot)$ (\cref{assum:smoothness}),
\begin{align}
    f(\bxtp, \byt) &+ \lan \Gy f(\bxtp, \byt), \bytp - \byt \ran - \frac{\Lf}{2} \norm{\bytp - \byt}^2 \leq f(\bxtp, \bytp) \nn \\
    \Rightarrow f(\bxtp, \byt) & \leq f(\bxtp, \bytp) - \cvxt \lan \Gy f(\bxtp, \byt), \Tbytp - \byt \ran + \frac{\cvxt^2 \Lf}{2} \norm{\Tbytp - \byt}^2. \label{eq_proof:lem:NC_PL_mom_phi_error_1}
\end{align}
Next, we bound the inner product in \eqref{eq_proof:lem:NC_PL_mom_phi_error_1}.
\begin{align}
    & - \cvxt \lan \Gy f(\bxtp, \byt), \Tbytp - \byt \ran = - \cvxt \lry \lan \Gy f(\bxtp, \byt), \bdyt \ran \tag{using \eqref{eq:NC_mom_update_avg}} \\
    &= -\frac{\cvxt \lry}{2} \lb \norm{\Gy f(\bxtp, \byt)}^2 + \norm{\bdyt}^2 - \norm{\Gy f(\bxtp, \byt) - \Gy f(\bxt, \byt) + \Gy f(\bxt, \byt) - \bdyt}^2 \rb \nn \\
    & \leq -\cvxt \lry \mu \lb \Phi (\bxtp) - f(\bxtp, \byt) \rb - \frac{\cvxt}{2 \lry} \norm{\Tbytp - \byt}^2 + \cvxt \lry \lb \Lf^2 \norm{\bxtp - \bxt}^2 + \norm{\Gy f(\bxt, \byt) - \bdyt}^2 \rb \label{eq_proof:lem:NC_PL_mom_phi_error_2}
\end{align}
where, \eqref{eq_proof:lem:NC_PL_mom_phi_error_2} follows from the quadratic growth condition of $\mu$-PL functions (\cref{lem:quad_growth}),
\begin{align}
    \norm{\Gy f(\bxtp, \byt)}^2 \geq 2 \mu \lp \max_\by f(\bxtp, \by) - f(\bxtp, \byt) \rp = 2 \mu \lp \Phi (\bxtp) - f(\bxtp, \byt) \rp. \nn
\end{align}
Substituting \eqref{eq_proof:lem:NC_PL_mom_phi_error_2} in \eqref{eq_proof:lem:NC_PL_mom_phi_error_1}, we get
\begin{align}
    f(\bxtp, \byt) & \leq f(\bxtp, \bytp) -\cvxt \lry \mu \lb \Phi (\bxtp) - f(\bxtp, \byt) \rb - \frac{\cvxt}{2 \lry} \norm{\Tbytp - \byt}^2 + \frac{\cvxt^2 \Lf}{2} \norm{\Tbytp - \byt}^2 \nn \\
    & \quad + \cvxt \lry \lb \Lf^2 \norm{\bxtp - \bxt}^2 + \norm{\Gy f(\bxt, \byt) - \bdyt}^2 \rb. \nn
\end{align}
Rearranging the terms we get
\begin{align}
    \Phi (\bxtp) - f(\bxtp, \bytp) & \leq \lp 1-\cvxt \lry \mu \rp \lb \Phi (\bxtp) - f(\bxtp, \byt) \rb - \frac{\cvxt}{2} \lp \frac{1}{\lry} - \cvxt \Lf \rp \norm{\Tbytp - \byt}^2 \nn \\
    & \quad + \cvxt \lry \lb \Lf^2 \norm{\bxtp - \bxt}^2 + \norm{\Gy f(\bxt, \byt) - \bdyt}^2 \rb. \label{eq_proof:lem:NC_PL_mom_phi_error_3}
\end{align}

Next, we bound $\Phi (\bxtp) - f(\bxtp, \byt)$.
\begin{align}
    & \Phi (\bxtp) - f(\bxtp, \byt) = \Phi (\bxtp) - \Phi (\bxt) + \lb \Phi (\bxt) - f(\bxt, \byt) \rb + \underbrace{f(\bxt, \byt) - f(\bxtp, \byt)}_{I}. \label{eq_proof:lem:NC_PL_mom_phi_error_4}
\end{align}
Next, we bound $I$. Using $\Lf$-smoothness of $f(\cdot, \byt)$,
\begin{align}
    & f(\bxt, \byt) + \lan \Gx f(\bxt, \byt), \bxtp - \bxt \ran - \frac{\Lf}{2} \norm{\bxtp - \bxt}^2 \leq f(\bxtp, \byt) \nn \\
    \Rightarrow I &=  f(\bxt, \byt) - f(\bxtp, \byt) \nn \\
    & \leq -\cvxt \lan \Gx f(\bxt, \byt), \Tbxtp - \bxt \ran + \frac{\cvxt^2 \Lf}{2} \norm{\Tbxtp - \bxt}^2 \nn \\
    &= -\cvxt \lan \Gx f(\bxt, \byt) - \G \Phi(\bxt), \Tbxtp - \bxt \ran -\cvxt \lan \G \Phi(\bxt), \Tbxtp - \bxt \ran + \frac{\cvxt^2 \Lf}{2} \norm{\Tbxtp - \bxt}^2 \nn \\
    & \leq \frac{\cvxt}{8 \lrx} \lnr \Tbxtp - \bxt \rnr^2 + \frac{4 \lrx \cvxt \Lf^2}{\mu} \lb \Phi(\bxt) - f(\bxt, \byt) \rb \tag{using \eqref{eq:lem:NC_PL_mom_Phi_1_step_decay_2b}} \\
    & \quad + \Phi (\bxt) - \Phi (\bxtp) + \frac{\cvxt^2 \Lp}{2} \lnr \Tbxtp - \bxt \rnr^2 + \frac{\cvxt^2 \Lf}{2} \norm{\Tbxtp - \bxt}^2 \tag{smoothness of $\Phi$ (\cref{lem:Phi_PL_smooth_nouiehed})} \\
    & = \Phi (\bxt) - \Phi (\bxtp) + \frac{4 \lrx \cvxt \Lf^2}{\mu} \lb \Phi(\bxt) - f(\bxt, \byt) \rb + \frac{\cvxt}{2} \lp \frac{1}{4 \lrx} + 2 \cvxt \Lp \rp \lnr \Tbxtp - \bxt \rnr^2 \tag{$\because \Lf \leq \Lp$}.
\end{align}
Using the bound on $I$ in \eqref{eq_proof:lem:NC_PL_mom_phi_error_4} and then substituting in \eqref{eq_proof:lem:NC_PL_mom_phi_error_3}, we get
\begin{align}
    & \Phi (\bxtp) - f(\bxtp, \bytp) \nn \\
    & \leq (1 - \cvxt \lry \mu) \lb \lp 1 + \frac{4 \lrx \cvxt \Lf^2}{\mu} \rp \lb \Phi(\bxt) - f(\bxt, \byt) \rb + \frac{\cvxt}{2} \lp \frac{1}{4 \lrx} + 2 \cvxt \Lp \rp \lnr \Tbxtp - \bxt \rnr^2 \rb \nn \\
    & \quad - \frac{\cvxt}{2} \lp \frac{1}{\lry} - \cvxt \Lf \rp \norm{\Tbytp - \byt}^2 + \cvxt \lry \lb \Lf^2 \norm{\bxtp - \bxt}^2 + \norm{\Gy f(\bxt, \byt) - \bdyt}^2 \rb \nn \\
    & \overset{(a)}{\leq} \lp 1 - \frac{\cvxt \lry \mu}{2} \rp \lb \Phi (\bxt) - f(\bxt, \byt) \rb + \frac{\cvxt}{2} \lp \frac{1}{4 \lrx} + 2 \cvxt \Lp + 2 \lry \Lf^2 \cvxt^2 \rp \lnr \Tbxtp - \bxt \rnr^2 \nn \\
    & \quad - \frac{\cvxt}{2} \lp \frac{1}{\lry} - \cvxt \Lf \rp \norm{\Tbytp - \byt}^2 + \cvxt \lry \norm{\Gy f(\bxt, \byt) - \bdyt}^2. \nn \\
    & \overset{(b)}{\leq} \lp 1 - \frac{\cvxt \lry \mu}{2} \rp \lb \Phi (\bxt) - f(\bxt, \byt) \rb + \frac{\cvxt}{2 \lrx} \lnr \Tbxtp - \bxt \rnr^2 - \frac{\cvxt}{4 \lry} \norm{\Tbytp - \byt}^2 + \cvxt \lry \norm{\Gy f(\bxt, \byt) - \bdyt}^2. \nn
\end{align}
where in $(a)$ we choose $\lrx$ such that $(1 - \cvxt \lry \mu) \lp 1 + \frac{4 \lrx \cvxt \Lf^2}{\mu} \rp \leq \lp 1 - \frac{\cvxt \lry \mu}{2} \rp$. 
This holds if $\frac{4 \lrx \cvxt \Lf^2}{\mu} \leq \frac{\cvxt \lry \mu}{2} \Rightarrow \lrx \leq \frac{\lry}{8 \kappa^2}$, where $\kappa = \Lf/\mu \geq 1$ is the condition number.
Finally, $(b)$ follows since $\cvxt \lry \leq \frac{1}{2 \Lf}$ and $\cvxt \leq \frac{\mu}{8 \lrx \Lf^2} = \frac{1}{8 \lrx \kappa \Lf}$. Therefore,
\begin{align*}
    & 2 \cvxt \Lp \leq 4 \kappa \cvxt \Lf \leq \frac{1}{2 \lrx} \tag{$\Lp \leq 2 \kappa \Lf$} \\
    & 2 \lry \Lf^2 \cvxt^2 \leq 2 \lry \cvxt \frac{\mu}{8 \lrx} \leq \frac{\mu}{8 \lrx} \frac{1}{\Lf} \leq \frac{1}{8 \lrx}.
\end{align*}
\end{proof}

\begin{proof}[Proof of \cref{lem:NC_PL_mom_grad_var_bound}]
We prove \eqref{eq:lem:NC_PL_mom_grad_var_bound_x} here. The proof for \eqref{eq:lem:NC_PL_mom_grad_var_bound_y} is analogous.
\begin{align}
    & \mbe \lnr \Gx f(\bxtp, \bytp) - \bdxtp \rnr^2 \nn \\
    &= \mbe \lnr \Gx f(\bxtp, \bytp) - (1 - \momx \cvxt) \bdxt - \momx \cvxt \frac{1}{n} \sumin \Gx f_i (\bxitp, \byitp; \xiitp) \rnr^2 \tag{see \eqref{eq:NC_mom_update_avg}} \\
    &= \mbe \left\| \Gx f(\bxtp, \bytp) - (1 - \momx \cvxt) \bdxt - \momx \cvxt \frac{1}{n} \sumin \Gx f_i (\bxitp, \byitp) \right. \nn \\
    & \qquad \qquad \qquad \qquad \qquad \left.- \momx \cvxt \frac{1}{n} \sumin \lp \Gx f_i (\bxitp, \byitp; \xiitp) - \Gx f_i (\bxitp, \byitp) \rp \right\|^2 \nn \\
    & \overset{(a)}{=} \mbe \lnr (1 - \momx \cvxt) \lp \Gx f(\bxtp, \bytp) - \bdxt \rp + \momx \cvxt \lp \Gx f(\bxtp, \bytp) - \frac{1}{n} \sumin \Gx f_i (\bxitp, \byitp) \rp \rnr^2 \nn \\
    & \qquad \qquad \qquad + \momx^2 \cvxt^2 \mbe \lnr \frac{1}{n} \sumin \lp \Gx f_i (\bxitp, \byitp; \xiitp) - \Gx f_i (\bxitp, \byitp) \rp \rnr^2 \nn \\
    & \leq (1 + a_1) (1 - \momx \cvxt)^2 \mbe \lnr \Gx f(\bxtp, \bytp) - \bdxt \rnr^2 \nn \\
    & \qquad + \momx^2 \cvxt^2 \lp 1 + \dfrac{1}{a_1} \rp \mbe \lnr \frac{1}{n} \sumin \lp \Gx f_i (\bxtp, \bytp) - \Gx f_i (\bxitp, \byitp) \rp \rnr^2 + \momx^2 \cvxt^2 \frac{\sigma^2}{n}. \label{eq:proof_lem:NC_PL_mom_grad_var_bound_x_3}
\end{align}
Here, $(a)$ follows from Assumption \ref{assum:bdd_var} (unbiasedness of stochastic gradients),
\begin{align*}
    & \mbe \left\langle (1 - \momx \cvxt) \lp \Gx f(\bxtp, \bytp) - \bdxt \rp + \momx \cvxt \lp \Gx f(\bxtp, \bytp) - \frac{1}{n} \sumin \Gx f_i (\bxitp, \byitp) \rp, \right. \nn \\
    & \qquad \qquad \qquad \left. \frac{1}{n} \sumin \lp \Gx f_i (\bxitp, \byitp; \xiitp) - \Gx f_i (\bxitp, \byitp) \rp \right\rangle \nn \\
    &= \mbe \left\langle (1 - \momx \cvxt) \lp \Gx f(\bxtp, \bytp) - \bdxt \rp + \momx \cvxt \lp \Gx f(\bxtp, \bytp) - \frac{1}{n} \sumin \Gx f_i (\bxitp, \byitp) \rp, \right. \nn \\
    & \qquad \qquad \qquad \left. \frac{1}{n} \sumin \lp \mbe \lb \Gx f_i (\bxitp, \byitp; \xiitp) \rb - \Gx f_i (\bxitp, \byitp) \rp \right\rangle = 0. \tag{Law of total expectation}
\end{align*}
Also, \eqref{eq:proof_lem:NC_PL_mom_grad_var_bound_x_3} follows from Assumption \ref{assum:bdd_var} (independence of stochastic gradients across clients), and \cref{lem:Young}
(with $\gamma = a_1$).
Next, in \eqref{eq:proof_lem:NC_PL_mom_grad_var_bound_x_3}, we choose $a_1$ such that $\lp 1 + \frac{1}{a_1} \rp \momx \cvxt = 1$, i.e., $a_1 = \frac{\momx \cvxt}{1 - \momx \cvxt}$. 
Therefore, $(1-\momx \cvxt) (1 + a_1) = 1$. Consequently, in \eqref{eq:proof_lem:NC_PL_mom_grad_var_bound_x_3} we get,
\begin{align}
    & \mbe \lnr \Gx f(\bxtp, \bytp) - \bdxtp \rnr^2 \nn \\
    & \leq (1 - \momx \cvxt) \mbe \lnr \Gx f(\bxtp, \bytp) - \Gx f(\bxt, \byt) + \Gx f(\bxt, \byt) - \bdxt \rnr^2 + \momx^2 \cvxt^2 \frac{\sigma^2}{n} \nn \\
    & \quad + \momx \cvxt \frac{1}{n} \sumin \Lf^2 \mbe \lb \lnr \bxtp - \bxitp \rnr^2 + \lnr \bytp - \byitp \rnr^2 \rb 
    \tag{Jensen's inequality with $\norm{\cdot}^2_2$; \cref{assum:smoothness}}
    \\
    & \leq (1 - \momx \cvxt) \lb (1 + a_2) \mbe \lnr \Gx f(\bxt, \byt) - \bdxt \rnr^2 + \lp 1 + \dfrac{1}{a_2} \rp \mbe \lnr \Gx f(\bxtp, \bytp) - \Gx f(\bxt, \byt) \rnr^2 \rb + \momx^2 \cvxt^2 \frac{\sigma^2}{n} \nn \\
    & \quad + \momx \cvxt \frac{1}{n} \sumin \Lf^2 \mbe \lb \lnr \bxtp - \bxitp \rnr^2 + \lnr \bytp - \byitp \rnr^2 \rb, \label{eq:proof_lem:NC_PL_mom_grad_var_bound_x_5}
\end{align}
In \eqref{eq:proof_lem:NC_PL_mom_grad_var_bound_x_5}, we choose $a_2 = \frac{\momx \cvxt}{2}$. Then, $(1-\momx \cvxt) \lp 1 + \frac{\momx \cvxt}{2} \rp \leq 1 - \frac{\momx \cvxt}{2}$, and $(1-\momx \cvxt) \lp 1 + \frac{2}{\momx \cvxt} \rp \leq \frac{2}{\momx \cvxt}$. Therefore, we get
\begin{align}
    & \mbe \lnr \Gx f(\bxtp, \bytp) - \bdxtp \rnr^2 \nn \\
    & \leq \lp 1 - \frac{\momx \cvxt}{2} \rp \mbe \lnr \Gx f(\bxt, \byt) - \bdxt \rnr^2 + \frac{2}{\momx \cvxt} \Lf^2 \mbe \lb \lnr \bxtp - \bxt \rnr^2 + \lnr \bytp - \byt \rnr^2 \rb \nn \\
    & \qquad + \momx^2 \cvxt^2 \frac{\sigma^2}{n} + \momx \cvxt \frac{1}{n} \sumin \Lf^2 \mbe \lb \lnr \bxtp - \bxitp \rnr^2 + \lnr \bytp - \byitp \rnr^2 \rb \nn \\
    &= \lp 1 - \frac{\momx \cvxt}{2} \rp \mbe \lnr \Gx f(\bxt, \byt) - \bdxt \rnr^2 + \frac{2 \Lf^2 \cvxt}{\momx} \mbe \lb \lnr \Tbxtp - \bxt \rnr^2 + \lnr \Tbytp - \byt \rnr^2 \rb \nn \\
    & \qquad + \momx^2 \cvxt^2 \frac{\sigma^2}{n} + \momx \cvxt \frac{1}{n} \sumin \Lf^2 \mbe \lb \lnr \bxtp - \bxitp \rnr^2 + \lnr \bytp - \byitp \rnr^2 \rb,
\end{align}
Finally, we choose $\momx = \mom$.
This concludes the proof.
\end{proof}

\begin{proof}[Proof of \cref{lem:NC_PL_mom_cons_errs_recursion}]
For the sake of clarity, we repeat the following notations:
$\CExyt \triangleq \frac{1}{n} \sumin \mbe \lp \lnr \bxit - \bxt \rnr^2 + \lnr \byit - \byt \rnr^2 \rp$, $\CEpt \triangleq \frac{1}{n} \sumin \mbe \lnr \bdxit - \bdxt \rnr^2$ and $\CEqt \triangleq \frac{1}{n} \sumin \mbe \lnr \bdyit - \bdyt \rnr^2$.

First we prove \eqref{eq:lem:NC_PL_mom_xy_cons_errs_recursion}.
\begin{align}
    \CExytp & \triangleq \frac{1}{n} \sumin \mbe \lp \lnr \bxitp - \bxtp \rnr^2 + \lnr \byitp - \bytp \rnr^2 \rp \nn \\
    & = \frac{1}{n} \sumin \mbe \lp \lnr \lp \bxit - \bxt \rp - \lrx \cvxt \lp \bdxit - \bdxt \rp \rnr^2 + \lnr \lp \byit - \byt \rp + \lry\cvxt \lp \bdyit - \bdyt \rp \rnr^2 \rp
    \tag{from \eqref{eq:NC_mom_update_avg}} \\
    & \leq \frac{1}{n} \sumin \lb (1 + c_1) \mbe \lp \lnr \bxit - \bxt \rnr^2 + \lnr \byit - \byt \rnr^2 \rp + \cvxt^2 \lp 1 + \dfrac{1}{c_1} \rp \mbe \lp \lrx^2 \lnr \bdxit - \bdxt \rnr^2 + \lry^2 \lnr \bdyit - \bdyt \rnr^2 \rp \rb
    \tag{from \cref{lem:Young}, with $\gamma = c_1$}
    \\
    &= (1+c_1) \CExyt + \lp 1 + \dfrac{1}{c_1} \rp \cvxt^2 \lp \lrx^2 \CEpt + \lry^2 \CEqt \rp. \nn
\end{align}
Next, we prove \eqref{eq:lem:NC_PL_mom_p_cons_errs_recursion}.
The proof of \eqref{eq:lem:NC_PL_mom_q_cons_errs_recursion} is analogous, so we skip it here.
\begin{align}
    & \CEptp \triangleq \frac{1}{n} \sumin \mbe \lnr \bdxitp - \bdxtp \rnr^2 \nn \\
    &= \frac{1}{n} \sumin \mbe \lnr (1 - \momx \cvxt) \lp \bdxit - \bdxt \rp + \momx \cvxt \Big( \Gx f_i (\bxitp, \byitp; \xiitp) - \frac{1}{n} \sumjn \Gx f_j (\bxjtp, \byjtp; \xijtp) \Big) \rnr^2 \tag{from \eqref{eq:NC_mom_update_avg}} \\
    & \leq (1 + c_2) (1 - \momx \cvxt)^2 \CEpt + \lp 1 + \dfrac{1}{c_2} \rp \frac{\momx^2 \cvxt^2}{n} \sumin \mbe \lnr \Gx f_i (\bxitp, \byitp; \xiitp) - \frac{1}{n} \sumjn \Gx f_j (\bxjtp, \byjtp; \xijtp) \rnr^2 
    \tag{\cref{lem:Young} (with $\gamma = c_2$)} \\
    & \overset{(a)}{=} (1 - \momx \cvxt) \CEpt \nn \\
    & \ + \momx \cvxt \frac{1}{n} \sumin \mbe \Bigg\| \Gx f_i (\bxitp, \byitp; \xiitp) - \Gx f_i (\bxitp, \byitp) + \Gx f_i (\bxitp, \byitp) - \Gx f_i (\bxtp, \bytp) + \Gx f_i (\bxtp, \bytp) \nn \\
    & \quad - \frac{1}{n} \sumjn \lp \Gx f_j (\bxjtp, \byjtp; \xijtp) - \Gx f_j (\bxjtp, \byjtp) + \Gx f_j (\bxjtp, \byjtp) - \Gx f_j (\bxtp, \bytp) + \Gx f_j (\bxtp, \bytp) \rp \Bigg\|^2 \nn
    \\
    & \overset{(b)}{\leq} (1 - \momx \cvxt) \CEpt + \momx \cvxt \frac{1}{n} \sumin \mbe \Bigg[ \lnr \Gx f_i (\bxitp, \byitp; \xiitp) - \Gx f_i (\bxitp, \byitp) \rnr^2 \nn \\
    & \qquad \qquad \qquad \qquad \qquad + \lnr \frac{1}{n} \sumjn \lp \Gx f_j (\bxjtp, \byjtp; \xijtp) - \Gx f_j (\bxjtp, \byjtp) \rp \rnr^2 \nn \\
    & \qquad \qquad \qquad \qquad \qquad + \Big\| \Gx f_i (\bxitp, \byitp) - \Gx f_i (\bxtp, \bytp) + \Gx f_i (\bxtp, \bytp) \nn \\
    & \qquad \qquad \qquad \qquad \qquad \qquad - \frac{1}{n} \sumjn \lp \Gx f_j (\bxjtp, \byjtp) - \Gx f_j (\bxtp, \bytp) \rp - \Gx f (\bxtp, \bytp) \Big\|^2 \Bigg]
    \nn \\
    & \overset{(c)}{\leq} (1 - \momx \cvxt) \CEpt + \momx \cvxt \frac{1}{n} \sumin \Bigg[ \sigma^2 + \frac{\sigma^2}{n} + 3 \mbe \lnr \Gx f_i (\bxitp, \byitp) - \Gx f_i (\bxtp, \bytp) \rnr^2 \nn \\
    & + 3 \mbe \lnr \Gx f_i (\bxtp, \bytp) - \Gx f (\bxtp, \bytp) \rnr^2 + 3 \mbe \lnr \frac{1}{n} \sumjn \lp \Gx f_j (\bxjtp, \byjtp) - \Gx f_j (\bxtp, \bytp) \rp \rnr^2 \Bigg]
    \nn \\
    & \overset{(d)}{\leq} (1 - \momx \cvxt) \CEpt + \momx \cvxt \frac{1}{n} \sumin \Bigg[ \sigma^2 + \frac{\sigma^2}{n} + 3 \Lf^2 \mbe \lp \lnr \bxitp - \bxtp \rnr^2 + \lnr \byitp - \bytp \rnr^2 \rp + 3 \heterox^2 \nn \\
    & \qquad \qquad \qquad \qquad + 3 \Lf^2 \frac{1}{n} \sumjn \mbe \lp \lnr \bxjtp - \bxtp \rnr^2 + \lnr \byjtp - \bytp \rnr^2 \rp \Bigg]
    \nn \\
    &= (1 - \momx \cvxt) \CEpt + 6 \momx \cvxt \Lf^2 \CExytp + \momx \cvxt \lb \sigma^2 \lp 1 + \dfrac{1}{n} \rp + 3 \heterox^2 \rb. \nn
\end{align}
In $(a)$ we choose $c_2$ such that $\lp 1 + \frac{1}{c_2} \rp \momx \cvxt = 1$, i.e., $c_2 = \frac{\momx \cvxt}{1 - \momx \cvxt}$ and $(1-\momx \cvxt) (1 + c_2) = 1$; 
$(b)$ follows from Assumption \ref{assum:bdd_var} (unbiasedness of stochastic gradients);
$(c)$ follows from Assumption \ref{assum:bdd_var} (bounded variance of stochastic gradients, and independence of stochastic gradients across clients), and the generic sum of squares inequality in \cref{lem:sum_of_squares};
$(d)$ follows from \cref{assum:smoothness} ($\Lf$-smoothness of $f_i$) \cref{assum:bdd_hetero} (bounded heterogeneity across clients).

Finally, we choose $\momx = \mom$.
This concludes the proof of \eqref{eq:lem:NC_PL_mom_p_cons_errs_recursion}.
\end{proof}

\begin{proof}[Proof of \cref{lem:NC_PL_mom_induct_bd_cons_error_xy}]
Substituting \eqref{eq:lem:NC_PL_mom_p_cons_errs_recursion}, \eqref{eq:lem:NC_PL_mom_q_cons_errs_recursion} from Lemma \ref{lem:NC_PL_mom_cons_errs_recursion} in \eqref{eq:lem:NC_PL_mom_xy_cons_errs_recursion}, we get
\begin{equation}
    \begin{aligned}
        \CExytp & \leq \lcb 1+c_1 + \lp 1 + \mfrac{1}{c_1} \rp 6 \Lf^2 \mom \cvx^3 (\lrx^2 + \lry^2) \rcb \CExyt + \lp 1 + \mfrac{1}{c_1} \rp \cvx^2 (1 - \mom \cvx) \lp \lrx^2 \CEptm + \lry^2 \CEqtm \rp \\
        & \qquad + \lp 1 + \mfrac{1}{c_1} \rp \mom \cvx^3 \lb \lp \lrx^2 + \lry^2 \rp \sigma^2 \lp 1 + \mfrac{1}{n} \rp + 3 \lrx^2 \heterox^2 + \lry^2 \heteroy^2 \rb.
    \end{aligned}
    \label{eq_proof:lem:NC_PL_mom_induct_bd_1}
\end{equation}
Using $c_1 = \frac{\mom \cvx}{1 - \mom \cvx}$ in \eqref{eq_proof:lem:NC_PL_mom_induct_bd_1} gives us
\begin{align}
    \CExytp & \leq \lcb 1+c_1 + 6 \Lf^2 \cvx^2 (\lrx^2 + \lry^2) \rcb \CExyt + \frac{\cvx}{\mom} (1 - \mom \cvx) \lp \lrx^2 \CEptm + \lry^2 \CEqtm \rp \nn \\
    & \qquad + \cvx^2 \lb \lp \lrx^2 + \lry^2 \rp \sigma^2 \lp 1 + \mfrac{1}{n} \rp + 3 \lrx^2 \heterox^2 + \lry^2 \heteroy^2 \rb \nn \\
    & = \lp 1 + \theta \rp \CExyt + \frac{\cvx}{\mom} (1 - \mom \cvx) \lp \lrx^2 \CEptm + \lry^2 \CEqtm \rp + \Upsilon, \label{eq_proof:lem:NC_PL_mom_induct_bd_2}
\end{align}
where we define $\theta \triangleq c_1 + 6 \Lf^2 \cvx^2 (\lrx^2 + \lry^2)$.

Now, we proceed to prove the induction. For $k=1$, it follows from \eqref{eq_proof:lem:NC_PL_mom_induct_bd_2} that \eqref{eq:lem:NC_PL_mom_induct_bd_cons_error_xy} holds. Next, we assume the induction hypothesis in \eqref{eq:lem:NC_PL_mom_induct_bd_cons_error_xy} holds for some $k > 1$ (assuming $t-1-k \geq s \sync + 1$). We prove that it also holds for $k+1$.
\begin{align}
    \CExyt & \leq (1 + 2 k \theta) \Delta_{t-k}^{\bx,\by} + 2 k \mfrac{\cvx}{\mom} (1-\mom \cvx) \lp \lrx^2 \Delta_{t-k-1}^{\bdx} + \lry^2 \Delta_{t-k-1}^{\bdy} \rp + k^2 (1+\theta) \Upsilon \tag{Induction hypothesis} \\
    & \leq \lcb (1 + 2 k \theta) (1 + \theta) + 2 k \mfrac{\cvx}{\mom} (1-\mom \cvx) (\lrx^2 + \lry^2) 6 \Lf^2 \mom \cvx \rcb \Delta_{t-k-1}^{x,y} \tag{\cref{lem:NC_PL_mom_cons_errs_recursion}, \eqref{eq_proof:lem:NC_PL_mom_induct_bd_2}} \\
    & \quad + \lcb (1 + 2 k \theta) \frac{\cvx}{\mom} (1 - \mom \cvx) + 2 k \mfrac{\cvx}{\mom} (1-\mom \cvx)^2 \rcb \lp \lrx^2 \Delta_{t-k-2}^{\bdx} + \lry^2 \Delta_{t-k-2}^{\bdy} \rp \nn \\
    & \quad + \lb 1 + 2k \theta + k^2 (1 + \theta) \rb \Upsilon + 2 k \mfrac{\cvx}{\mom} (1-\mom \cvx) \mom \cvx \lb \lp \lrx^2 + \lry^2 \rp \sigma^2 \lp 1 + \mfrac{1}{n} \rp + 3 \lrx^2 \heterox^2 + 3 \lry^2 \heteroy^2 \rb \nn \\
    & \leq \lcb (1 + 2 k \theta) (1 + \theta) + 2 k (1-\mom \cvx) (\theta - c_1) \rcb \Delta_{t-k-1}^{x,y} \tag{see definition of $\theta$ in \cref{lem:NC_PL_mom_induct_bd_cons_error_xy}} \\
    & \quad + \lb 1 + 2 k \theta + 2 k (1-\mom \cvx) \rb \frac{\cvx}{\mom} (1 - \mom \cvx) \lp \lrx^2 \Delta_{t-k-2}^{\bdx} + \lry^2 \Delta_{t-k-2}^{\bdy} \rp \nn \\
    & \quad + \lb 1 + 2k \theta + k^2 (1 + \theta) + 2 k (1-\mom \cvx) \rb \Upsilon. \tag{see definition of $\Upsilon$ in \cref{lem:NC_PL_mom_induct_bd_cons_error_xy}}
\end{align}
Next, we see how the parameter choices in \cref{lem:NC_PL_mom_induct_bd_cons_error_xy} satisfy the induction hypothesis.
Basically, we need to satisfy the following three conditions:
\begin{align}
    \begin{aligned}
        (1 + 2 k \theta) (1 + \theta) + 2 k (1-\mom \cvx) (\theta - c_1) & \leq 1 + 2 (k+1) \theta, \\
        1 + 2 k \theta + 2 k (1-\mom \cvx) & \leq 2(k+1), \\
        1 + 2k \theta + k^2 (1 + \theta) + 2 k (1-\mom \cvx) & \leq (k+1)^2 (1+\theta).
    \end{aligned}
    \label{eq:NC_PL_mom_induct_bd_cons_error_xy:param_condition}
\end{align}
\begin{enumerate}
    \item The first condition in \eqref{eq:NC_PL_mom_induct_bd_cons_error_xy:param_condition} is equivalent to
    \begin{align}
        \theta + 2 k \theta^2 + 2 k (1-\mom \cvx) (\theta - c_1) & \leq 2 \theta. \label{eq:cond_theta_1}
    \end{align}
    Recall that in \cref{lem:NC_PL_mom_induct_bd_cons_error_xy}, $\theta - c_1 = 6 \Lf^2 \cvx^2 (\lry^2 + \lrx^2)$. 
    If $6 \Lf^2 \cvx^2 (\lry^2 + \lrx^2) \leq \min \{ c_1, \theta^2 \}$, a \textit{sufficient} condition for \eqref{eq:cond_theta_1} is
    \begin{align*}
        4 k \theta^2 \leq \theta \quad \Rightarrow \quad \theta \leq 1/4k.
    \end{align*}
    Since $\theta \leq 2 c_1$ and $c_1 \leq 2 \mom \cvx$ (if $\cvx \leq 1/(2 \mom)$), this is satisfied if $\cvx \leq \frac{1}{16 \mom k}$.
    Next, we verify that $6 \Lf^2 \cvx^2 (\lry^2 + \lrx^2) \leq \min \{ c_1, \theta^2 \}$ holds.
    \begin{itemize}
        \item $6 \Lf^2 \cvx^2 (\lry^2 + \lrx^2) \leq c_1$ follows from the condition $\cvx \leq \frac{\mom}{6 \Lf^2 (\lry^2 + \lrx^2)}$ (since $c_1 \geq \mom \cvx$).
        \item $6 \Lf^2 \cvx^2 (\lry^2 + \lrx^2) \leq \theta^2$ follows from the condition $\Lf^2 (\lry^2 + \lrx^2) \leq \frac{\mom^2}{6}$ (since $\theta \geq c_1 \geq \cvx \mom$).
    \end{itemize}
        
    \item The second condition in \eqref{eq:NC_PL_mom_induct_bd_cons_error_xy:param_condition} is equivalent to
    \begin{align*}
        2k (\theta - \mom \cvx) \leq 1.
    \end{align*}
    A \textit{sufficient} condition for this to be satisfied is $\theta \leq \frac{1}{2k}$, which, as seen above, is already satisfied if $\cvx \leq \frac{1}{16 \mom k}$.
    \item The third condition in \eqref{eq:NC_PL_mom_induct_bd_cons_error_xy:param_condition} is equivalent to
    \begin{align*}
        1 + 2k \theta + 2 k (1-\mom \cvx) & \leq 2k (1+\theta) + (1+\theta) \\
        \Leftrightarrow - 2k \mom \cvx & \leq \theta.
    \end{align*}
    which is trivially satisfied.
\end{enumerate}
Hence, the parameter choices in \cref{lem:NC_PL_mom_induct_bd_cons_error_xy} satisfy the induction hypothesis, which completes the proof.
\end{proof}

\begin{proof}[Proof of \cref{cor:NC_PL_mom_induct_bd_cons_error_xy}]
For $k = k_0$ such that $(t-k_0-1) \mod \sync = 0$, then by \cref{alg_NC_momentum} $$\Delta_{t-k_0-1}^{\bx, \by} = \Delta_{t-k_0-1}^{\bdx} = \Delta_{t-k_0-1}^{\bdy} = 0.$$
From \cref{lem:NC_PL_mom_cons_errs_recursion}, $\Delta_{t-k_0}^{\bx, \by} = 0$. 
Using this information in \cref{lem:NC_PL_mom_induct_bd_cons_error_xy}, we get
\begin{align*}
    \CExyt & \leq (1 + 2 k_0 \theta) \Delta_{t-k_0}^{x,y} + k_0^2 (1+\theta) \Upsilon \\
    & \leq (\sync - 1)^2 \cvx^2 \lb \lp \lrx^2 + \lry^2 \rp \sigma^2 \lp 1 + \mfrac{1}{n} \rp + 3 \lrx^2 \heterox^2 + 3 \lry^2 \heteroy^2 \rb. \tag{Using $\Upsilon$ from \cref{lem:NC_PL_mom_induct_bd_cons_error_xy}}
\end{align*}
\end{proof}

\newpage
\section{Nonconvex-Concave Functions: Local SGDA+ (\texorpdfstring{\cref{thm:NC_C}}{Theorem 1})} \label{app:NC_C}

\begin{algorithm}[ht]
\caption{Local SGDA+ \cite{mahdavi21localSGDA_aistats}}
\label{alg_local_SGDA_plus}
\begin{algorithmic}[1]
	\STATE{\textbf{Input: }{\small$\bx_0^i = \Tbx_0 = \bx_0, \by_0^i = \by_0$}, for all $i \in [n]$; step-sizes $\lrx, \lry$; $\sync$, $T$, $S, k=0$}
	\FOR[At all clients $i=1,\hdots, n$]{$t=0$ to $T-1$}
	    \STATE{Sample minibatch $\xiit$ from local data}
        \STATE{$\bxitp = \bxit - \lrx \Gx f_i (\bxit, \byit; \xiit)$}
        \STATE{$\byitp = \byit + \lry \Gy f_i (\Tbxk, \byit; \xiit)$}
        \IF{$t+1$ mod $\sync = 0$}
            \STATE{Clients send $\{ \bxitp, \byitp \}$ to the server}
            \STATE{Server computes averages $\bxtp \triangleq \frac{1}{n} \sumin \bxitp$, 
            $\bytp \triangleq \frac{1}{n} \sumin \byitp$, and sends to all the clients}
            \STATE{$\bxitp = \bxtp$, $\byitp = \bytp$, for all $i \in [n]$}
        \ENDIF
        \IF{$t+1$ mod $S = 0$}
            \STATE{Clients send $\{ \bxitp \}$ to the server}
            \STATE{$k \gets k+1$}
            \STATE{Server computes averages $\Tbxk \triangleq \frac{1}{n} \sumin \bxitp$, and sends to all the clients}
        \ENDIF
	\ENDFOR
	\STATE{\textbf{Return: }$\bbxT$ drawn uniformly at random from $\{ \bxt \}$, where $\bxt \triangleq \frac{1}{n} \sumin \bxit$}
\end{algorithmic}
\end{algorithm}

We organize this section as follows. First, in \cref{sec:NC_C_int_results} we present some intermediate results, which we use in the proof of \cref{thm:NC_C}. Next, in \cref{sec:NC_C_thm_proof}, we present the proof of \cref{thm:NC_C}, which is followed by the proofs of the intermediate results in \cref{sec:NC_C_int_results_proofs}.

\subsection{Intermediate Lemmas} \label{sec:NC_C_int_results}

\begin{lemma}
\label{lem:NC_C_Phi_smooth_decay_one_iter}
Suppose the local loss functions $\{ f_i \}$ satisfy Assumptions \ref{assum:smoothness}, \ref{assum:bdd_var}, \ref{assum:bdd_hetero}, \ref{assum:concavity}, \ref{assum:Lips_cont_x}.
Then, the iterates generated by \cref{alg_local_SGDA_plus} satisfy
\begin{align}
    \mbe \lb \Phi_{1/2\Lf} (\bxtp) \rb & \leq \mbe \lb \Phi_{1/2\Lf} (\bxt) \rb + \lrx^2 \Lf \lp G_x^2 + \frac{\sigma^2}{n} \rp + 2 \lrx \Lf^2 \CExyt \nn \\
    & \quad + 2 \lrx \Lf \mbe \lb \Phi(\bxt) - f(\bxt, \byt) \rb - \frac{\lrx}{8} \mbe \norm{\G \Phi_{1/2\Lf} (\bxt)}^2. \nn
\end{align}
where $\CExyt = \frac{1}{n} \sumin \mbe \lp \lnr \bxit - \bxt \rnr^2 + \lnr \byit - \byt \rnr^2 \rp$ is the synchronization error at time $t$.
\end{lemma}

Next, we bound the difference $\mbe \lb \Phi(\bxt) - f(\bxt, \byt) \rb$.

\begin{lemma}
\label{lem:NC_C_Phi_f_diff}
Suppose the local functions satisfy Assumptions \ref{assum:smoothness}, \ref{assum:bdd_var}, \ref{assum:bdd_hetero}, \ref{assum:Lips_cont_x}.
Further, suppose we choose the step-size $\lry$ such that $\lry \leq \frac{1}{8 \Lf \sync}$.
Then the iterates generated by \cref{alg_local_SGDA_plus} satisfy
\begin{align}
    \avgtT \mbe \lb \Phi(\bxt) - f(\bxt, \byt) \rb & \leq 2 \lrx G_x S \sqrt{G_x^2 + \frac{\sigma^2}{n}} + \frac{4 D}{\lry S} + \frac{20 \lry \sigma^2}{n} + 16 \lry^2 \Lf (\sync-1)^2 \lp \sigma^2 + \heteroy^2 \rp. \nn
\end{align}
\end{lemma}

\begin{lemma}
\label{lem:NC_C_consensus_error}
Suppose the local loss functions $\{ f_i \}$ satisfy Assumptions \ref{assum:smoothness}, \ref{assum:bdd_hetero},
and the stochastic oracles for the local
functions satisfy \cref{assum:bdd_var}.
Further, in \cref{alg_local_SGDA}, we choose step-sizes $\lrx, \lry \leq \frac{1}{8 \sync \Lf}$.
Then, the iterates $\{ \bxit, \byit \}$ generated by \cref{alg_local_SGDA_plus} satisfy
\begin{align}
    \frac{1}{T} \sumtT \CEyt & \triangleq \frac{1}{T} \sumtT \frac{1}{n} \sumin \mbe \lp \lnr \byit - \byt \rnr^2 \rp \leq 2 (\sync-1)^2 \lry^2 \lb \sigma^2 \lp 1 + \frac{1}{n} \rp + 3 \heteroy^2 \rb, \nn \\
    \frac{1}{T} \sumtT \CExt & \triangleq \frac{1}{T} \sumtT \frac{1}{n} \sumin \mbe \lp \lnr \bxit - \bxt \rnr^2 \rp \leq 2 (\sync-1)^2 \lb \lp \lrx^2 + \lry^2 \rp \sigma^2 \lp 1 + \frac{1}{n} \rp + 3\lp \lrx^2 \heterox^2 + \lry^2 \heteroy^2 \rp \rb. \nn
    \label{eq:lem:NC_C_consensus_error}
\end{align}
\end{lemma}

\subsection{Proof of \texorpdfstring{\cref{thm:NC_C}}{Theorem 3}}
\label{sec:NC_C_thm_proof}
For the sake of completeness, we first state the full statement of \cref{thm:NC_C} here.

\begin{theorem*}
Suppose the local loss functions $\{ f_i \}$ satisfy Assumptions \ref{assum:smoothness}, \ref{assum:bdd_var}, \ref{assum:bdd_hetero}, \ref{assum:concavity}, \ref{assum:Lips_cont_x}.
Further, let $\norm{\byt}^2 \leq D$ for all $t$.
Suppose the step-sizes $\lrx, \lry$ are chosen such that $\lrx, \lry \leq \frac{1}{8 \Lf \sync}$.
Then the iterates generated by \cref{alg_local_SGDA_plus} satisfy
\begin{align}
    \avgtT \mbe \norm{\G \Phi_{1/2\Lf} (\bxt)}^2 & \leq \frac{8 \widetilde{\Delta}_{\Phi}}{\lrx T} + 8 \lrx \Lf \lp G_x^2 + \frac{\sigma^2}{n} \rp + \frac{320 \lry \Lf \sigma^2}{n} + 16 \Lf \lb 2 \lrx G_x S \sqrt{G_x^2 + \frac{\sigma^2}{n}} + \frac{4 D}{\lry S} \rb \nn \\
    & \quad + 64 \Lf^2 (\sync-1)^2 \lb \lp \lrx^2 + \lry^2 \rp \sigma^2 \lp 1 + \frac{1}{n} \rp + 3 \lp \lrx^2 \heterox^2 + \lry^2 \heteroy^2 \rp + 4 \lry^2 \lp \sigma^2 + \heteroy^2 \rp \rb. \nn
\end{align}
With the following parameter values:
\begin{align*}
    \lrx = \Theta \lp \frac{n^{1/4}}{T^{3/4}} \rp, \qquad \lry = \Theta \lp \frac{n^{3/4}}{T^{1/4}} \rp, \qquad S = \Theta \lp \sqrt{\frac{T}{n}} \rp,
\end{align*}
we can further simplify to
\begin{align}
    & \avgtT \mbe \norm{\G \Phi_{1/2\Lf} (\bxt)}^2 \leq \mco \lp \frac{1}{(nT)^{1/4}} \rp + \mco \lp \frac{n^{1/4}}{T^{3/4}} \rp + \mco \lp \frac{n^{3/2} (\sync-1)^2}{T^{1/2}} \rp + \mco \lp (\sync-1)^2 \frac{\sqrt{n}}{T^{3/2}} \rp. \nn
\end{align}
\end{theorem*}

\begin{proof}
We sum the result in \cref{lem:NC_C_Phi_smooth_decay_one_iter} over $t = 0$ to $T-1$ and rearrange the terms to get
\begin{align}
    & \avgtT \mbe \norm{\G \Phi_{1/2\Lf} (\bxt)}^2 \leq \frac{8}{\lrx} \avgtT \lp \mbe \lb \Phi_{1/2\Lf} (\bxt) \rb - \mbe \lb \Phi_{1/2\Lf} (\bxtp) \rb \rp + 8 \lrx \Lf \lp G_x^2 + \frac{\sigma^2}{n} \rp \nn \\
    & \qquad \qquad \qquad \qquad \qquad \qquad + 16 \Lf \avgtT \mbe \lb \Phi(\bxt) - f(\bxt, \byt) \rb + 16 \Lf^2 \CExyt \nn \\
    & \leq \frac{8}{\lrx T} \lb \Phi_{1/2\Lf} (\bx_0) - \mbe \lb \Phi_{1/2\Lf} (\bx_T) \rb \rb + 8 \lrx \Lf \lp G_x^2 + \frac{\sigma^2}{n} \rp + 16 \Lf^2 \CExyt \nn \\
    & \quad + 16 \Lf \lb 2 \lrx G_x S \sqrt{G_x^2 + \frac{\sigma^2}{n}} + \frac{4 D}{\lry S} + \frac{20 \lry \sigma^2}{n} + 16 \lry^2 \Lf (\sync-1)^2 \lp \sigma^2 + \heteroy^2 \rp \rb \tag{\cref{lem:NC_C_Phi_f_diff}} \\
    & \leq \frac{8 \widetilde{\Delta}_{\Phi}}{\lrx T} + 8 \lrx \Lf \lp G_x^2 + \frac{\sigma^2}{n} \rp + \frac{320 \lry \Lf \sigma^2}{n} + 16 \Lf \lb 2 \lrx G_x S \sqrt{G_x^2 + \frac{\sigma^2}{n}} + \frac{4 D}{\lry S} \rb \nn \\
    & \quad + 64 \Lf^2 (\sync-1)^2 \lb \lp \lrx^2 + \lry^2 \rp \sigma^2 \lp 1 + \frac{1}{n} \rp + 3 \lp \lrx^2 \heterox^2 + \lry^2 \heteroy^2 \rp + 4 \lry^2 \lp \sigma^2 + \heteroy^2 \rp \rb, \tag{\cref{lem:NC_C_consensus_error}}
\end{align}
where $\widetilde{\Delta}_{\Phi} = \Phi_{1/2\Lf} (\bx_0) - \min_\bx \Phi_{1/2\Lf} (\bx)$.

If $D = 0$, we let $S=1$. Else, let $S = \sqrt{\frac{2 D}{\lrx \lry G_x \sqrt{G_x^2 + \sigma^2/n}}}$. Then we get
\begin{align}
    & \avgtT \mbe \norm{\G \Phi_{1/2\Lf} (\bxt)}^2 \leq \frac{8 \widetilde{\Delta}_{\Phi}}{\lrx T} + 8 \lrx \Lf \lp G_x^2 + \frac{\sigma^2}{n} \rp + \frac{320 \lry \Lf \sigma^2}{n} + 64 \Lf \sqrt{\frac{2 D \lrx G_x \sqrt{G_x^2 + \frac{\sigma^2}{n}}}{\lry}} \nn \\
    & \quad + 64 \Lf^2 (\sync-1)^2 \lb \lp \lrx^2 + \lry^2 \rp \sigma^2 \lp 1 + \frac{1}{n} \rp + 3 \lp \lrx^2 \heterox^2 + \lry^2 \heteroy^2 \rp + 4 \lry^2 \lp \sigma^2 + \heteroy^2 \rp \rb, \label{eq_proof:thm_NC_C_1}
\end{align}
For $\lry \leq 1$, the terms containing $\lry^2$ are of higher order, and we focus only on the other terms containing $\lry$, i.e., 
\begin{align*}
    64 \Lf \lb \frac{5 \lry \sigma^2}{n} + \sqrt{\frac{2 D \lrx G_x \sqrt{G_x^2 + \frac{\sigma^2}{n}}}{\lry}} \rb.
\end{align*}
To optimize these, we choose $\lry = \lp \frac{n}{10 \sigma^2} \rp^{2/3} \lp 2 D \lrx G_x \sqrt{G_x^2 + \frac{\sigma^2}{n}} \rp^{1/3}$. Substituting in \eqref{eq_proof:thm_NC_C_1}, we get
\begin{align}
    & \avgtT \mbe \norm{\G \Phi_{1/2\Lf} (\bxt)}^2 \leq \frac{8 \widetilde{\Delta}_{\Phi}}{\lrx T} + 8 \lrx \Lf \lp G_x^2 + \frac{\sigma^2}{n} \rp + 320 \Lf \lp 10 \frac{\sigma^2}{n} D \lrx G_x \sqrt{G_x^2 + \frac{\sigma^2}{n}} \rp^{1/3} \nn \\
    & \quad + 200 \Lf^2 (\sync-1)^2 \lb 4 \lrx^{2/3} \lp \frac{n}{10 \sigma^2} \rp^{4/3} \lp 2 D G_x \sqrt{G_x^2 + \frac{\sigma^2}{n}} \rp^{2/3} \lp \sigma^2 + \heteroy^2 \rp + \lrx^2 \lp \sigma^2 + \heterox^2 \rp \rb,
    \label{eq_proof:thm_NC_C_2}
\end{align}
Again, we ignore the higher order terms of $\lrx$, and only focus on
\begin{align*}
    \frac{8 \widetilde{\Delta}_{\Phi}}{\lrx T} + 320 \Lf \lp 10 \frac{\sigma^2}{n} D \lrx G_x \sqrt{G_x^2 + \frac{\sigma^2}{n}} \rp^{1/3}.
\end{align*}
With $\lrx = \lp \frac{3}{40 \Lf T} \rp^{3/4} \lp 10 \frac{\sigma^2}{n} D G_x \sqrt{G_x^2 + \frac{\sigma^2}{n}} \rp^{-1/4}$,
and absorbing numerical constants inside $\mco (\cdot)$ we get,
\begin{align}
    & \avgtT \mbe \norm{\G \Phi_{1/2\Lf} (\bxt)}^2 \leq \mco \lp \lp \sigma^2 D G_x \sqrt{G_x^2 + \frac{\sigma^2}{n}} \rp^{1/4} \frac{\Lf^{3/4}}{(nT)^{1/4}} \rp \nn \\
    & \quad + \mco \lp \frac{\Lf^{1/4}}{T^{3/4}} \lp \frac{\sigma^2}{n} D G_x \sqrt{G_x^2 + \frac{\sigma^2}{n}} \rp^{-1/4} \lp G_x^2 + \frac{\sigma^2}{n} \rp \rp \nn \\
    & \quad + \mco \lp \frac{\Lf^{3/2} (\sync-1)^2}{T^{1/2}} \lp \frac{n}{\sigma^2} \rp^{3/2} \lp D G_x \sqrt{G_x^2 + \frac{\sigma^2}{n}} \rp^{1/2} \lp \sigma^2 + \heteroy^2 \rp \rp, \nn \\
    & \quad + \mco \lp (\sync-1)^2 \lp \sigma^2 + \heterox^2 \rp \frac{\sqrt{\Lf}}{T^{3/2}} \lp \frac{\sigma^2}{n} D G_x \sqrt{G_x^2 + \frac{\sigma^2}{n}} \rp^{-1/2} \rp, \label{eq_proof:thm_NC_C_3} \\
    & \leq \mco \lp \frac{\sigma^2 + D + G_x^2}{(nT)^{1/4}} \rp + \mco \lp \frac{n^{1/4}}{T^{3/4}} \rp + \mco \lp \frac{n^{3/2} (\sync-1)^2}{T^{1/2}} \rp + \mco \lp (\sync-1)^2 \frac{\sqrt{n}}{T^{3/2}} \rp,
    \label{eq_proof:thm_NC_C_4}
\end{align}
where in \eqref{eq_proof:thm_NC_C_4}, we have dropped all the problem-specific parameters, to show dependence only on $\sync, n, T$.

Lastly, we specify the algorithm parameters in terms of $n,T$.
\begin{itemize}
    \item $\lrx = \lp \frac{3}{40 \Lf T} \rp^{3/4} \lp 10 \frac{\sigma^2}{n} D G_x \sqrt{G_x^2 + \frac{\sigma^2}{n}} \rp^{-1/4} = \Theta \lp \frac{n^{1/4}}{T^{3/4}} \rp$,
    \item $\lry = \lp \frac{n}{10 \sigma^2} \rp^{2/3} \lp 2 D \lrx G_x \sqrt{G_x^2 + \frac{\sigma^2}{n}} \rp^{1/3} = \Theta \lp \frac{n^{3/4}}{T^{1/4}} \rp$,
    \item $S = \sqrt{\frac{2 D}{\lrx \lry G_x \sqrt{G_x^2 + \sigma^2/n}}} = \Theta \lp \sqrt{\frac{T}{n}} \rp$.
\end{itemize}

\end{proof}

\begin{proof}[Proof of \cref{cor:NC_C_comm_cost}]
We assume $T \geq n^7$.
To reach an $\epsilon$-accurate point, i.e., $\bbxT$ such that $\mbe \lnr \G \Phi_{1/2\Lf} (\bbxT) \rnr \leq \epsilon$, we need
\begin{align*}
    \mbe \lnr \G \Phi_{1/2\Lf} (\bbxT) \rnr & \leq \lb \frac{1}{T} \sumtT \mbe \lnr \G \Phi_{1/2\Lf} (\bxt) \rnr^2 \rb^{1/2} \nn \\
    & \leq \mco \lp \frac{1}{(nT)^{1/8}} \rp + \mco \lp \frac{n^{1/8}}{T^{3/8}} \rp + \mco \lp \frac{n^{3/4} (\sync-1)}{T^{1/4}} \rp + \mco \lp (\sync-1) \frac{n^{1/4}}{T^{3/4}} \rp.
\end{align*}
We can choose $\sync \leq \mco \lp \frac{T^{1/8}}{n^{7/8}} \rp$ without affecting the convergence rate $\mco \lp \frac{1}{(nT)^{1/8}} \rp$.
In that case, we need $T = \mco \lp \frac{1}{n \epsilon^8} \rp$ iterations to reach an $\epsilon$-accurate point.
And the minimum number of communication rounds is 
$$\mc O \lp \frac{T}{\sync} \rp = \mc O \lp (n T)^{7/8} \rp = \mco \lp \frac{1}{\epsilon^7} \rp.$$
\end{proof}

\subsection{Proofs of the Intermediate Lemmas}
\label{sec:NC_C_int_results_proofs}

\begin{proof}[Proof of \cref{lem:NC_C_Phi_smooth_decay_one_iter}]
We borrow the proof steps from \cite{lin_GDA_icml20, mahdavi21localSGDA_aistats}. Define $\widetilde{\bx}_{t} = \argmin_\bx \Phi (\bx) + \Lf \norm{\bx - \bxt}^2$, then using the definition of $\Phi_{1/2\Lf}$, we get
\begin{align}
    \Phi_{1/2\Lf} (\bxtp) & \triangleq \min_\bx \Phi (\bx) + \Lf \norm{\bx - \bxtp}^2 \nn \\
    & \leq \Phi (\widetilde{\bx}_t) + \Lf \norm{\widetilde{\bx}_t - \bxtp}^2.
    \label{eq:lem:NC_C_Phi_smooth_decay_one_iter_1}
\end{align}
Using the $\bxit$ updates in \cref{alg_local_SGDA_plus},
\begin{align}
    & \mbe \norm{\widetilde{\bx}_t - \bxtp}^2 = \mbe \norm{\widetilde{\bx}_t - \bxt + \lrx \avgin \Gx f_i (\bxit, \byit; \xiit)}^2 \nn \\
    &= \mbe \norm{\widetilde{\bx}_t - \bxt}^2 + \lrx^2 \mbe \norm{\avgin \Gx f_i (\bxit, \byit; \xiit)}^2 + 2 \lrx \mbe \lan \widetilde{\bx}_t - \bxt, \avgin \Gx f_i (\bxit, \byit) \ran \tag{\cref{assum:bdd_var}} \\
    & \leq \mbe \norm{\widetilde{\bx}_t - \bxt}^2 + \lrx^2 \mbe \norm{\avgin \Gx f_i (\bxit, \byit)}^2 + \frac{\lrx^2 \sigma^2}{n} + 2 \lrx \mbe \lan \widetilde{\bx}_t - \bxt, \Gx f (\bxt, \byt) \ran \nn \\
    & \quad + \lrx \mbe \lb \frac{\Lf}{2} \norm{\widetilde{\bx}_t - \bxt}^2 + \frac{2}{\Lf} \norm{\avgin \Gx f_i (\bxit, \byit) - \Gx f (\bxt, \byt)}^2 \rb \tag{\cref{lem:Young}} \\
    & \leq \mbe \norm{\widetilde{\bx}_t - \bxt}^2 + \lrx^2 \lp \mbe \norm{\avgin \Gx f_i (\bxit, \byit)}^2 + \frac{\sigma^2}{n} \rp + 2 \lrx \mbe \lan \widetilde{\bx}_t - \bxt, \Gx f (\bxt, \byt) \ran \nn \\
    & \quad + \frac{\lrx \Lf}{2} \norm{\widetilde{\bx}_t - \bxt}^2 + 2 \lrx \Lf \CExyt
    \label{eq:lem:NC_C_Phi_smooth_decay_one_iter_2}
\end{align}
where \eqref{eq:lem:NC_C_Phi_smooth_decay_one_iter_2} follows from \cref{assum:smoothness}.
Next, we bound the inner product in \eqref{eq:lem:NC_C_Phi_smooth_decay_one_iter_2}.
Using $\Lf$-smoothness of $f$ (\cref{assum:smoothness}):
\begin{align}
    \mbe \lan \widetilde{\bx}_t - \bxt, \Gx f (\bxt, \byt) \ran & \leq \mbe \lb f(\widetilde{\bx}_t, \byt) - f(\bxt, \byt) + \frac{\Lf}{2} \norm{\widetilde{\bx}_t - \bxt}^2 \rb \nn \\
    & \leq \mbe \lb \Phi(\widetilde{\bx}_t) - f(\bxt, \byt) + \frac{\Lf}{2} \norm{\widetilde{\bx}_t - \bxt}^2 \rb \nn \\
    & = \mbe \lb \Phi(\widetilde{\bx}_t) + \Lf \norm{\widetilde{\bx}_t - \bxt}^2 \rb - \mbe f(\bxt, \byt) - \frac{\Lf}{2} \mbe \norm{\widetilde{\bx}_t - \bxt}^2 \nn \\
    & \leq \mbe \lb \Phi(\bxt) + \Lf \norm{\bxt - \bxt}^2 \rb - \mbe f(\bxt, \byt) - \frac{\Lf}{2} \mbe \norm{\widetilde{\bx}_t - \bxt}^2 \tag{by definition of $\widetilde{\bx}_t$} \\
    & \leq \mbe \lb \Phi(\bxt) - f(\bxt, \byt) - \frac{\Lf}{2} \norm{\widetilde{\bx}_t - \bxt}^2 \rb. \label{eq:lem:NC_C_Phi_smooth_decay_one_iter_3}
\end{align}
Substituting the bounds in \eqref{eq:lem:NC_C_Phi_smooth_decay_one_iter_2} and \eqref{eq:lem:NC_C_Phi_smooth_decay_one_iter_3} into \eqref{eq:lem:NC_C_Phi_smooth_decay_one_iter_1}, we get
\begin{align}
    \mbe \lb \Phi_{1/2\Lf} (\bxtp) \rb &\leq \mbe \Phi (\widetilde{\bx}_t) + \Lf \lb \mbe \norm{\widetilde{\bx}_t - \bxt}^2 + \lrx^2 \lp G_x^2 + \frac{\sigma^2}{n} \rp \rb + \frac{\lrx \Lf^2}{2} \norm{\widetilde{\bx}_t - \bxt}^2 + 2 \lrx \Lf^2 \CExyt \nn \\
    & \qquad + 2 \lrx \Lf \mbe \lb \Phi(\bxt) - f(\bxt, \byt) - \frac{\Lf}{2} \norm{\widetilde{\bx}_t - \bxt}^2 \rb \nn \\
    & \leq \mbe \lb \Phi_{1/2\Lf} (\bxt) \rb + \lrx^2 \Lf \lp G_x^2 + \frac{\sigma^2}{n} \rp + 2 \lrx \Lf^2 \CExyt - \frac{\lrx \Lf^2}{2} \mbe \norm{\widetilde{\bx}_t - \bxt}^2 \nn \\
    & \quad + 2 \lrx \Lf \mbe \lb \Phi(\bxt) - f(\bxt, \byt) \rb \nn \\
    & = \mbe \lb \Phi_{1/2\Lf} (\bxt) \rb + \lrx^2 \Lf \lp G_x^2 + \frac{\sigma^2}{n} \rp + 2 \lrx \Lf^2 \CExyt - \frac{\lrx}{8} \mbe \norm{\G \Phi_{1/2\Lf} (\bxt)}^2 \nn \\
    & \quad + 2 \lrx \Lf \mbe \lb \Phi(\bxt) - f(\bxt, \byt) \rb. \nn
\end{align}
where we use the result $\G \Phi_{1/2\Lf} (\bx) = 2 \Lf (\bx - \widetilde{\bx})$ from Lemma~2.2 in \cite{davis19wc_siam}. This concludes the proof.
\end{proof}

\begin{proof}[Proof of \cref{lem:NC_C_Phi_f_diff}]
Let $t = kS + 1$ to $(k+1) S$, where $k = \lfloor T/S \rfloor$ is a positive integer.
Let $\Tbxk$ is the latest snapshot iterate in \cref{alg_local_SGDA_plus}. Then
\begin{align}
    & \mbe \lb \Phi(\bxt) - f(\bxt, \byt) \rb \nn \\
    &= \mbe \lb f(\bxt, \by^*(\bxt)) - f(\Tbxk, \by^*(\Tbxk)) \rb + \mbe \lb f(\Tbxk, \by^*(\Tbxk)) - f(\Tbxk, \byt) \rb + \mbe \lb f(\Tbxk, \byt) - f(\bxt, \byt) \rb \nn \\
    & \leq \mbe \lb f(\bxt, \by^*(\bxt)) - f(\Tbxk, \by^*(\bxt)) \rb + \mbe \lb f(\Tbxk, \by^*(\Tbxk)) - f(\Tbxk, \byt) \rb + G_x \mbe \norm{\Tbxk - \bxt} \nn \\
    & \leq 2 G_x \mbe \norm{\Tbxk - \bxt} + \mbe \lb f(\Tbxk, \by^*(\Tbxk)) - f(\Tbxk, \byt) \rb. \label{eq_proof:lem:NC_C_Phi_f_diff_1}
\end{align}
where, \eqref{eq_proof:lem:NC_C_Phi_f_diff_1} follows from $G_x$-Lipschitz continuity of $f(\cdot, \by)$ (\cref{assum:Lips_cont_x}), and since $\by^*(\cdot) \in \argmax_\by f(\cdot, \by)$.
Next, we see that
\begin{align*}
    & \mbe G_x \norm{\Tbxk - \bxt} \leq \lrx S G_x \sqrt{G_x^2 + \frac{\sigma^2}{n}},
\end{align*}
This is because $\bxit$ can be updated at most $S$ times between two consecutive updates of $\Tbx$. Also, at any time $t$,
\begin{align*}
    \mbe \norm{\avgin \Gx f_i (\bxit, \byit; \xiit)}^2 & = \mbe \norm{\avgin \lb \Gx f_i (\bxit, \byit; \xiit) - \Gx f_i (\bxit, \byit) \rb}^2 + \mbe \norm{\avgin \Gx f_i (\bxit, \byit)}^2 \\
    & \leq \frac{\sigma}{n} + G_x^2,
\end{align*}
where the expectation is conditioned on the past.
Therefore, from \eqref{eq_proof:lem:NC_C_Phi_f_diff_1} we get
\begin{align}
    & \sum_{t=kS+1}^{(k+1)S} \mbe \lb \Phi(\bxt) - f(\bxt, \byt) \rb \leq 2 \lrx G_x S^2 \sqrt{G_x^2 + \frac{\sigma^2}{n}} + \sum_{t=kS+1}^{(k+1)S} \mbe \lb f(\Tbxk, \by^*(\Tbxk)) - f(\Tbxk, \byt) \rb. \label{eq_proof:lem:NC_C_Phi_f_diff_2}
\end{align}
Next, we bound $\mbe \lb f(\Tbxk, \by^*(\Tbxk)) - f(\Tbxk, \byt) \rb$.
Since in localSGDA+, during the updates of $\{ \byit \}$, for $t = kS + 1$ to $(k+1) S$, the corresponding $\bx$ remains constant at $\Tbxk$. Therefore, for $t = kS + 1$ to $(k+1) S$, the $\by$ updates behave like maximizing a concave function $f(\Tbxk, \cdot)$.
With $\{ \byit \}$ being averaged every $\sync$ iterations, these $\byit$ updates can be interpreted as iterates of a Local Stochastic Gradient Ascent (Local SGA) algorithm.

Using \cref{lem:local_SGD_khaled} for Local SGD (\cref{alg_local_SGD}), and modifying the result for concave function maximization, we get
\begin{align*}
    \frac{1}{S} \sum_{t=kS+1}^{(k+1)S} \mbe \lb f(\Tbxk, \by^*(\Tbxk)) - f(\Tbxk, \byt) \rb & \leq \frac{4 \norm{\by_{kS+1} - \by^*(\Tbxk)}^2}{\lry S} + \frac{20 \lry \sigma^2}{n} + 16 \lry^2 \Lf (\sync-1)^2 \lp \sigma^2 + \heteroy^2 \rp \nn \\
    & \leq \underbrace{\frac{4 D}{\lry S} + \frac{20 \lry \sigma^2}{n}}_{\substack{\text{error with full} \\ \text{synchronization}}} + \underbrace{16 \lry^2 \Lf (\sync-1)^2 \lp \sigma^2 + \heteroy^2 \rp}_{\text{error due to local updates}}. \nn
\end{align*}
Substituting this bound in \eqref{eq_proof:lem:NC_C_Phi_f_diff_2}, we get
\begin{align*}
    & \sum_{t=kS+1}^{(k+1)S} \mbe \lb \Phi(\bxt) - f(\bxt, \byt) \rb \leq 2 \lrx G_x S^2 \sqrt{G_x^2 + \frac{\sigma^2}{n}} + \frac{4 D}{\lry} + \frac{20 \lry \sigma^2 S}{n} + 16 S \lry^2 \Lf (\sync-1)^2 \lp \sigma^2 + \heteroy^2 \rp.
\end{align*}
Summing over $k = 0$ to $T/S - 1$, we get
\begin{align*}
    & \frac{1}{T} \sum_{k=0}^{T/S-1} \sum_{t=kS+1}^{(k+1)S} \mbe \lb \Phi(\bxt) - f(\bxt, \byt) \rb \leq 2 \lrx G_x S \sqrt{G_x^2 + \frac{\sigma^2}{n}} + \frac{4 D}{\lry S} + \frac{20 \lry \sigma^2}{n} + 16 \lry^2 \Lf (\sync-1)^2 \lp \sigma^2 + \heteroy^2 \rp.
\end{align*}
\end{proof}

\begin{proof}[Proof of \cref{lem:NC_C_consensus_error}]
The proof follows analogously to the proof of \cref{lem:NC_PL_consensus_error}.
\end{proof}

\newpage
\section{Nonconvex-One-Point-Concave Functions: Local SGDA+ (\texorpdfstring{\cref{thm:NC_1PC}}{Theorem 4})} \label{app:NC_1PC}

The proof of \cref{thm:NC_1PC} is similar to the proof of \cref{thm:NC_C}.
We organize this section as follows. First, in \cref{sec:NC_1PC_int_results} we present some intermediate results, which we use in the proof of \cref{thm:NC_1PC}. Next, in \cref{sec:NC_1PC_thm_proof}, we present the proof of \cref{thm:NC_1PC}, which is followed by the proofs of the intermediate results in \cref{sec:NC_1PC_int_results_proofs}.
In \cref{app:NC_1PC_tau_1}, we prove convergence for the full synchronized Local SGDA+.

\subsection{Intermediate Lemmas} \label{sec:NC_1PC_int_results}

The main difference with the nonconvex-concave problem is the bound on the difference $\mbe \lb \Phi(\bxt) - f(\bxt, \byt) \rb$.
In case of concave functions, as we see in \cref{lem:NC_C_Phi_f_diff}, this difference can be bounded using standard results for Local SGD (\cref{lem:local_SGD_khaled}), which have a linear speedup with the number of clients $n$ (notice the $\frac{\lry \sigma^2}{n}$ term in \cref{lem:NC_C_Phi_f_diff}).
The corresponding result for minimization of smooth one-point-convex function using local SGD is an open problem.
Recent works on deterministic and stochastic quasar-convex problems (of which one-point-convex functions are a special case) \cite{gasnikov17acc_quasar_convex_arxiv, hinder20near_opt_star_convex_colt, jin20quasar_convex_arxiv} have achieved identical (within multiplicative constants) convergence rates, as smooth convex functions, for this more general class of functions, using SGD.
This leads us to conjecture that local SGD should achieve identical communication savings, along with linear speedup (as in \cref{lem:local_SGD_khaled}), for one-point-convex problems.
However, proving this claim formally remains an open problem.

In absence of this desirable result, we bound $\mbe \lb \Phi(\bxt) - f(\bxt, \byt) \rb$ in the next result, but without any linear speedup in $n$.

\begin{lemma}
\label{lem:NC_1PC_Phi_f_diff}
Suppose the local functions satisfy Assumptions \ref{assum:smoothness}, \ref{assum:bdd_var}, \ref{assum:bdd_hetero}, \ref{assum:Lips_cont_x}, \ref{assum:1pc_y}.
Further, suppose we choose the step-size $\lry$ such that $\lry \leq \frac{1}{8 \Lf \sync}$.
Then the iterates generated by \cref{alg_local_SGDA_plus} satisfy
\begin{align}
    \avgtT \mbe \lb \Phi(\bxt) - f(\bxt, \byt) \rb & \leq 2 \lrx G_x S \sqrt{G_x^2 + \frac{\sigma^2}{n}} + \frac{4 D}{\lry S} + 20 \lry \sigma^2 + 16 \lry^2 \Lf (\sync-1)^2 \lp \sigma^2 + \heteroy^2 \rp. \nn
\end{align}
\end{lemma}

\subsection{Proof of \texorpdfstring{\cref{thm:NC_1PC}}{Theorem 4}}
\label{sec:NC_1PC_thm_proof}
For the sake of completeness, we first state the full statement of \cref{thm:NC_1PC} here.

\begin{theorem*}
Suppose the local loss functions $\{ f_i \}$ satisfy Assumptions \ref{assum:smoothness}, \ref{assum:bdd_var}, \ref{assum:bdd_hetero}, \ref{assum:Lips_cont_x}, \ref{assum:1pc_y}.
Further, let $\norm{\byt}^2 \leq D$ for all $t$.
Suppose the step-size $\lry$ is chosen such that $\lry \leq \frac{1}{8 \Lf \sync}$.
Then the output $\bbxT$ of \cref{alg_local_SGDA_plus} satisfies
\begin{equation}
    \begin{aligned}
        \avgtT \mbe \norm{\G \Phi_{1/2\Lf} (\bxt)}^2 & \leq \mco \lp \frac{\widetilde{\Delta}_{\Phi}}{\lrx T} + \lrx \Lf \lp G_x^2 + \frac{\sigma^2}{n} \rp + \lry \Lf \sigma^2 + \Lf \lb \lrx G_x S \sqrt{G_x^2 + \frac{\sigma^2}{n}} + \frac{D}{\lry S} \rb \rp \\
        & \quad + \mco \lp \Lf^2 (\sync-1)^2 \lb \lp \lrx^2 + \lry^2 \rp \sigma^2 \lp 1 + \frac{1}{n} \rp + \lp \lrx^2 \heterox^2 + \lry^2 \heteroy^2 \rp + \lry^2 \lp \sigma^2 + \heteroy^2 \rp \rb \rp,
    \end{aligned}
    \label{eq:thm:NC_1PC}
\end{equation}
where {\small$\widetilde{\Delta}_{\Phi} \triangleq \Phi_{1/2 \Lf} (\bx_0) - \min_\bx \Phi_{1/2 \Lf} (\bx)$}.
With the following parameter values:
\begin{align*}
    \lrx = \Theta \lp \frac{1}{T^{3/4}} \rp, \qquad \lry = \Theta \lp \frac{1}{T^{1/4}} \rp, \qquad S = \Theta \lp \sqrt{T} \rp,
\end{align*}
we get
\begin{align}
    & \avgtT \mbe \norm{\G \Phi_{1/2\Lf} (\bxt)}^2 \leq \mco \lp \frac{1}{T^{1/4}} \rp + \mco \lp \frac{1}{T^{3/4}} \rp + \mco \lp \frac{(\sync-1)^2}{T^{1/2}} \rp + \mco \lp \frac{(\sync-1)^2}{T^{3/2}} \rp.
    \label{eq:thm:NC_1PC_conv_rate}
\end{align}
\end{theorem*}

\begin{cor}
\label{cor:NC_1PC_comm_cost}
To reach an $\epsilon$-accurate point, i.e., $\bx$ such that $\mbe \| \G \Phi_{1/2\Lf} (\bx) \| \leq \epsilon$,
the stochastic gradient complexity of \cref{alg_local_SGDA_plus} is $\mco (1/\epsilon^8)$.
The number of communication rounds required for the same is $T/\sync = \mco ( 1/\epsilon^{7} )$.
\end{cor}

\begin{remark}
Note that the only difference between the convergence rates for NC-1PC functions in \eqref{eq:thm:NC_1PC_conv_rate}, and for NC-C functions in \eqref{eq:thm:NC_C_conv_rate} is the absence of $n$ from the leading $\mco (1/T^{1/4})$ term.
This implies we do not observe a linear speedup in $n$ in this case.
As stated earlier, this limitation stems from the fact that even for simple minimization of one-point-convex functions, proving linear speedup in convergence rate in the presence of local updates at the clients is an open problem.
\end{remark}

\begin{proof}
We sum the result in \cref{lem:NC_C_Phi_smooth_decay_one_iter} over $t = 0$ to $T-1$ and rearrange the terms to get
\begin{align}
    & \avgtT \mbe \norm{\G \Phi_{1/2\Lf} (\bxt)}^2 \leq \frac{8}{\lrx} \avgtT \lp \mbe \lb \Phi_{1/2\Lf} (\bxt) \rb - \mbe \lb \Phi_{1/2\Lf} (\bxtp) \rb \rp + 8 \lrx \Lf \lp G_x^2 + \frac{\sigma^2}{n} \rp \nn \\
    & \qquad \qquad \qquad \qquad \qquad \qquad + 16 \Lf \avgtT \mbe \lb \Phi(\bxt) - f(\bxt, \byt) \rb + 16 \Lf^2 \CExyt \nn \\
    & \leq \frac{8}{\lrx T} \lb \Phi_{1/2\Lf} (\bx_0) - \mbe \lb \Phi_{1/2\Lf} (\bx_T) \rb \rb + 8 \lrx \Lf \lp G_x^2 + \frac{\sigma^2}{n} \rp + 16 \Lf^2 \CExyt \nn \\
    & \quad + 16 \Lf \lb 2 \lrx G_x S \sqrt{G_x^2 + \frac{\sigma^2}{n}} + \frac{4 D}{\lry S} + 20 \lry \sigma^2 + 16 \lry^2 \Lf (\sync-1)^2 \lp \sigma^2 + \heteroy^2 \rp \rb \tag{\cref{lem:NC_1PC_Phi_f_diff}} \\
    & \leq \frac{8 \widetilde{\Delta}_{\Phi}}{\lrx T} + 8 \lrx \Lf \lp G_x^2 + \frac{\sigma^2}{n} \rp + 320 \lry \Lf \sigma^2 + 16 \Lf \lb 2 \lrx G_x S \sqrt{G_x^2 + \frac{\sigma^2}{n}} + \frac{4 D}{\lry S} \rb \nn \\
    & \quad + 32 \Lf^2 (\sync-1)^2 \lb \lp \lrx^2 + \lry^2 \rp \sigma^2 \lp 1 + \frac{1}{n} \rp + 3 \lp \lrx^2 \heterox^2 + \lry^2 \heteroy^2 \rp + 8 \lry^2 \lp \sigma^2 + \heteroy^2 \rp \rb, \tag{\cref{lem:NC_PL_consensus_error}}
\end{align}
where $\widetilde{\Delta}_{\Phi} = \Phi_{1/2\Lf} (\bx_0) - \min_\bx \Phi_{1/2\Lf} (\bx)$.
Following similar technique as in the proof of \cref{thm:NC_C}, using the following parameter values,
\begin{align*}
    S = \Theta \lp \sqrt{T} \rp, \qquad \lrx = \Theta \lp \frac{1}{T^{3/4}} \rp, \qquad \lry = \Theta \lp \frac{1}{T^{1/4}} \rp,
\end{align*}
we get the following bound.
\begin{align}
    & \avgtT \mbe \norm{\G \Phi_{1/2\Lf} (\bxt)}^2 \leq \mco \lp \frac{\sigma^2 + D + G_x^2}{T^{1/4}} \rp + \mco \lp \frac{1}{T^{3/4}} \rp + \mco \lp \frac{(\sync-1)^2}{T^{1/2}} \rp + \mco \lp \frac{(\sync-1)^2}{T^{3/2}} \rp,
    \label{eq_proof:thm_NC_1PC_4}
\end{align}
which completes the proof
\end{proof}

\begin{proof}[Proof of \cref{cor:NC_1PC_comm_cost}]
To reach an $\epsilon$-accurate point, i.e., $\bx$ such that $\mbe \lnr \G \Phi_{1/2\Lf} (\bx) \rnr \leq \epsilon$, we need
\begin{align*}
    \mbe \lnr \G \Phi_{1/2\Lf} (\bbxT) \rnr & \leq \lb \frac{1}{T} \sumtT \mbe \lnr \G \Phi_{1/2\Lf} (\bxt) \rnr^2 \rb^{1/2} \nn \\
    & \leq \mco \lp \frac{1}{T^{1/8}} \rp + \mco \lp \frac{1}{T^{3/8}} \rp + \mco \lp \frac{\sync-1}{T^{1/4}} \rp + \mco \lp \frac{\sync-1}{T^{3/4}} \rp.
\end{align*}
We can choose $\sync \leq \mco \lp T^{1/8} \rp$ without affecting the convergence rate $\mco \lp \frac{1}{T^{1/8}} \rp$.
In that case, we need $T = \mco \lp \frac{1}{\epsilon^8} \rp$ iterations to reach an $\epsilon$-accurate point.
And the minimum number of communication rounds is 
$$\mc O \lp \frac{T}{\sync} \rp = \mc O \lp T^{7/8} \rp = \mco \lp \frac{1}{\epsilon^7} \rp.$$
\end{proof}

\subsection{Proofs of the Intermediate Lemmas}
\label{sec:NC_1PC_int_results_proofs}

\begin{proof}[Proof of \cref{lem:NC_1PC_Phi_f_diff}]
The proof proceeds the same way as for \cref{lem:NC_C_Phi_f_diff}.
Let $t = kS + 1$ to $(k+1) S$, where $k = \lfloor T/S \rfloor$ is a positive integer.
Let $\Tbxk$ is the latest snapshot iterate in \cref{alg_local_SGDA_plus}. From \eqref{eq_proof:lem:NC_C_Phi_f_diff_2}, we get
\begin{align}
    & \sum_{t=kS+1}^{(k+1)S} \mbe \lb \Phi(\bxt) - f(\bxt, \byt) \rb \leq 2 \lrx G_x S^2 \sqrt{G_x^2 + \frac{\sigma^2}{n}} + \sum_{t=kS+1}^{(k+1)S} \mbe \lb f(\Tbxk, \by^*(\Tbxk)) - f(\Tbxk, \byt) \rb. \label{eq_proof:lem:NC_1PC_Phi_f_diff_2}
\end{align}
Next, we bound $\mbe \lb f(\Tbxk, \by^*(\Tbxk)) - f(\Tbxk, \byt) \rb$.
Since in \cref{alg_local_SGDA_plus}, during the updates of $\{ \byit \}$, for $t = kS + 1$ to $(k+1) S$, the corresponding $\bx$ remains constant at $\Tbxk$. Therefore, for $t = kS + 1$ to $(k+1) S$, the $\by$ updates behave like maximizing a concave function $f(\Tbxk, \cdot)$.
With $\{ \byit \}$ being averaged every $\sync$ iterations, these $\byit$ updates can be interpreted as iterates of a Local Stochastic Gradient Ascent (Local SGA) (\cref{alg_local_SGD}).

However, since the function is no longer concave, but one-point-concave, we lose the linear speedup in \cref{lem:local_SGD_khaled}, and get
\begin{align*}
    \frac{1}{S} \sum_{t=kS+1}^{(k+1)S} \mbe \lb f(\Tbxk, \by^*(\Tbxk)) - f(\Tbxk, \byt) \rb & \leq \frac{4 \norm{\by_{kS+1} - \by^*(\Tbxk)}^2}{\lry S} + 20 \lry \sigma^2 + 16 \lry^2 \Lf (\sync-1)^2 \lp \sigma^2 + \heteroy^2 \rp \nn \\
    & \leq \underbrace{\frac{4 D}{\lry S} + 20 \lry \sigma^2}_{\substack{\text{error with full} \\ \text{synchronization}}} + \underbrace{16 \lry^2 \Lf (\sync-1)^2 \lp \sigma^2 + \heteroy^2 \rp}_{\text{error due to local updates}}. \nn
\end{align*}
Substituting this bound in \eqref{eq_proof:lem:NC_1PC_Phi_f_diff_2}, we get
\begin{align*}
    & \sum_{t=kS+1}^{(k+1)S} \mbe \lb \Phi(\bxt) - f(\bxt, \byt) \rb \leq 2 \lrx G_x S^2 \sqrt{G_x^2 + \frac{\sigma^2}{n}} + \frac{4 D}{\lry} + 20 \lry \sigma^2 S + 16 S \lry^2 \Lf (\sync-1)^2 \lp \sigma^2 + \heteroy^2 \rp.
\end{align*}
Summing over $k = 0$ to $T/S - 1$, we get
\begin{align*}
    & \frac{1}{T} \sum_{k=0}^{T/S-1} \sum_{t=kS+1}^{(k+1)S} \mbe \lb \Phi(\bxt) - f(\bxt, \byt) \rb \leq 2 \lrx G_x S \sqrt{G_x^2 + \frac{\sigma^2}{n}} + \frac{4 D}{\lry S} + 20 \lry \sigma^2 + 16 \lry^2 \Lf (\sync-1)^2 \lp \sigma^2 + \heteroy^2 \rp.
\end{align*}
\end{proof}

\subsection{With full synchronization}
\label{app:NC_1PC_tau_1}
In this subsection, we discuss the case when the clients perform a single local update between successive communications $\sync = 1$.
The goal of the results in this subsection is to show that at least in this specialized case, linear speedup can be achieved for NC-1PC functions.

\begin{lemma}
\label{lem:NC_1PC_Phi_f_diff_tau_1}
Suppose the local functions satisfy Assumptions \ref{assum:smoothness}, \ref{assum:bdd_var}, \ref{assum:bdd_hetero}, \ref{assum:Lips_cont_x}, \ref{assum:1pc_y}.
Further, suppose we choose the step-size $\lry$ such that $\lry \leq \frac{1}{2 \Lf}$.
Then the iterates generated by \cref{alg_local_SGDA_plus} satisfy
\begin{align}
    \avgtT \mbe \lb \Phi(\bxt) - f(\bxt, \byt) \rb & \leq 2 \lrx G_x S \sqrt{G_x^2 + \frac{\sigma^2}{n}} + \frac{D}{2 \lry S} + \frac{\lry \sigma^2}{n}. \nn
\end{align}
\end{lemma}

\begin{proof}
The proof follows similar technique as in \cref{lem:NC_C_Phi_f_diff}. From \eqref{eq_proof:lem:NC_C_Phi_f_diff_2}, we get
\begin{align}
    & \sum_{t=kS+1}^{(k+1)S} \mbe \lb \Phi(\bxt) - f(\bxt, \byt) \rb \leq 2 \lrx G_x S^2 \sqrt{G_x^2 + \frac{\sigma^2}{n}} + \sum_{t=kS+1}^{(k+1)S} \mbe \lb f(\Tbxk, \by^*(\Tbxk)) - f(\Tbxk, \byt) \rb. \label{eq_proof:lem:NC_1PC_Phi_f_diff_tau_1}
\end{align}
We only need to bound the second term in \eqref{eq_proof:lem:NC_1PC_Phi_f_diff_tau_1}.
With $\sync = 1$, the $\byit$ updates reduce to minibatch stochastic gradient ascent, with batch-size $\mco (n)$. Using the result for stochastic minimization of $\gamma$-quasar convex functions (for one-point-concave functions, $\gamma = 1$) using SGD (Theorem~3.3 in \cite{jin20quasar_convex_arxiv}), we get
\begin{align*}
    \frac{1}{S} \sum_{t=kS+1}^{(k+1)S} \mbe \lb f(\Tbxk, \by^*(\Tbxk)) - f(\Tbxk, \byt) \rb \leq \frac{D}{2 \lry S} + \frac{\lry \sigma^2}{n},
\end{align*}
which completes the proof.
\end{proof}

Next, we state the convergence result.

\begin{theorem*}
Suppose the local loss functions $\{ f_i \}$ satisfy Assumptions \ref{assum:smoothness}, \ref{assum:bdd_var}, \ref{assum:bdd_hetero}, \ref{assum:Lips_cont_x}, \ref{assum:1pc_y}.
Further, let $\norm{\byt}^2 \leq D$ for all $t$.
Suppose the step-size $\lry$ is chosen such that $\lry \leq \frac{1}{2 \Lf}$.
Then the output $\bbxT$ of \cref{alg_local_SGDA_plus} satisfies
\begin{equation}
    \begin{aligned}
        & \avgtT \mbe \norm{\G \Phi_{1/2\Lf} (\bxt)}^2 \leq \mco \lp \frac{\widetilde{\Delta}_{\Phi}}{\lrx T} + \lrx \Lf \lp G_x^2 + \frac{\sigma^2}{n} \rp + \frac{\lry \Lf \sigma^2}{n} + \Lf \lb \lrx G_x S \sqrt{G_x^2 + \frac{\sigma^2}{n}} + \frac{D}{\lry S} \rb \rp,
    \end{aligned}
    \label{eq:thm:NC_1PC_tau_1}
\end{equation}
where {\small$\widetilde{\Delta}_{\Phi} \triangleq \Phi_{1/2 \Lf} (\bx_0) - \min_\bx \Phi_{1/2 \Lf} (\bx)$}.
With the following parameter values:
\begin{align*}
    S = \Theta \lp \sqrt{\frac{T}{n}} \rp, \qquad \lrx = \Theta \lp \frac{n^{1/4}}{T^{3/4}} \rp, \qquad \lry = \Theta \lp \frac{n^{3/4}}{T^{1/4}} \rp,
\end{align*}
we get
\begin{align}
    & \avgtT \mbe \norm{\G \Phi_{1/2\Lf} (\bxt)}^2 \leq \mco \lp \frac{1}{(n T)^{1/4}} \rp + \mco \lp \frac{n^{1/4}}{T^{3/4}} \rp. \nn
\end{align}
\end{theorem*}

\begin{cor}
\label{cor:NC_1PC_comm_cost_tau_1}
To reach an $\epsilon$-accurate point, i.e., $\bx$ such that $\mbe \| \G \Phi_{1/2\Lf} (\bx) \| \leq \epsilon$,
the stochastic gradient complexity of \cref{alg_local_SGDA_plus} is $\mco (1/n \epsilon^8)$.
\end{cor}

\begin{proof}
We sum the result in \cref{lem:NC_C_Phi_smooth_decay_one_iter} over $t = 0$ to $T-1$. Since $\sync = 1$, $\CExyt = 0$ for all $t$. Rearranging the terms, we get
\begin{align}
    & \avgtT \mbe \norm{\G \Phi_{1/2\Lf} (\bxt)}^2 \leq \frac{8}{\lrx} \avgtT \lp \mbe \lb \Phi_{1/2\Lf} (\bxt) \rb - \mbe \lb \Phi_{1/2\Lf} (\bxtp) \rb \rp + 8 \lrx \Lf \lp G_x^2 + \frac{\sigma^2}{n} \rp \nn \\
    & \qquad \qquad \qquad \qquad \qquad \qquad + 16 \Lf \avgtT \mbe \lb \Phi(\bxt) - f(\bxt, \byt) \rb \nn \\
    & \leq \frac{8}{\lrx T} \lb \Phi_{1/2\Lf} (\bx_0) - \mbe \lb \Phi_{1/2\Lf} (\bx_T) \rb \rb + 8 \lrx \Lf \lp G_x^2 + \frac{\sigma^2}{n} \rp \nn \\
    & \quad + 16 \Lf \lb 2 \lrx G_x S \sqrt{G_x^2 + \frac{\sigma^2}{n}} + \frac{D}{2 \lry S} + \frac{\lry \sigma^2}{n} \rb \tag{\cref{lem:NC_1PC_Phi_f_diff_tau_1}} \\
    & \leq \frac{8 \widetilde{\Delta}_{\Phi}}{\lrx T} + 8 \lrx \Lf \lp G_x^2 + \frac{\sigma^2}{n} \rp + \frac{16 \lry \Lf \sigma^2}{n} + 16 \Lf \lb 2 \lrx G_x S \sqrt{G_x^2 + \frac{\sigma^2}{n}} + \frac{D}{2 \lry S} \rb, \nn
\end{align}
where $\widetilde{\Delta}_{\Phi} = \Phi_{1/2\Lf} (\bx_0) - \min_\bx \Phi_{1/2\Lf} (\bx)$.
Following similar technique as in the proof of \cref{thm:NC_C}, using the following parameter values,
\begin{align*}
    S = \Theta \lp \sqrt{\frac{T}{n}} \rp, \qquad \lrx = \Theta \lp \frac{n^{1/4}}{T^{3/4}} \rp, \qquad \lry = \Theta \lp \frac{n^{3/4}}{T^{1/4}} \rp,
\end{align*}
we get the following bound.
\begin{align}
    & \avgtT \mbe \norm{\G \Phi_{1/2\Lf} (\bxt)}^2 \leq \mco \lp \frac{\sigma^2 + D + G_x^2}{(n T)^{1/4}} \rp + \mco \lp \frac{n^{1/4}}{T^{3/4}} \rp. \nn
\end{align}

\end{proof}

\begin{proof}[Proof of \cref{cor:NC_1PC_comm_cost_tau_1}]
We assume $T \geq n$.
To reach an $\epsilon$-accurate point, i.e., $\bx$ such that $\mbe \lnr \G \Phi_{1/2\Lf} (\bx) \rnr \leq \epsilon$, since
\begin{align*}
    \mbe \lnr \G \Phi_{1/2\Lf} (\bbxT) \rnr \leq \lb \frac{1}{T} \sumtT \mbe \lnr \G \Phi_{1/2\Lf} (\bxt) \rnr^2 \rb^{1/2} \leq \mco \lp \frac{1}{(n T)^{1/8}} \rp + \mco \lp \frac{n^{1/8}}{T^{3/8}} \rp,
\end{align*}
we need $T = \mco \lp \frac{1}{n \epsilon^8} \rp$ iterations.
\end{proof}

\newpage
\section{Additional Experiments}
\label{app:add_exp}

\begin{algorithm}[ht]
\caption{Local SGDA+ \cite{mahdavi21localSGDA_aistats}}
\label{alg_mom_local_SGDA_plus}
\begin{algorithmic}[1]
	\STATE{\textbf{Input:} {\small$\bx_0^i = \Tbx_0 = \bx_0, \by_0^i = \by_0$, $\mbf d_{x,0}^i = \Gx f_i (\bx^i_0, \by^i_0; \xi^i_0)$, $\mbf d_{y,0}^i = \Gy f_i (\bx^i_0, \by^i_0; \xi^i_0)$} for all $i \in [n]$; step-sizes $\lrx, \lry$; synchronization intervals $\sync, S$; $T, k = 0$}
	\FOR[At all clients $i=1,\hdots, n$]{$t=0$ to $T-1$}
	    \STATE{{\small$\Tbxitp = \bxit - \lrx \bdxit$, 
        $\ \bxitp = \bxit + \cvxt ( \Tbxitp - \bxit )$}}
        \STATE{{\small$\Tbyitp = \byit + \lry \bdyit$, $\ \byitp = \byit + \cvxt ( \Tbyitp - \byit )$}}
        \STATE{Sample minibatch $\xiitp$ from local data}
        \STATE{{\small$\bdxitp = (1 - \momx \cvxt) \bdxit + \momx \cvxt \Gx f_i (\bxitp, \byitp; \xiitp)$}}
        \STATE{{\small$\bdyitp = (1 - \momy \cvxt) \bdyit + \momy \cvxt \Gy f_i (\Tbxk, \byitp; \xiitp)$}}
        \IF{$t+1$ mod $\sync = 0$}
            \STATE{Clients send $\{ \bxitp, \byitp \}$ to the server}
            \STATE{Server computes averages $\bxtp \triangleq \frac{1}{n} \sumin \bxitp$, 
            $\bytp \triangleq \frac{1}{n} \sumin \byitp$, and sends to all the clients}
            \STATE{$\bxitp = \bxtp$, $\byitp = \bytp$, for all $i \in [n]$}
            \STATE{$\bdxitp = 0$, $\bdyitp = 0$, for all $i \in [n]$}
        \ENDIF
        \IF{$t+1$ mod $S = 0$}
            \STATE{Clients send $\{ \bxitp \}$ to the server}
            \STATE{$k \gets k+1$}
            \STATE{Server computes averages $\Tbxk \triangleq \frac{1}{n} \sumin \bxitp$, and sends to all the clients}
        \ENDIF
	\ENDFOR
	\STATE{\textbf{Return: }$\bbxT$ drawn uniformly at random from $\{ \bxt \}$, where $\bxt \triangleq \frac{1}{n} \sumin \bxit$}
\end{algorithmic}
\end{algorithm}

\subsection{Fair Classification}
Batch-size of $32$ is used. Momentum parameter $0.9$ is used only in Momentum Local SGDA (\cref{alg_NC_momentum}) and corresponds to $\cvx \beta$ in the pseudocode.

\begin{table}[ht]
\begin{center}
\caption{Parameter values for experiments in \cref{sec:exp_fair}}
\begin{tabular}{llll}
\hline
Parameter &  &  &  \\
\hline
Learning Rate $(\lry)$ & $0.02$ & $2 \times 10^{-3}$ & $2 \times 10^{-4}$ \\
Learning Rate $(\lrx)$ & $0.016$ & $1.6 \times 10^{-3}$ & $1.6 \times 10^{-4}$ \\
Communication rounds & 150 & 75 & 75 \\
\hline
\end{tabular}
\end{center}
\end{table}

\subsection{Robust Neural Network Training}
Batch-size of $32$ is used. Momentum parameter $0.9$ is used only in Momentum Local SGDA+ (\cref{alg_mom_local_SGDA_plus}) and corresponds to $\cvx \beta$ in the pseudocode. $S = \sync^2$ in both \cref{alg_local_SGDA_plus} and \cref{alg_mom_local_SGDA_plus}.

\begin{table}[ht]
\begin{center}
\caption{Parameter values for experiments in \cref{sec:exp_fair}}
\begin{tabular}{llll}
\hline
Parameter &  &  &  \\
\hline
Learning Rate $(\lry)$ & $0.02$ & $2 \times 10^{-3}$ & $2 \times 10^{-4}$ \\
Learning Rate $(\lrx)$ & $0.016$ & $1.6 \times 10^{-3}$ & $1.6 \times 10^{-4}$ \\
Communication rounds & 150 & 75 & 75 \\
\hline
\end{tabular}
\end{center}
\end{table}

\begin{figure}[ht]
     \centering
     \includegraphics[width=0.55\textwidth]{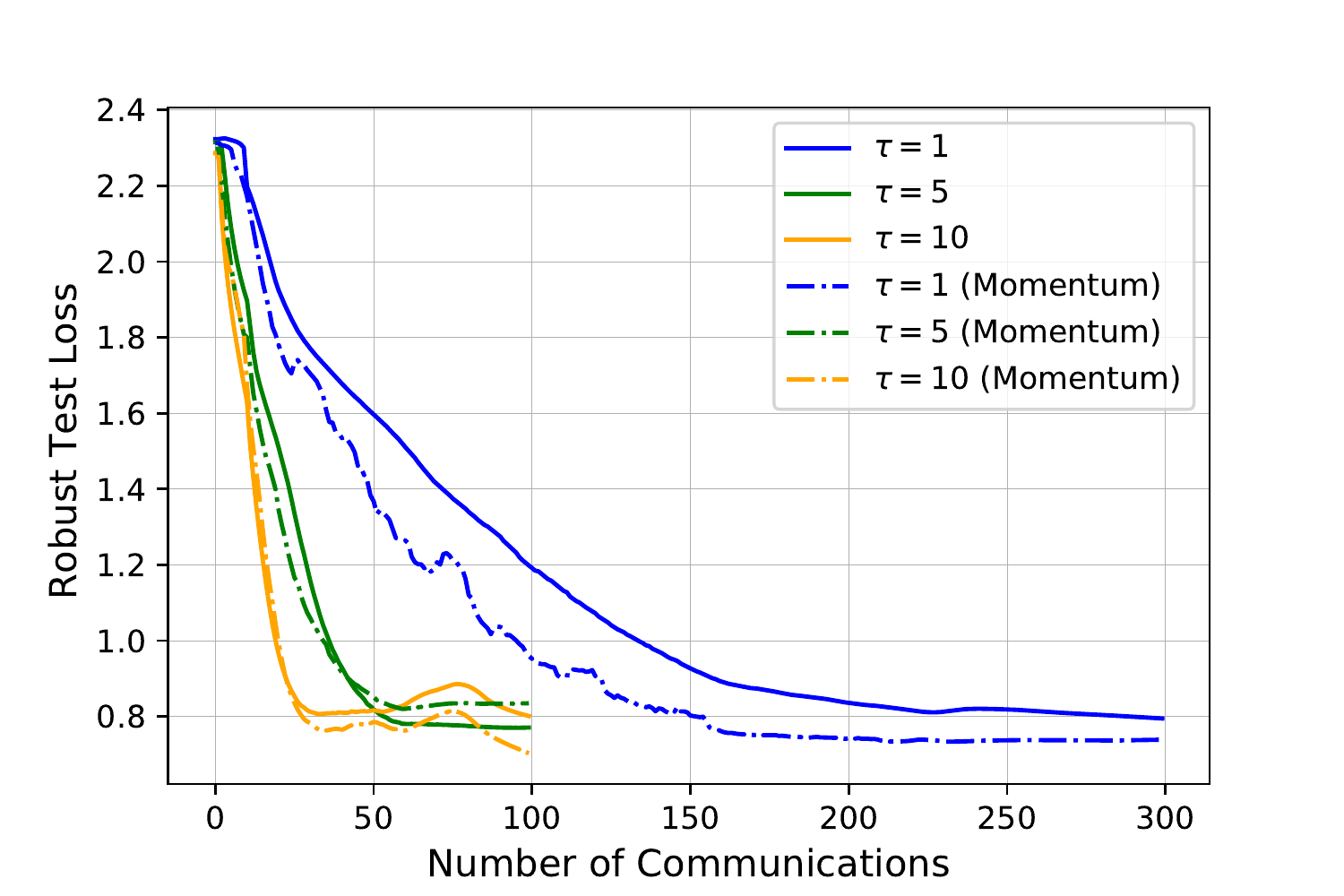}
     \caption{Robust test loss for the CIFAR10 experiment shown in \cref{sec:exp_robustnn}. The test loss in \cref{eq:exp_robustnn} is computed using some steps of gradient ascent to find an estimate of $\by^*$.}
\end{figure}

\begin{figure}[ht]
     \centering
     \begin{subfigure}
         \centering
         \includegraphics[width=0.45\textwidth]{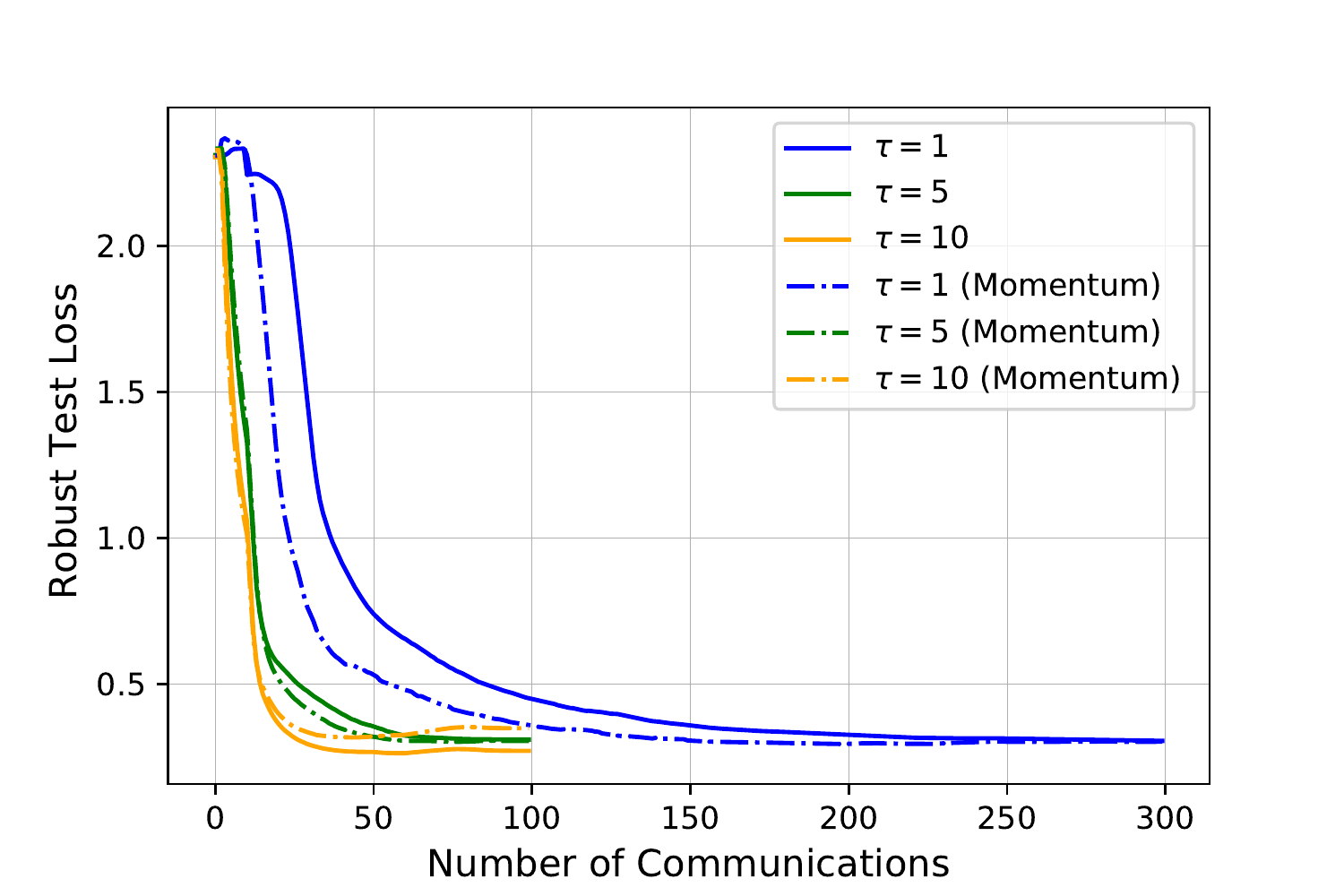}
         \label{fig:robustnn_fashionmnist_loss}
     \end{subfigure}
     \begin{subfigure}
         \centering
         \includegraphics[width=0.45\textwidth]{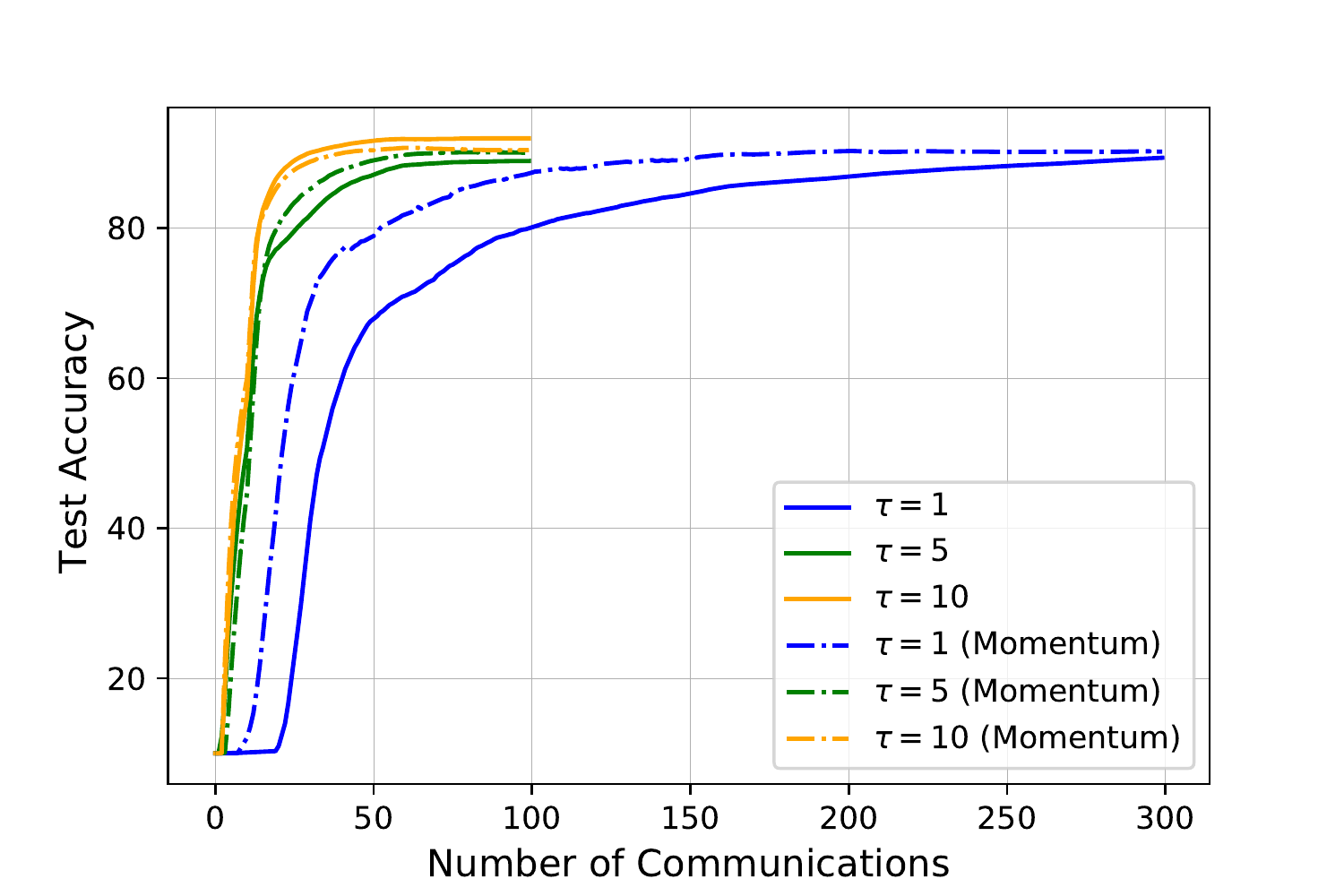}
         \label{fig:robustnn_fashionmnist_acc}
     \end{subfigure}
     \caption{Comparison of the effects of $\sync$ on the performance of Local SGDA and Momentum Local SGDA algorithms, for the robust NN training problem on the FashionMNIST dataset, with the VGG11 model. The figures show the robust test loss and robust test accuracy. \label{fig:robustnn_fashionmnist}}
\end{figure}

\ifx
\newpage
\section{Extra: Nonconvex-Concave (NC-C) case} 
We organize this section as follows. First, in  we present some intermediate results, which we use to prove the main theorem. Next, in , we present the proof of \cref{thm:NC_C}, which is followed by the proofs of the intermediate results in 

\subsection{Intermediate Lemmas}

\subsection{Proof of \cref{thm:NC_C}}

\subsection{Proofs of the Intermediate Lemmas}

\subsection{Algorithm}
The algorithm is modified from the single-client approach proposed in \cite{rafique18WCC_oms}, and incorporates local updates at the clients and periodic communication with the server.

\begin{equation}
    \min_{\bx} \max_{\by} \lcb \Fgp (\bx, \by; \bar{\bx}, \bar{\by}) \triangleq f(\bx, \by) + \frac{1}{2\gamma} \lnr \bx - \bar \bx \rnr^2 - \frac{1}{2 \psi} \lnr \by - \bar \by \rnr^2 \rcb
\end{equation}
where $\gamma < \frac{1}{L}$.
For every $t$, the problem 
\begin{align*}
    \min_{\bx} \max_{\by} \Fgpt (\bx, \by; \bar{\bx}_t, \bar{\by}_t)
\end{align*}
is a strongly convex-strongly concave minimax problem, which we solve using the Local SGDA algorithm proposed in \cite{mahdavi21localSGDA_aistats}. 
The function $\Fgpt (\bx, \by; \bar{\bx}_t, \bar{\by}_t)$ is $\mu_x$-strongly convex in $\bx$, and $\mu_y$-strongly concave in $\by$, where
$\mu_x = \frac{1}{\gamma}-L$ and $\mu_y = \frac{1}{\pt}$.
Also, $\Fgpt$ is $L_x$-smooth in $\bx$, and $L_y$-smooth in $\by$, where $L_x = \Lf + \frac{1}{\gamma}$, and $L_y = \Lf + \frac{1}{\pt}$.
{\color{blue}We define $\mu = \min \lcb \mu_x, \mu_y \rcb$ and $\Lft = \max \lcb L_x, L_y \rcb$.}

\subsection{Analysis}
\paragraph*{Update Equations}
\begin{equation}
    \begin{aligned}
        \bxitkp &= \mc P_{\mc X} \lb \bxitk - \eta_k \lp \Tilde{\G}_x f_i (\bxitk, \byitk) + \frac{1}{\gamma} \lp \bxitk - \bar{\bx}_t \rp \rp \rb \\
        \byitkp &= \mc P_{\mc Y} \lb \byitk + \eta_k \lp \Tilde{\G}_y f_i (\bxitk, \byitk) - \frac{1}{\pt} \lp \byitk - \bar{\by}_t \rp \rp \rb
    \end{aligned}
\end{equation}

\paragraph*{Convergence Result from \cite{mahdavi21localSGDA_aistats} for the Strongly-Convex-Strongly-Concave Case}

\begin{theorem*}
If each local function $\{ f_i \}$ satisfies Assumptions \ref{assum:smoothness}, \ref{assum:concavity}, \ref{assum:bdd_var}. 
Choose $\lrx = \lry = \lrk = \frac{8}{\mu (k + a)}$, then the iterates generated by \cref{alg_local_SGDA} satisfy
\begin{align}
    \mbe \lb \norm{\bx_K - \bxs}^2 + \norm{\by_K - \bys}^2 \rb \leq \mco \lp \frac{a^3}{K^3} + \frac{\kappa^2 \sync^2 (\heterox + \heteroy) }{\mu K^2} + \frac{\kappa^2 \sync^2 \sigma^2}{\mu K^2} + \frac{\sigma^2}{\mu^2 n K} \rp.
\end{align}
With synchronization frequency $\sync = \sqrt{K/n}$, we get
\begin{align}
    \mbe \lb \norm{\bx_K - \bxs}^2 + \norm{\by_K - \bys}^2 \rb \leq \mco \lp \frac{a^3}{K^3} + \frac{\kappa^2 (\heterox + \heteroy) }{\mu n K} + \frac{\kappa^2 \sigma^2}{\mu n K} + \frac{\sigma^2}{\mu^2 n K} \rp.
\end{align}
\end{theorem*}

\paragraph*{Extension of Lemma 3.1 in \cite{mahdavi21localSGDA_aistats}}
$(\bxts, \byts)$ is the saddle-point of
\begin{align*}
    \min_{\bx \in \mc X} \max_{\by \in \mc Y} \lcb \Fgpt (\bx, \by; \bar{\bx}_t, \bar{\by}_t) \triangleq f(\bx, \by) + \frac{1}{2 \gamma} \norm{\bx - \bar{\bx}_t}^2 - \frac{1}{2 \pt} \norm{\by - \bar{\by}_t}^2 \rcb.
\end{align*}

\begin{theorem}{Convergence Rate}
\label{thm:nc_c_conv_rate}
\begin{align}
    \mbf E_{t, K_t} &= \mbe \lb \norm{\mbf x_{t,K_t} - \bxts}^2 + \norm{\mbf y_{t,K_t} - \bxts}^2 \rb \nn \\
    &\leq \mco \lp \frac{a_t^3}{K_t^3} \rp + \mco \lp \frac{\kappa_t^2 \sync (\sync - 1)}{\mu K_t^2} \lp \sigma^2 + \delta_x + \delta_y \rp \rp + \mco \lp \frac{\sigma^2}{\mu^2 n K_t} \rp.
\end{align}
\end{theorem}

\newpage
\paragraph*{Global Algorithm:}
\begin{align*}
    \Fgpt (\bx, \by; \bar{\bx}_t, \bar{\by}_t) & \triangleq f(\bx, \by) + \frac{1}{2 \gamma} \norm{\bx - \bar{\bx}_t}^2 - \frac{1}{2 \pt} \norm{\by - \bar{\by}_t}^2, \\
    \Fg (\bx, \by; \bar{\bx}_t) & \triangleq f(\bx, \by) + \frac{1}{2 \gamma} \norm{\bx - \bar{\bx}_t}^2.
\end{align*}
Next,
\begin{align}
    & \max_{\by} \Fg (\bx_{t,K_t}, \by; \bar{\bx}_t)  - \min_{\bx} \max_{\by} \Fg (\bx, \by; \bar{\bx}_t) \nn \\
    & \leq \max_{\by} \Fgpt (\bx_{t,K_t}, \by; \bar{\bx}_t, \bar{\by}_t)  - \min_{\bx} \max_{\by} \Fgpt (\bx, \by; \bar{\bx}_t, \bar{\by}_t) + \frac{D_y^2}{2 \pt} \nn \\
\end{align}
using $\lp L + \frac{1}{\gamma} - \frac{1}{\pt} \rp$-smoothness of $\Fgpt$.
Define $\Phi (\bx) \triangleq \max_{\by} f(\bx, \by)$. 
Then, $\max_{\by} \Fg (\bx, \by; \bar{\bx}) = \Phi(\bx) + \frac{1}{2 \gamma} \norm{\bx - \bar{\bx}}^2$.
We denote $\bbxtp = \bx_{t,K_t}$, $\bxts = \argmin_{\bx} \max_{\by} \Fg (\bx, \by; \bar{\bx})$.
From ()
, we see that for some constant $c_1 > 0$ (following proof of Theorem 4.2 in \cite{rafique18WCC_oms})
\begin{align}
    & \mbe \lb \Phi(\bbxtp) + \frac{1}{2 \gamma} \norm{\bbxtp - \bbxt}^2 - \Phi(\bxts) - \frac{1}{2 \gamma} \norm{\bxts - \bbxt}^2 \rb \leq c_1 \frac{\sigma^2 \pt^2}{n K_t} + \frac{D_y^2}{2 \pt} \\
    \Rightarrow & \mbe \Phi(\bbxtp) \leq \mbe \lb \Phi(\bxts) + \frac{1}{2 \gamma} \norm{\bxts - \bbxt}^2 - \frac{1}{2 \gamma} \norm{\bbxtp - \bbxt}^2 \rb + c_1 \frac{\sigma^2 \pt^2}{n K_t} + \frac{D_y^2}{2 \pt} \\
    \Rightarrow & \mbe \Phi(\bbxtp) \leq \mbe \Phi(\bxts) + \frac{1}{2 \gamma} \mbe \lp \frac{1}{3} \norm{\bxts - \bbxt}^2 + 4 \norm{\bbxtp - \bxts}^2 \rp + c_1 \frac{\sigma^2 \pt^2}{n K_t} + \frac{D_y^2}{2 \pt} \\
    \Rightarrow & \mbe \Phi(\bbxtp) \leq \mbe \Phi(\bxts) + \frac{1}{6 \gamma} \mbe \norm{\bxts - \bbxt}^2 + c \lp c_1 \frac{\sigma^2 \pt^2}{n K_t} + \frac{D_y^2}{2 \pt} \rp \\
    \Rightarrow & \mbe \Phi(\bbxtp) \leq \mbe \Phi(\bbxt) - \frac{1}{3 \gamma} \mbe \norm{\bxts - \bbxt}^2 + c \lp c_1 \frac{\sigma^2 \pt^2}{n K_t} + \frac{D_y^2}{2 \pt} \rp \\
    \Rightarrow & \mbe \Phi(\bbxtp) \leq \mbe \Phi(\bar{\bx}_0) - \frac{1}{3 \gamma} \sumtT \mbe \norm{\bxts - \bbxt}^2 + c \sumtT \lp c_1 \frac{\sigma^2 \pt^2}{n K_t} + \frac{D_y^2}{2 \pt} \rp.
\end{align}
where
\begin{align}
    \norm{\bxts - \bbxt}^2 - \norm{\bbxtp - \bbxt}^2 &= \lp \norm{\bxts - \bbxt} + \norm{\bbxtp - \bbxt} \rp \lp \norm{\bxts - \bbxt} - \norm{\bbxtp - \bbxt} \rp \nn \\
    & \leq \lp 2 \norm{\bxts - \bbxt} + \norm{\bbxtp - \bxts} \rp \norm{\bxts - \bbxtp} \nn \\
    & = 2 \norm{\bxts - \bbxt} \norm{\bxts - \bbxtp} + \norm{\bbxtp - \bxts}^2 \nn \\
    & = \frac{1}{3} \norm{\bxts - \bbxt}^2 + 4 \norm{\bbxtp - \bxts}^2.
\end{align}
Rearranging the terms and summing over $t=0,\hdots, T-1$, we get
\begin{align}
    \mbe \norm{\bxtaus - \bar{\bx}_{\sync}}^2 \leq \frac{1}{T} \sumtT \mbe \norm{\bxts - \bbxt}^2 \leq \frac{3 \gamma}{T} \lp \Phi(\bar{\bx}^0) - \Phi^* + c \sumtT \lp c_1 \frac{\sigma^2 \pt^2}{n K_t} + \frac{D_y^2}{2 \pt} \rp \rp.
\end{align}
Since $\mbe \lb \text{Dist} \lp \mbf 0, \partial \Phi (\bxtaus) \rp^2 \rb \leq \frac{1}{\gamma^2} \mbe \norm{\bxtaus - \bbxt}^2$, we can ensure $\mbe \lb \text{Dist} \lp \mbf 0, \partial \Phi (\bxtaus) \rp^2 \rb \leq \epsilon^2$ by choosing
\begin{align*}
    T = \frac{3}{\gamma \epsilon^2} \lp \Phi(\bar{\bx}^0) - \Phi^* + c \sumtT \lp c_1 \frac{\sigma^2 \pt^2}{n K_t} + \frac{D_y^2}{2 \pt} \rp \rp
\end{align*}
Using $\pt = \mco \lp \frac{1}{\epsilon^2} \rp, K_t = \mco \lp \frac{1}{n \epsilon^6} \rp$, we get
\begin{align*}
    \frac{1}{T} \sumtT \lp c_1 \frac{\sigma^2 \pt^2}{n K_t} + \frac{D_y^2}{2 \pt} \rp \leq \mco \lp \epsilon^2 \rp.
\end{align*}
Total IFO complexity
\begin{align}
    \sumtT K_t &= \mco \lp \frac{1}{\gamma \epsilon^2} \frac{1}{n \epsilon^6} \rp = \mco \lp \frac{1}{\gamma n \epsilon^8} \rp.
\end{align}
\fi

\end{document}